\documentclass[UTF8,10pt]{article}

\usepackage[left=1.91cm,right=1.91cm,top=2.54cm,bottom=2.54cm]{geometry}
\usepackage{amsfonts}
\usepackage{amsmath}
\usepackage{amssymb}
\usepackage{pgfplots}
\usepackage{titlesec}
\usepackage{bm}
\usepackage{appendix}
\usepackage{mathtools}
\usepackage{amsthm}
\usepackage{enumitem}
\usepackage[colorlinks,linkcolor=black,anchorcolor=black,citecolor=black]{hyperref}
\usepackage[numbers]{natbib}
\usepackage{setspace}
\usepackage{mathrsfs}
\usepackage{cases}
\usepackage{fancyhdr}

\usepackage{txfonts}

\theoremstyle{definition}
\newtheorem{thm}{Theorem}[section]

\newtheorem{lem}[thm]{Lemma}
\newtheorem{prop}[thm]{Proposition}

\newtheorem*{rmk}{Remark}

\newcommand{\R}{\mathbb{R}}
\newcommand{\Z}{\mathbb{Z}}

\newcommand{\N}{\mathbb{N}}

\newcommand{\T}{\mathbb{T}}

\newcommand{\TP}{\overline{\partial}{}}

\newcommand{\dive}{\text{div }}
\newcommand{\bp}{(b_0\cdot\partial)}
\newcommand{\bzp}{(\bz\cdot\nab)}

\newcommand{\q}{\quad}
\newcommand{\p}{\partial}
\newcommand{\dd}{\mathfrak{D}}
\newcommand{\DD}{\mathcal{D}}

\newcommand{\nab}{\nabla}
\newcommand{\pa}{\nabla_A}
\newcommand{\pA}{\nabla_{\hat{A}}}
\newcommand{\diva}{\text{div}_{A}}

\newcommand{\sh}{\sharp}
\newcommand{\eps}{\varepsilon}

\newcommand{\cnab}{\overline{\nab}}
\newcommand{\dx}{\,\mathrm{d}x}
\newcommand{\dy}{\,\mathrm{d}y}

\newcommand{\dt}{\,\mathrm{d}t}
\newcommand{\dS}{\,\mathrm{d}S}
\newcommand{\dyy}{\,\mathrm{d}y'}
\newcommand{\ddt}{\frac{\mathrm{d}}{\mathrm{d}t}}

\newcommand{\lee}{\langle}
\newcommand{\ree}{\rangle}

\newcommand{\EE}{\mathcal{E}}

\newcommand{\ff}{\mathcal{F}}

\newcommand{\VV}{\mathbf{V}}
\newcommand{\NN}{\mathcal{N}}
\newcommand{\MM}{\mathcal{M}}
\newcommand{\RR}{\mathcal{R}}
\newcommand{\QQ}{\mathbf{Q}}
\newcommand{\GG}{\mathbf{G}}
\newcommand{\FF}{\mathbf{F}}

\newcommand{\PP}{\mathcal{P}}

\newcommand{\ww}{\textbf{w}_0}
\newcommand{\vv}{\textbf{v}_0}
\newcommand{\pp}{\textbf{P}_0}
\newcommand{\qq}{\textbf{Q}_0}
\newcommand{\bz}{\textbf{b}_0}

\newcommand{\idt}{\int_{\mathcal{D}_t}}
\newcommand{\ipdt}{\int_{\partial\mathcal{D}_t}}
\newcommand{\io}{\int_{\Omega}}

\newcommand{\ig}{\int_{\Gamma}}
\numberwithin{equation}{section}
\usepackage{xcolor}

\begin{document}
\bibliographystyle{plain}
\title{\textbf{Anisotropic Regularity of the Free-Boundary Problem in \\ Compressible Ideal Magnetohydrodynamics}}
\author{
{\sc Hans Lindblad}\thanks{Department of Mathematics, Johns Hopkins University, 3400 N Charles St, Baltimore, MD 21218, USA. Email: \texttt{lindblad@math.jhu.edu}}
\and
{\sc Junyan Zhang}\thanks{Department of Mathematics, National University of Singapore, 10 Lower Kent Ridge Road, Singapore 119076, Singapore. Email: \texttt{zhangjy@nus.edu.sg}}
 }
\date{\today}
\maketitle

\setcounter{tocdepth}{1}

\begin{abstract}
We consider 3D free-boundary compressible ideal magnetohydrodynamic (MHD) system under the Rayleigh-Taylor sign condition. It describes the motion of a free-surface perfect conducting fluid in an electro-magnetic field. A local existence and uniqueness result was recently proved by Trakhinin and Wang \cite{TW2020MHDLWP} by using Nash-Moser iteration. However, that result loses regularity going from data to solution. In this paper, we show that the Nash-Moser iteration scheme in \cite{TW2020MHDLWP} can be improved such that the local-in-time \textit{smooth} solution exists and is unique when the initial data is \textit{smooth} and satisfies the compatibility condition up to \textit{infinite} order. Second, we prove the a priori estimates without loss of regularity for the free-boundary compressible MHD system in Lagrangian coordinates in anisotropic Sobolev space, with more regularity tangential to the boundary than in the normal direction. It is based on modified Alinhac good unknowns, which take into account the covariance under the change of coordinates to avoid the derivative loss; full utilization of the cancellation structures of MHD system, to turn normal derivatives into tangential ones; and delicate analysis in anisotropic Sobolev spaces. As a result, we can prove the uniqueness and the continuous dependence on initial data provided the local existence, and a continuation criterion for smooth solution. Finally, we extend the local well-posedness theorem to the case of initial data only satisfying compatibility conditions up to finite order, assuming these can be approximated by data satisfying infinitely many compatibility conditions. 
\end{abstract}

\tableofcontents

\section{Introduction}
In this paper, we consider the 3D compressible ideal magnetohydrodynamics (MHD) equations
\begin{equation}
\begin{cases}
\rho D_t u=B\cdot\nabla B-\nabla P,~~P:=p+\frac{1}{2}|B|^2~~~& \text{in}~\DD; \\
D_t\rho+\rho(\nabla\cdot u)=0~~~&\text{in}~\DD; \\
D_t B=B\cdot\nabla u-B(\nabla\cdot u),~~~&\text{in}~\DD; \\
\nabla\cdot B=0~~~&\text{in}~\DD,
\end{cases}\label{CMHD}
\end{equation}
describing the motion of a compressible conducting fluid in an electro-magnetic field. Here $\DD:={\bigcup}_{0\leq t\leq T}\{t\}\times \DD_t$ and $\DD_t\subset \R^3$ is the domain occupied by the conducting fluid whose boundary $\p\DD_t$ moves with the velocity of the fluid. $\nabla:=(\p_{x_1},\p_{x_2},\p_{x_3})$ is the standard spatial derivative and $\dive X:=\p_{x_i} X^i$ is the standard divergence for any vector field $X$. $D_t:=\p_t+u\cdot\nabla$ is the material derivative. Throughout this paper, $X^i=\delta^{li}X_l$ for any vector field $X$, i.e., we use Einstein summation convention. The fluid velocity $u=(u_1,u_2,u_3)$, the magnetic field $B=(B_1,B_2,B_3)$, the fluid density $\rho$, the fluid pressure $p$ and the domain $\DD\subseteq[0,T]\times\R^3$ are to be determined. Here we consider the isentropic case, and thus the fluid pressure $p=p(\rho)$ should be a given strictly increasing smooth function of the density $\rho$.

\subsection{Initial and boundary conditions and constraints}

The boundary conditions of \eqref{CMHD} are
\begin{equation}\label{MHDB}
\begin{cases}
D_t|_{\partial\DD}\in \mathcal{T}(\partial\DD) \\
P=0&\text{on}~\partial\DD, \\
B\cdot n=0&\text{on}~\partial\DD,
\end{cases}
\end{equation} where $\mathcal{T}(\p\DD)$ denotes the tangent bundle of $\p\DD$ and $n$ denotes the unit exterior normal vector to $\p\DD_t$. The first condition in \eqref{MHDB} means that the boundary moves with the velocity of the fluid. It can be equivalently rewritten as ``$V(\p\DD_t)=u\cdot n$ on $\p\DD$" or ``$(1,u)$ is tangent to $\p\DD$". The second condition in \eqref{MHDB} means that outside the fluid region $\DD_t$ is the vacuum. The third boundary condition $B\cdot n=0$ shows that the fluid is a perfect conductor.

\begin{rmk} The conditions $\nabla\cdot B=0$ in $\DD$ and $B\cdot n=0$ on $\p\DD$ are both constraints only for initial data so that the system is not over-determined. They can propagate to any time $t>0$ if initially hold. See Hao-Luo \cite{HaoLuo2014priori} for details.
\end{rmk}

We consider the Cauchy problem of \eqref{CMHD}: Given a bounded domain $\DD_0\subset \R^3$ and the initial data $u_0$, $\rho_0$ and $B_0$ satisfying the constraints $\nabla\cdot B_0=0$ in $\DD_0$ and $(B_0\cdot n)|_{\{0\}\times\p\DD_0}=0$, we want to find a set $\DD$, the vector field $u$, the magnetic field $B$, and the density $\rho$ solving \eqref{CMHD} satisfying the boundary conditions \eqref{MHDB} and the initial data
\begin{equation}\label{MHDI}
\DD_0=\{x: (0,x)\in \DD\},\q (u,B,\rho)=(u_0, B_0,\rho_0),\q \text{in}\,\,\{0\}\times \DD_0,
\end{equation}

\paragraph*{Energy conservation law}

The free-boundary compressible MHD system \eqref{CMHD} together with the boundary conditions \eqref{MHDB} satisfies the following energy conservation law. Let $Q(\rho)=\int_1^{\rho} p(R)/R^2 dR$, then we use \eqref{CMHD} to get
\begin{equation}\label{conserve1}
\begin{aligned}
&~~~~\ddt\left(\frac{1}{2}\idt\rho |u|^2\dx+\frac{1}{2}\idt|B|^2 \dx +\idt \rho Q(\rho)\dx\right)  \\
&=\idt \rho u\cdot D_tu\dx+\idt B\cdot D_t B\dx+\idt \rho D_t Q(\rho)\dx+\frac{1}{2}\idt \rho D_t(1/\rho)|B|^2\dx\\
&=\idt u\cdot (B\cdot\nabla B)\dx-\idt u\cdot\nabla P \dx+\idt B\cdot(B\cdot\nabla u)\dx-\idt |B|^2(\nabla\cdot u)\dx \\
&~~~~+\idt p(\rho) \frac{D_t \rho}{\rho} \dx -\frac{1}{2}\idt \frac{D_t \rho}{\rho}|B|^2\dx.
\end{aligned}
\end{equation}

Integrating by part in the first term in the last equality, this term will cancel with $\idt B\cdot(B\cdot\nabla u)\dx$ because the boundary term and the other interior term vanish due to $B\cdot n|_{\p\DD_t}=0$ and $\dive B=0$. Also we integrate by parts in the second term and then use the continuity equation to get
\begin{equation}\label{conserve2}
\begin{aligned}
-\idt u\cdot\nabla P\dx&=\idt P(\nabla\cdot u)\dx-\underbrace{\ipdt (u\cdot N)P \dS}_{=0}=-\idt p\frac{D_t \rho}{\rho} \dx +\frac{1}{2}\idt |B|^2(\nabla\cdot u)\dx\\
&=-\idt p\frac{D_t \rho}{\rho} \dx +\idt |B|^2(\nabla\cdot u)\dx -\frac{1}{2}\idt |B|^2(\nabla\cdot u)\dx \\
&=-\idt p\frac{D_t \rho}{\rho} \dx +\idt |B|^2(\nabla\cdot u)\dx+\frac{1}{2}\idt \frac{D_t \rho}{\rho}|B|^2\dx.
\end{aligned}
\end{equation}where $dS$ is the surface measure of $\p\DD_t$.

Summing up \eqref{conserve1} and \eqref{conserve2}, one can get the energy conservation
\begin{equation}\label{conserve}
\ddt\left(\frac{1}{2}\idt\rho |u|^2\dx+\frac{1}{2}\idt|B|^2 \dx +\idt \rho Q(\rho)\dx\right)  =0.
\end{equation} When $B=\mathbf{0}$, one can see such energy conservation exactly coincides with that of free-boundary compressible Euler equations established in Lindblad-Luo \cite{LL2018priori}.

\paragraph*{Equation of state: Isentropic liquid}

 We assume the fluid considered in this paper is an isentropic liquid, i.e., there exists some constant $0<\rho_1<\rho_2$ such that $\rho_1\leq \rho\leq \rho_2$ as opposed to a gas\footnote{In the case of a gas, the boundary condition should be $\rho=0$.}, and the fluid pressure $p=p(\rho)$ is an increasing smooth function of $\rho$. Next we impose the following natural conditions on $\rho'(p)$ for some fixed constant $A_0>1$.
\begin{equation}\label{weight}
A_0^{-1}\leq|\rho^{(m)}(p)|\leq A_0~~\text{ for }m\geq 1.
\end{equation}For example, the equation of state $p(\rho)=\gamma^{-1}(\rho^\gamma-1)/c^2~(\gamma\geq 1)$ satisfies this relation.

\paragraph*{Rayleigh-Taylor sign condition}

We also need to impose the Rayleigh-Taylor sign condition
\begin{equation}\label{sign}
-\nabla_n P\geq c_0>0~~\text{ on }\p\DD_t,
\end{equation} where $c_0>0$ is a constant and $P:=p+\frac{1}{2}|B|^2$ is the total pressure. When $B=\mathbf{0}$, Ebin \cite{Ebin1987} proved the ill-posedness of the free-boundary incompressible Euler equations when the Rayleigh-Taylor sign condition is violated. For the free-boundary MHD equations, \eqref{sign} is also necessary: Hao-Luo \cite{HaoLuo2018ill} proved that the free-boundary problem of 2D incompressible MHD equations is ill-posed when \eqref{sign} fails. We also note that \eqref{sign} is only required for initial data and it propagates in a short time interval because one can prove it is $C_{t,x}^{0,\frac14}$ H\"older continuous by using Morrey's embedding. 

\paragraph*{Compatibility conditions on initial data}

To make the initial-boundary value problem \eqref{CMHD}-\eqref{MHDI} well-posed, the initial data has to satisfy certain compatibility conditions on the boundary. In fact, we need to require $P_0|_{\p\DD_0}$=0. Also the constraints on the magnetic field $\dive B=0$ and $B\cdot n|_{\p\DD}=\mathbf{0}$ requires that $\nabla\cdot B_0=0$ and $B_0\cdot n|_{\{0\}\times\p \DD_0}=\mathbf{0}$. Furthermore, we say the initial data satisfies the compatibility condition up to $k$-th($k\geq 0$) order if
\begin{equation}\label{cck}
D_t^j P|_{\{t=0\}\times\p\DD_0}=0~~\forall 0\leq j\leq k.
\end{equation}When \eqref{cck} is fulfilled for any $j\in\N$, we say the initial data satisfies the compatibility conditions to infinite order.


\subsection{History and background}

\subsubsection{Background in physics}

 The free-boundary problem \eqref{CMHD}-\eqref{MHDI} can be considered as the basic model of the plasma-vacuum free-interface problem which is important in the study of confined plasma both in laboratory and in astro-physical magnetohydrodynamics. The plasma is confined in a vacuum in which there is another magnetic field $\hat{B}$, and there is a free interface $\Gamma(t)$, moving with the motion of plasma, between the plasma region $\Omega_+(t)$ and the vacuum region $\Omega_{-}(t)$. This model requires that \eqref{CMHD} holds in the plasma region $\Omega_+(t)$ and the pre-Maxwell system holds in vacuum $\Omega_{-}(t)$:
\begin{equation}\label{outsideB}
\nabla\times\hat{B}=\mathbf{0},~~~\nabla\cdot\hat{B}=0.
\end{equation}
On the interface $\Gamma(t)$, it is required that there is no jump for the pressure or the normal components of the magnetic fields:
\begin{equation}\label{interface}
B\cdot n=\hat{B}\cdot n=0,~~~P:=p+\frac12|B|^2=\frac12|\hat{B}|^2
\end{equation}  where $n$ is the exterior unit normal to $\Gamma(t)$.  Finally, there is a rigid wall $W$ wrapping the vacuum region, on which the following boundary condition holds
\[
\hat{B}\times \hat{n}=\mathbf{J}~~\text{ on }W,
\] where $\mathbf{J}$ is the given outer surface current density (as an external input of energy) and $\hat{n}$ is the exterior normal to the rigid wall $W$. Note that for ideal MHD, the conditions $\dive B=0$ and $B\cdot n=0$ should also be constraints for initial data instead of imposed conditions. For details we refer to \cite[Chapter 4, 6]{MHDphy}.

Hence, the free-boundary problem \eqref{CMHD}-\eqref{MHDI} can be considered as a special case of plasma-vacuum model that the vacuum magnetic field $\hat{B}$ vanishes. It characterizes the motion of an isolated perfect conducting fluid in an electro-magnetic field.

\subsubsection{An overview of previous results}

In the past a few decades, there have been numerous studies of the free-boundary inviscid fluids. We start with incompressible Euler equations.

\paragraph*{Free-boundary Euler equations}

The free-boundary Euler equations have been studied intensively by a lot of authors. The first breakthrough in solving the local well-posedness (LWP) for the incompressible irrotational problem for general initial data came in the work of Wu \cite{Wu1997LWP, Wu1999LWP} who proved the LWP of 2D and 3D full water wave system. In the case of nonzero vorticity, Christodoulou-Lindblad \cite{CL2000priori} first established the a priori estimates and then Lindblad \cite{Lindblad2003LLWP1, Lindblad2005LWP1} proved the LWP by using Nash-Moser iteration. Coutand-Shkoller \cite{CS2007LWP, CS2010LWP} proved the LWP for incompressible Euler equations with or without surface tension and avoid the loss of regularity by introducing tangential smoothing method. We also refer to the related works \cite{ZZ2008LWP,ABZ2014LWP,SZ1,SZ2,SZ3} and references therein.

The study of compressible perfect fluid is not quite developed as opposed to the incompressible case. Lindblad \cite{Lindblad2003LLWP2, Lindblad2005LWP2} established the first LWP result by Nash-Moser iteration. Trakhinin \cite{Trakhinin2009gas} proved the LWP for the non-isentropic case by a hyperbolic approach and Nash-Moser iteration. Lindblad-Luo \cite{LL2018priori} established the first result of the a priori estimates and the incompressible limit. Then Luo \cite{Luo2018CWW} generalized \cite{LL2018priori} to compressible water wave with vorticity. Later, Ginsberg-Lindblad-Luo \cite{GLL2019LWP} proved the LWP for a self-gravitating liquid. Luo-Zhang \cite{LuoZhang2020CWW} proved the LWP for a compressible gravity water wave with vorticity. In the case of nonzero surface tension, we refer to Coutand-Hole-Shkoller \cite{CHS2013LWP} for the LWP and the vanishing surface tension limit and Disconzi-Luo \cite{DL2019limit} for the incompressible limit. For the case of a gas, we refer to \cite{CLS2010priori,CS2012LWP,Jang2014gas,LuoXinZeng2014gas,Hao2015gas,IT2020gas} and references therein.

\paragraph*{Free-boundary MHD equations: Incompressible case}

The study of free-boundary MHD is much more complicated than Euler equations due to the strong coupling between fluid and magnetic field and the failure of irrotational assumption. For the incompressible ideal free-boundary MHD under Rayleigh-Taylor sign condition, Hao-Luo \cite{HaoLuo2014priori} established the Christodoulou-Lindblad \cite{CL2000priori} type a priori estimates and Gu-Wang \cite{GuWang2016LWP} proved the LWP. Hao-Luo \cite{HaoLuo2019LLWP} also proved the LWP for the linearized problem when the fluid region is diffeomorphic to a ball and of large curvature. Luo-Zhang \cite{LuoZhang2019MHD2.5} proved the low regularity a priori estimates when the fluid domain is small. We also mention that Lee \cite{LeeMHD1,LeeMHD2} obtained a local solution via the vanishing viscosity-resistivity limit.

For the full plasma-vacuum model, Gu \cite{Guaxi1,Guaxi2} proved the LWP for the axi-symmetric case with nontrivial vacuum magnetic field in a non-simply connected domain under Rayleigh-Taylor sign condition. Hao \cite{Hao2017MHD} proved the LWP in the case of $\mathbf{J}=\mathbf{0}$. For the general case, all of the previous results require a non-collinearity condition $|B\times\hat{B}|\geq c_0>0$ on the free interface\footnote{The non-collinearity condition enhaces extra 1/2-order regularity of the free-interface than Taylor sign condtion \eqref{sign}. Such condition originates from the stabilization condition for the current-vortex sheet model.}. Under this condition, Morando-Trakhinin-Trebeschi \cite{MTT2014MHDLLWP} proved LWP for the linearized problem and then Sun-Wang-Zhang \cite{SWZ2017MHDLWP} proved the LWP for the full plasma-vacuum model. We also note that the study of the full plasma-vacuum model in ideal MHD under Rayleigh-Taylor sign condition is still an open problem when the vacuum magnetic field $\hat{B}$ is non-trivial with $\mathbf{J}\neq\mathbf{0}$. For the incompressible current-vortex sheets, we refer to Coulombel-Morando-Secchi-Trebeschi \cite{CMST2012MHDVS} for the a priori estimates and Sun-Wang-Zhang \cite{SWZ2015MHDLWP} for the LWP.

For incompressible ideal MHD with surface tension, Luo-Zhang \cite{LuoZhang2019MHDST} proved the a priori estimates and Gu-Luo-Zhang \cite{GuLuoZhang2021MHDST} proved the LWP. For incompressible dissipative MHD with surface tension, we refer to Chen-Ding \cite{ChenDingMHDST} for the inviscid limit for viscous non-resistive MHD, Wang-Xin \cite{WangXinMHDST} for the global well-posedness of the plasma-vacuum model for inviscid resistive MHD around a uniform transversal magnetic field, and Padula-Solonnikov \cite{Solonnikov} and Guo-Zeng-Ni \cite{GuoMHDSTviscous} for viscous-resistive MHD.

\paragraph*{Free-boundary MHD equations: Compressible case}

Compared with compressible Euler equations and incompressible MHD, compressible MHD has an extra coupling between the pressure wave and the magnetic field which makes the analysis completely different. Here we emphasize that there is a normal derivative loss in the div-curl analysis of compressible MHD. On the one hand, the second author \cite{Zhang2019CRMHD,Zhang2020CRMHD} recently observed that the magnetic resistivity exactly compensates the derivative loss mentioned above. However, it is still hopeless to derive the vanishing resistivity limit. On the other hand, one can still expect to establish the tame estimates for the linearized equation. Based on this and Nash-Moser iteration, Trakhinin-Wang \cite{TW2020MHDLWP,TW2021MHDSTLWP} recently proved the LWP for free-boundary compressible ideal MHD with or without surface tension. We also mention that Chen-Wang \cite{ChenWangCMHDVS} and Trakhinin \cite{Trakhinin2008CMHDVS} proved the LWP for the current-vortex sheets, and Secchi-Trakhinin \cite{Secchi2013CMHDLWP} proved the LWP for the full plasma-vacuum problem for compressible ideal MHD under the non-collinearity condition. However, Nash-Moser iteration leads to a big loss of regularity and does not give the continuous depedence on initial data. It is still unknown whether the local well-posedness result can be improved such that the regularity loss can be avoided and the continuous dependence on initial data can be established.

In this paper, we first prove the a priori estimates without loss of regularity for the free-boundary compressible ideal MHD system in the anisotropic Sobolev spaces. Our proof is based on the modified Alinhac good unknown method, full utilization of the cancellation structure of MHD system and very delicate analysis under the setting of anisotropic Sobolev spaces. Using a parallel argument, we can also prove the uniqueness and the continuous dependence on initial data provided the solution exists. Then we prove a local existence result and a continuation criterion for the smooth solutions with smooth data. Based on these results, we can improve the local existence result to the case that the initial data only satisfies the compatibility conditions up to finite order, such that the regularity loss can be avoided and the continuous dependence on initial data can be established.

\subsection{Reformulation in Lagrangian coordinates and main result}

We use Lagrangian coordinates to reduce the free-boundary problem to a fixed-domain problem. We assume $\Omega:=\T^2\times(-1,1)$ to be the reference domain and $\Gamma:=\T^2\times(\{-1\}\cup\{1\})$ to be the boundary. The coordinates on $\Omega$ is $y:=(y',y_3)=(y_1,y_2,y_3)$. We define $\eta:[0,T]\times\Omega\to \DD$ as the flow map of velocity field $u$, i.e.,
\begin{align}
\p_t\eta(t,y) = u(t, \eta(t,y)),\q \eta(0,y)=\eta_0(y),
\end{align}where $\eta_0$ is a diffeomorphism between $\Omega$ and $\DD_0$. For technical simplicity\footnote{The domain $\T^2\times(-1,1)$ is known to be the reference domain. Using a partition of unity, e.g., \cite{CS2007LWP}, a general bounded domain can also be treated in the same way. Choosing a reference domain allows us to focus on the real issues and avoid the involved calculation caused by partition ofunity. Indeed, our proof is also applicable to the case that $\eta_0$ is a general diffeomorphism that has the same regularity of $v_0$ if we use similar technical modifications as in \cite{GuLuoZhang2021MHD0ST}.} we assume $\eta_0=$ Id. By chain rule, it is easy to see that the material derivative $D_t$ becomes $\p_t$ in the $(t,y)$ coordinates and the free-boundary $\p\DD_t$ becomes fixed ($\Gamma=\T^2\times(\{-1\}\cup\{1\})$). We introduce the Lagrangian variables as follow: $v(t,y):=u(t, \eta(t,y))$, $b(t,y):=B(t, \eta(t,y))$, $q(t,y):=p(t, \eta(t,y))$, $Q(t,y):=P(t,\eta(t,y))$ and $R(t,y):=\rho(t, \eta(t,y))$.

Let $\p=\p_y$ be the spatial derivative in Lagrangian coordinates and we define $\dive Y=\p_iY^i$ to be the (Lagrangian) divergence of the vector field $Y$. We introduce the matrix $A=[\p \eta]^{-1}$, specifically $A^{li}:=\frac{\p y^{l}}{\p x^i}$ where $x^i=\eta^i(t,y)$ is the $i$-th variable in Eulerian coordinates. From now on, we define $\nabla_A^i=\frac{\p}{\p x^i}=A^{li}\p_l$ to be the covariant derivative in Lagrangian coordinates (or say Eulerian derivative) and $\diva X:=\nabla_A\cdot X=A^{li}\p_l X_i$ to be the Eulerian divergence mof the vector field $X$. In the manuscript, we adopt the convention that the Latin indices range over $1,2,3$.  In addition, since $\eta(0,\cdot)=\text{Id}$, we have $A(0,\cdot)=I$, where $I$ is the identity matrix, and $(u_0,B_0,p_0)$ and $(v_0,b_0,q_0)$ agree respectively.

In terms of $\eta,v,b,q,R$, the system \eqref{CMHD}-\eqref{sign} becomes

\begin{equation}\label{CMHDL1}
\begin{cases}
\p_t\eta=v &~~~ \text{in}\,[0,T]\times\Omega\\
R\p_t v=\left(b\cdot\pa\right)b-\pa Q,~~Q=q+\frac12|b|^2 &~~~ \text{in}\,[0,T]\times\Omega\\
\p_tR+R\diva v=0 &~~~ \text{in}\,[0,T]\times\Omega\\
q=q(R) &~~~ \text{in}\,[0,T]\times\bar{\Omega}\\
\p_t b=\left(b\cdot\pa\right)v-b\diva v  &~~~ \text{in}\,[0,T]\times\Omega\\
\diva b=0  &~~~ \text{in}\,[0,T]\times\Omega\\
Q=0,~~b_iA^{li}N_l=0 &~~~\text{on}\,[0,T]\times\Gamma\\
-\frac{\p Q}{\p N}|_{\Gamma}\geq c_0>0  &~~~\text{on}\,\{t=0\}\\
(\eta,v,b,q,R)|_{t=0}=(\text{Id},v_0,b_0,q_0,\rho_0).
\end{cases}
\end{equation}Here $N=(0,0,\pm 1)$ is the unit outer normal of the boundary $\T^2\times\{\pm 1\}$ and $q=q(R)$ is a strictly increasing function of $R$ with $A_0^{-1}\leq q'(R)\leq A_0$ for some constant $A_0>1$.

Let $J:=\det [\p\eta]$ and $\hat{A}:=JA$. Then we have the Piola's identity
\begin{align}\label{piola}
\p_l \hat{A}^{li}=0,
\end{align}
and $J$ satisfies
\begin{align}\label{divJ}
\p_t J = J\diva v
\end{align} which together with $\p_t R+R\diva v=0$ gives that $\rho_0=RJ$.

Suppose $D$ is the derivative $\p$ or $\p_t$, then we have the following identity
\begin{equation}\label{da}
DA^{li}=-A^{lr}~\p_kD\eta_r~A^{ki}.
\end{equation}

Next we express the magnetic field $b$ in terms of $b_0$ and $\eta$ in the following Lemma. This is called the ``frozen effect of the magnetic field".

\begin{lem}\label{b}
We have $b=J^{-1}\bp\eta.$
\end{lem}
\begin{proof}
Let us first compute the equation of $b/R$. We have
\begin{align*}
\p_t\left(\frac{b}{R}\right)=&\frac{1}{R}\p_t b +b\p_t \left(\frac{1}{R}\right)=\frac{1}{R}\p_t b +b\p_t \left(\frac{J}{RJ}\right)=\frac{1}{R}\p_t b +\frac{b}{\rho_0}\p_t J\\
=&\frac{1}{R}\left((b\cdot \pa) v-b\diva v\right)+\frac{b}{\rho_0}J\diva v=\frac{b}{R}\cdot\pa v-\frac{b}{R}\diva v+\frac{b}{R}\diva v\\
=&\left(\frac{b}{R}\cdot\pa\right) v.
\end{align*}

Therefore, invoking \eqref{da} we have
\begin{align*}
\p_t\left(\frac{b_i}{R} A^{li}\right)=&\p_t\left(\frac{b_i}{R}\right) A^{li}+\frac{b_i}{R}\p_t A^{li}= \frac{b_j}{R}~A^{kj}~\p_k v_i~A^{li}-\frac{b_i}{R}~A^{lj}~\p_k v_j~A^{ki}=0,
\end{align*}
which implies $\frac{b_i}{R} A^{li}=\frac{b_{0i}}{\rho_0} \delta^{li}=\frac{{b_0}^l}{\rho_0} $, i.e., $ b_i A^{li}=\frac{{b_0}^l R}{\rho_0}=J^{-1} b_0^l $. Finally, the identity $A^{li}\p_l \eta_i=1$ gives us $b_i=J^{-1}b_0^l\p_l\eta_i=J^{-1}\bp\eta_i$.
\end{proof}

Inserting $\rho_0=RJ$ and Lemma \ref{b} into \eqref{CMHDL1}, we get the following system with the initial constraints $\dive b_0 =0$ in $\Omega$, $b_0^3=0$ on $\Gamma$ and $-\frac{\p Q_0}{\p N}|_{\Gamma}\geq c_0>0$. From now on, we call these three conditions to be ``initial constraints" without more explanation.
\begin{equation}\label{CMHDL}
\begin{cases}
\p_t\eta=v &~~~ \text{in}\,[0,T]\times\Omega\\
R\p_t v=J^{-1}\bp\left(J^{-1}\bp\eta\right)-\pa Q,~~Q=q+\frac12|J^{-1}\bp\eta|^2 &~~~ \text{in}\,[0,T]\times\Omega\\
\frac{JR'(q)}{\rho_0}\p_tq+\diva v=0 &~~~ \text{in}\,[0,T]\times\Omega\\
q=q(R)\text{ strictly increasing} &~~~ \text{in}\,[0,T]\times\bar{\Omega}\\
Q=0, &~~~\text{on}\,[0,T]\times\Gamma\\
(\eta,v,q,Q)|_{t=0}=(\text{Id},v_0,q_0,Q_0),~~Q_0=q_0+\frac12|b_0|^2.
\end{cases}
\end{equation}

Before stating our results, we should first define the anisotropic Sobolev space $H_*^m(\Omega)$ for $m\in\N^*$. Let $\sigma=\sigma(y_3)$ be a cutoff function on $[-1,1]$ defined by $\sigma(y_3)=(1-y_3)(1+y_3)$. Then we define $H_*^m(\Omega)$ for $m\in\N^*$ as follows
\[
H_*^m(\Omega):=\left\{f\in L^2(\Omega)\bigg| (\sigma\p_3)^{i_4}\p_1^{i_1}\p_2^{i_2}\p_3^{i_3} f\in L^2(\Omega),~~\forall i_1+i_2+2i_3+i_4\leq m\right\},
\]equipped with the norm
\[
\|f\|_{H_*^m(\Omega)}^2:=\sum_{i_1+i_2+2i_3+i_4\leq m}\|(\sigma\p_3)^{i_4}\p_1^{i_1}\p_2^{i_2}\p_3^{i_3} f\|_{L^2(\Omega)}^2.
\] For any multi-index $I:=(i_0,i_1,i_2,i_3,i_4)\in\N^5$, we define
\[
\p_*^I:=\p_t^{i_0}(\sigma\p_3)^{i_4}\p_1^{i_1}\p_2^{i_2}\p_3^{i_3},~~\lee I\ree:=i_0+i_1+i_2+2i_3+i_4,
\]and define the \textbf{space-time anisotropic Sobolev norm} $\|\cdot\|_{m,*}$ by
\[
\|f\|_{m,*}^2:=\sum_{\lee I\ree\leq m}\|\p_*^I f\|_{L^2(\Omega)}^2=\sum_{i_0\leq m}\|\p_t^{i_0}f\|_{H_*^{m-i_0}(\Omega)}^2.
\]

We define $f_{(j)}=\p_t^j f|_{t=0}$ for $j\in\N$. The main results in this manuscript are the following theorems. The first one is the improved local existence theorem for \textit{smooth} data satisfying the compatibility conditions up to infinite order.
\begin{thm}[Local existence for smooth solutions]\label{CMHDsmooth}
Let $(v_0,b_0,Q_0)\in C^{\infty}(\bar\Omega)$ be the initial data of \eqref{CMHDL} satisfying
\begin{itemize}
\setlength{\itemsep}{0pt}
\setlength{\parsep}{0pt}
\setlength{\parskip}{0pt}
\item the compatibility conditions up to infinite order, i.e., $Q_{(j)}|_{\Gamma}=0,~\forall j\geq 0,~j\in\Z$;
\item the initial constraints $\dive b_0=0$ in $\Omega$, $b_0^3|_{\Gamma}=0$ and the Rayleigh-Taylor sign condition $-\frac{\p Q_0}{\p N}|_{\Gamma}\geq c_0>0$.
\end{itemize}Then there exists some $T_0>0$ only depending on initial data, $c_0$ and $A_0$ defined in \eqref{weight}, such that \eqref{CMHDL} has a unique smooth solution $(\eta,v,b,Q)$ in $C^{\infty}([0,T_0]\times\bar\Omega)$.
\end{thm}
\begin{rmk}[On the existence of smooth initial data satisfying the compatibility conditions up to infinite order]
One should prove the existence of smooth initial data satisfying the compatibility conditions up to infinite order. This can be done by a parallel argument as in Lindblad \cite[Lemma 16.2]{Lindblad2005LWP2}. See the explanation in Appendix \ref{appendix data inf}.
\end{rmk}

The next two theorems show the a priori bounds without loss of regularity, the uniqueness and continuous dependence on initial data provided that the solution exists. They also give the energy estimates without loss of regularity and the continuous dependence on intial data in anisotropic Sobolev spaces for the smooth solution obtained in Theorem \ref{CMHDsmooth}.

\begin{thm}[A priori estimates]\label{CMHDEE}
Assume $m\geq 8$ is an integer. Let the initial data be $(v_0,b_0,Q_0)\in H_*^{m}(\Omega)$ satisfying that
\begin{itemize}
\setlength{\itemsep}{0pt}
\setlength{\parsep}{0pt}
\setlength{\parskip}{0pt}
\item  $(v_{(j)},b_{(j)},Q_{(j)})\in H_*^{m-j}(\Omega)$ for $1\leq j\leq m$, where $f_{(j)}:=\p_t^j f|_{t=0}$;
\item the compatibility condition holds up to $(m-1)$-th order, i.e., $Q_{(j)}|_{\Gamma}=0$ for $0\leq j\leq m-1$;
\item the initial constraints $\dive b_0=0$ in $\Omega$, $b_0^3|_{\Gamma}=0$ and the Rayleigh-Taylor sign condition $-\frac{\p Q_0}{\p N}|_{\Gamma}\geq c_0>0$.
\end{itemize} Then there exists some $T_1>0$ only depending on $\|v_0,b_0,Q_0\|_m$, $c_0$ and $A_0$ (defined in \eqref{weight}), such that the solution $(\eta,v,Q)$ to the system \eqref{CMHDL} satisfies the following estimates in $[0,T_1]$
\begin{align}\label{noloss}
\sup_{0\leq t\leq T_1}\EE_{m}(t)\leq P(\EE_{m}(0)),
\end{align}under the a priori assumptions on $[0,T_1]$
\begin{align}
\label{small} \|J-1\|_{m-1,*}\leq& \frac{1}{4}\\
\label{sign2} -\frac{\p Q}{\p N}\geq& \frac{3}{4}c_0.
\end{align}
Here the energy functional $\EE(t)$ is defined to be
\begin{equation}\label{EE}
\EE_{m}(t):=\|\eta(t,\cdot)\|_{m,*}^2+\|v(t,\cdot)\|_{m,*}^2+\|J^{-1}\bp\eta(t,\cdot)\|_{m,*}^2+\|q(t,\cdot)\|_{m,*}^2+\sum_{\lee I\ree = m}\left|A^{3i}\p_*^I\eta_i\right|_0^2,
\end{equation} and $P(\cdots)$ is a generic polynomial in its arguments.
\end{thm}

\begin{thm}[Continuous dependence on initial data and uniqueness]\label{CMHDEE2}
Assume $m\geq 8$ is an integer. Let $(v_0^{(i)},b_0^{(i)},Q_0^{(i)})\in H_*^{m}(\Omega)~(i=1,2)$ be two initial datum satisfying the hypothesis in Theorem \ref{CMHDEE}. Let $(\eta^{(i)},v^{(i)},Q^{(i)})$ be the solution to \eqref{CMHDL} with initial data $(v_0^{(i)},b_0^{(i)},Q_0^{(i)})$. Define $[f]=f^{(1)}-f^{(2)}$ for any function $f$ in $\bar\Omega$ and define the energy functional $[\EE](t)$  to be
\begin{equation}\label{EE2}
[\EE]_{m}(t):=\|[\eta](t,\cdot)\|_{m-2,*}^2+\|[v](t,\cdot)\|_{m-2,*}^2+\|(J^{-1}b_0)^{(1)}\cdot\p[\eta](t,\cdot)\|_{m-2,*}^2+\|[q](t,\cdot)\|_{m-2,*}^2+\sum_{\lee I\ree = m-2}\left|A^{(1)3i}\p_*^I[\eta]_i\right|_0^2.
\end{equation}
Then there exists some $T_2>0$ depending on $\|v_0^{(i)},b_0^{(i)},Q_0^{(i)}\|_m~(i=1,2)$, $c_0$ and $A_0$ such that the following estimates hold
\begin{equation}\label{cont dep}
\sup_{0\leq t\leq T_2}[\EE]_{m}(t)\leq [\EE]_{m}(0)P(\EE_m(0))\leq P(\|[v_0],[b_0],[q_0]\|_{m-2},\EE_m(0)),
\end{equation}where $P(\cdot)>0$ is a generic polynomial of its arguments.
\end{thm}

\begin{rmk}[Control of $\EE_m(0)$] It would be better to construct the initial data $(v_0,b_0,Q_0)$ satisfying the compatibility conditions up to $(m-1)$-th order in $H^{m}(\Omega)\hookrightarrow H_*^{m}(\Omega)$, such that
\begin{equation}
\sum_{j=1}^{m}\|(v_{(j)},b_{(j)},Q_{(j)})\|_{H^{m-j}(\Omega)}^2\lesssim P(K_{m}),
\end{equation}where we define $K_{m}:=\|v_0\|_{H^{m}(\Omega)}^2+\|b_0\|_{H^{m}(\Omega)}^2+\|Q_0\|_{H^{m}(\Omega)}^2$. In particular, by the Sobolev embedding $H^{m-j}(\Omega)\hookrightarrow H_*^{m-j}(\Omega)$ for $0\leq j\leq m$, we have
\begin{equation}\label{lossless}
\EE_{m}(0)\lesssim P(K_{m}).
\end{equation}
If we only focus on $(v_0,b_0,Q_0)\in H_*^{m}(\Omega)$ then we can only get $(v_{(j)},b_{(j)},Q_{(j)})\in H_*^{m-2j}(\Omega)$ and thus $\EE_{m}(0)<\infty$ may fail. See Section \ref{data} for detailed discussion.
\end{rmk}

Next, we want to extend the local existence theorem to the case of initial data satisfying compatibility conditions up to finite order. To achieve this, we need a continuation criterion for the smooth solution obtained in Theorem \ref{CMHDsmooth}, which shows that, for any $m\in\N^*, m\geq8$, the $\|\cdot\|_{m,*}$ norm of a smooth solution remains bounded as long as the $\|\cdot\|_{k,*}$ norms $(k\leq m-1)$ are bounded.
\begin{thm}[Continuation of smooth solution]\label{CMHDcontinue}
Assume $m\geq 8$ to be an integer. For the smooth solution $(\eta,v,b,Q)$ obtained in Theorem \ref{CMHDsmooth}, we define
\begin{equation}
T^*:=\sup\left\{T>0\big|(\eta,v,b,Q)\text{ can be extended in }C^{\infty}([0,T]\times\bar\Omega)\right\}.
\end{equation}If $T^*<+\infty$, then either $\lim\limits_{t\nearrow T^*}\EE_{m-1}(t)=+\infty$ for some $m$ or $\lim\limits_{t\nearrow T^*}\inf\limits_{\Gamma}(-\frac{\p Q}{\p N})=0$ holds.
\end{thm}

\begin{rmk}  The proof of this continuation criterion requires the energy estimates for $\EE_m(t)$ to be linear in the highest-order terms. To achieve this, it suffices to carefully analyze each commutator term in the anisotropic Sobolev spaces to ensure the linearity of the highest-order terms, such that the energy inequality becomes
\[
\EE_m(t)\lesssim \EE_m(0)+\int_0^t P(\EE_{m-1}(\tau))\EE_m(\tau)\mathrm{d}\tau.
\]  This also inherits the frameworks of \cite{CL2000priori,LL2018priori} which proved that the solutions to free-boundary Euler equations, if exist, can be extended after time $t=T_*$ provided that all lower order terms of $v,q$ and the second fundamental form of the free surface are bounded at time $t=T_*$. 
\end{rmk}

Finally, we show that, one can prove the local well-posedness for initial data (not necessarily smooth) satisfying the compatibility conditions up to only finite order, provided that one can construct a sequence of smooth data satisfying the compatibility conditions up to infinite order that \textbf{converges to the given data} in $H^m(\Omega)$. However, it is still unknown to how achieve such construction in general due to some technical difficulties. (We expect this to be true since one can construct data satisfying infinitely many compatibility conditions as in \cite{Lindblad2005LWP2} and one can construct data satisfying any number of compatibility conditions approximating given data as in \cite{LL2018priori,Zhang2021elasto}.) We have the following theorem.
\begin{thm}[Local well-posedness]\label{CMHDLWPm}
Assume $m\geq 8$ to be an integer. Let $(v_0,b_0,Q_0)$ be the initial data (not necessarily smooth!) of \eqref{CMHDL} satisfying
\begin{itemize}
\setlength{\itemsep}{0pt}
\setlength{\parsep}{0pt}
\setlength{\parskip}{0pt}
\item the compatibility conditions up to $(m-1)$-th order, i.e., $Q_{(j)}|_{\Gamma}=0,~0\leq j\leq m-1$;
\item the initial constraints $\dive b_0=0$ in $\Omega$, $b_0^3|_{\Gamma}=0$ and the Rayleigh-Taylor sign condition $-\frac{\p Q_0}{\p N}|_{\Gamma}\geq c_0>0$.
\end{itemize}
Assume also there exists a sequence of smooth data $\{(v_0^{(n)},b_0^{(n)},Q_0^{(n)})\}_{m\in\N^*}$ satisfying the compatibility conditions up to infinite order that converges to the given data $(v_0,b_0,Q_0)$ in $H^m(\Omega)$, i.e.,
\[
\lim_{n\to\infty}\|v_0^{(n)}-v_0\|_m+\|b_0^{(n)}-b_0\|_m+\|Q_0^{(n)}-Q_0\|_m=0.
\]
Then there exists some $T_m>0$ only depending on $\|v_0,b_0,Q_0\|_m$, $c_0$ and $A_0$ defined in \eqref{weight}, such that the solution $(\eta,v,b,Q)$ to \eqref{CMHDL} exists in $C([0,T_m]; H_*^{m}(\Omega))$. The solution also satisfies the conclusions of Theorem \ref{CMHDEE}-Theorem \ref{CMHDEE2}, that is, the a priori estimates without loss of regularity in $\|\cdot\|_{m,*}$ norm, the uniqueness and continuous dependence on initial data in $\|\cdot\|_{m-2,*}$ norm.
\end{thm}

\paragraph*{Organization of the paper. }In Section \ref{stat}, we briefly introduce the strategies and the main techniques used in our proof. In Section \ref{sect lemma} we record the lemmas which will be repeatedly used in the manuscript. Then we show the detailed analysis of MHD system in anisotropic Sobolev space in Section \ref{sect n4} $\sim$ Section \ref{normal8}. And we conclude the a priori estimates without loss of regularity, the uniqueness and the continuos dependence on data in Section \ref{apriori}. Finally, in Section \ref{sect LWP}, we explain how to improve the Nash-Moser iteration scheme in \cite{TW2020MHDLWP} such that a local existence theorem for $C^{\infty}$ data can be proved, and then show that continuation criterion in anisotropic Sobolev spaces by further analysis of the commutators. After that, we prove Theorem \ref{CMHDLWPm} by using the conclusions of Theorem \ref{CMHDsmooth} $\sim$ Theorem \ref{CMHDEE2}. The construction of initial data are discussed in Appendix \ref{appendix data}.

\noindent\textbf{List of Notations: }
\begin{itemize}
\setlength{\itemsep}{0pt}
\setlength{\parsep}{0pt}
\setlength{\parskip}{0pt}
\item $\Omega:=\T^2\times(-1,1)$ and $\Gamma:=\T^2\times(\{-1\}\cup\{1\})$.
\item $\|\cdot\|_{s}$:  We denote $\|f\|_{s}:= \|f(t,\cdot)\|_{H^s(\Omega)}$ for any function $f(t,y)\text{ on }[0,T]\times\Omega$.
\item $|\cdot|_{s}$:  We denote $|f|_{s}:= |f(t,\cdot)|_{H^s(\Gamma)}$ for any function $f(t,y)\text{ on }[0,T]\times\Gamma$.
\item $\|\cdot\|_{m,*}$: For any function $f(t,y)\text{ on }[0,T]\times\Omega$, $\|f\|_{m,*}^2:= \sum_{\lee I\ree\leq m}\|\p_*^If(t,\cdot)\|_{L^2}^2$ denotes the $m$-th order space-time anisotropic Sobolev norm of $f$.
\item $P(\cdots)$:  A generic polynomial in its arguments;
\item $\PP_0$:  $\PP_0=P(\EE(0))$;
\item $[T,f]g:=T(fg)-fT(g)$, and $[T,f,g]:=T(fg)-T(f)g-fT(g)$, where $T$ denotes a differential operator and $f,g$ are arbitrary functions.
\item $\TP$: $\TP=\p_1,\p_2$ denotes the spatial tangential derivative.
\item $\nabla^i_A f:=A^{li}\p_{l} f$ denotes the covariant (Eulerian) derivative.
\item $X\cdot\nabla_A f$: For any function $f$ and vector field $X$, such notation denotes the inner-product defined by $X\cdot \nabla_A f:= X_pA^{lp}\p_l f$.
\item $X\cdot \nabla_A Y\cdot \nabla_A f$: For any function $f$ and vector field $X,Y$, such notation denotes the inner-product defined by $X\cdot \nabla_A Y\cdot \nabla_A f:= X_pA^{lp}\p_l Y_r A^{mr}\p_m f$.
\end{itemize}

\paragraph*{Acknowledgement.}The authours thank the anonymous referees for their comments and suggestions that help us improve the quality of this paper. Hans Lindblad was supported in part by Simons Foundation Collaboration Grant 638955. Junyan Zhang would like to thank Tao Wang and Chenyun Luo for helpful discussion.

\section{Strategy of the proof}\label{stat}

Before going to the details, we introduce the basic strategies and techniques of our proof, especially for the proof of energy estimates without loss of regularity. At the end of this section we will also explain how we improve the local well-posedness result using our energy estimates and the local existence result of \cite{TW2020MHDLWP}. From now on, we will only show the proof for Theorem \ref{CMHDEE}$\sim$\ref{CMHDEE2} for $m=8$ and drop the index $m$ in $\EE_m(t)$ for simplicity of notations.

\subsection{Choice of the function spaces}\label{stat1}

The compressible MHD system \eqref{CMHD}-\eqref{MHDI} is a hyperbolic system with charactersitic boundary conditions and violates the uniform Kreiss-Lopatinski\u{\i} condition \cite{lopatinskii}. This usually causes a loss of normal derivative. For certain types of such hyperbolic system, e.g., compressible Euler equations \cite{LL2018priori,Trakhinin2009gas}, one can control the normal derivatives by the div-curl analysis so that the energy estimates and the LWP can be established in standard Sobolev spaces. However, such div-curl analysis is not applicable to compressible ideal MHD. In fact, taking curl eliminates the symmetry enjoyed by the equations, and there is also a derivative loss in the source term of the wave equation of pressure which is the key to the divergence estimates. For related details, we refer to \cite[Section 1.5]{Zhang2019CRMHD}.

To compensate such derivative loss, Chen \cite{ChenSX1982} first introduced the anisotropic Sobolev spaces $H_*^m$ to study the hyperbolic system with characteristic boundary conditions. Then Yanagisawa-Matsumura \cite{1991MHDfirst} established the first LWP result for the fixed-domain problem of compressible ideal MHD in anisotropic Sobolev spaces. Later, \cite{1991MHDfirst} was improved by Secchi \cite{Secchi1995,Secchi1996} such that the regularity loss was avoided. On the other hand, Ohno-Shirota \cite{OS1998MHDill} constructed an explicit counterexample to show the ill-posedness for the linearized fixed-domain problem for compressible MHD in $H^l(l\geq 2)$.

Hence, the failure of div-curl analysis and the results of the fixed-domain problem \cite{ChenSX1982,1991MHDfirst,Secchi1995,OS1998MHDill,ChenSX2013} motivate us to study the free-boundary compressible ideal MHD system under the setting of anisotropic Sobolev spaces instead of standard Sobolev spaces. However, we emphasize that it is still difficult to directly generalize Secchi \cite{Secchi1995} to the free-boundary problem due to the following three reasons:
\begin{enumerate}
\item The regularity of the boundary is no longer $C^{\infty}$ as in the case of fixed domain. In fact, the regularity of the free boundary enters to the highest order.
\item The regularity of the flow map is limited. After reducing the free-boundary problem to a fixed-domain problem, the commutator of the covariant derivative and the full derivative cannot be controlled directly.
\item The Eulerian normal velocity $u\cdot n$ does not vanish on the free boundary. However, $u\cdot n=0$ plays an important role in the proof of \cite{1991MHDfirst, Secchi1995}.
\end{enumerate}

In fact, our analysis in the presenting manuscript is based on the modified Alinhac good unknown method, subtle cancellation structures of MHD system and the utilization of the anisotropy of the function space $H_*^m$. Here we also emphasize that our strategy is completely applicable to compressible Euler equations just by setting $b_0=0$. Our result also gives an alternative energy estimate for compressible Euler equations without the analysis of div-curl decomposition or the wave equation.

\subsection{Motivation for introducing Alinhac good unknowns}\label{stat2}

Denote $\p_*^I=\p_t^{i_0}(\sigma\p_3)^{i_4}\p_1^{i_1}\p_2^{i_2}\p_3^{i_3}$ with $\lee I\ree:=i_0+i_1+i_2+2i_3+i_4=8$. For simplicity of the notations, we use $\|\cdot\|_{s},|\cdot|_s$ to represent the $H^s(\Omega)$ norm and the $H^s(\Gamma)$ norm respectively. Taking $\p_*^I$ in the second equation of \eqref{CMHDL} and multiplying $J$, we get
\[
\rho_0\p_t \p_*^I v=-J\p_*^I(\pa Q)+\bp\p_*^I b+[\p_*^I,\bp]b-J[\p_*^I,R]\p_t v
\]
In the energy estimates, we need to commute $\pa$ with $\p_*^I$ and then integrate by parts. However, the commutator $[\p_*^I,A^{li}]\p_l f$ contains the following terms whose $L^2(\Omega)$-norms cannot be controlled in the anisotropic Sobolev space
\begin{itemize}
\setlength{\itemsep}{0pt}
\setlength{\parsep}{0pt}
\setlength{\parskip}{0pt}
\item $(\p_*^I A^{li})(\p_l f)$, which cannot be controlled even in the standard Sobolev spaces when $i_0=0$;
\item $(\p_*^{I-I'}A^{li})(\p_*^{I'}\p_l f)$, when $l=1,2$ since $A^{li}$ consists of $(\TP\eta)(\p_3\eta)$;
\item $(\p_*^{I'} A^{li})(\p_*^{I-I'}\p_lf)$, when $l=3$,
\end{itemize} where $f=Q$ or $v_i$ and $I'$ is a multi-index with $\lee I'\ree=1$. To overcome such difficulty, we can use the ideas of the Alinhac good unknown method, i.e., we can rewrite $\p_*^I(\pa Q)$ and $\p_*^I(\pa\cdot v)$ in terms of the sum of the covariant derivative part and the commutator part satisfying
\begin{align}
\label{goodq} \p_*^I(\pa Q)=\pa\QQ+C(Q),\text{ with }\|\QQ-\p_*^I Q\|_0+\|\p_t(\QQ-\p_*^I Q)\|_0+\|C(Q)\|_0\leq P(\EE(t)),\\
\label{goodv} \p_*^I(\pa\cdot v)=\pa\cdot\VV+C(v),\text{ with }\|\VV-\p_*^I v\|_0+\|\p_t(\VV-\p_*^I v)\|_0+\|C(v)\|_0\leq P(\EE(t)).
\end{align}Here $\QQ,\VV$ are called the ``Alinhac good unknowns" of $Q,v$ (The precise expressions will be determined later).

In other words, the above analysis shows that the essential highest order term in $\p_*^I(\pa f)$ is not the term got by simply commuting $\p_*^I$ with $\pa$. Instead, the essential highest order term in $\p_*^I(\pa f)$ is \textbf{exactly} the covariant derivative of the Alinhac good unknown of $f$, and the good unknowns $\VV$ and $\QQ$ are essentially formed by replacing the derivatives in the Lagragian coordinates $\p_*^I$ by the covariant derivatives with respect to the Eulerian coordinates expressed in the Lagrangian coordinates. Such crucial fact was first observed by Alinhac \cite{Alinhacgood89} and has been widely used for quasilinear hyperbolic system. In the study of free-surface fluid, such method was first implicitly used in the $Q$-tensor energy introduced by Christodoulou-Lindblad \cite{CL2000priori} which was later generalized by \cite{HaoLuo2014priori,LL2018priori,Luo2018CWW,Ginsberg2018rel,Zhang2019CRMHD}. See also \cite{MRgood2017,WangXingood,GuWang2016LWP,LuoZhang2020CWW,Zhang2020CRMHD,GL2021rel} for the explicit applications.

Under the setting of \eqref{goodq}-\eqref{goodv}, we can do the energy estimates by analyzing the Alinhac good unknowns via their evolution equation instead of the $\p_*^I$-differentiated variables.
\begin{align}\label{goodeq}
\rho_0\p_t\VV=-J\pa \QQ+\bp(\p_*^I b)+\underbrace{\left(\rho_0\p_t(\VV-\p_*^I v)-C(Q)+[\p_*^I,\bp]b-J[\p_*^I,\rho]\p_t v\right)}_{=:\FF}.
\end{align} Taking $L^2(\Omega)$ inner product of \eqref{goodeq} and $\VV$ and then integrating by parts, we can get the energy identity
\begin{equation}
\frac12\ddt\io\rho_0|\VV|^2=-\io (\p_*^I(J^{-1}\bp\eta))\cdot\bp\VV\dy+\io J\QQ(\pa\cdot\VV)\dy+\io \FF\cdot\VV\dy-\ig JA^{3i}N_3\QQ\VV_i\dyy.
\end{equation}By direct computation we can prove $\|\FF\|_0\leq P(\EE(t))$, so it remains to control
\begin{align}
K_1:=&-\io (\p_*^I(J^{-1}\bp\eta))\cdot\bp\VV\dy,\\
I_1:=&\io J\QQ(\pa\cdot\VV)\dy,\\
\label{IBIB} IB:=&-\ig JA^{3i}N_3\QQ\VV_i\dyy,
\end{align}where $\dyy:=\dy_1\dy_2$ is the area unit of the boundary $\Gamma$.

\subsection{Interior estimates and cancellation structure}\label{stat3}

Below we use ``$\cdots$" to represent the terms whose $L^2$ norms can be directly controlled by $P(\EE(t))$. The term $K_1$ gives the energy of the magnetic field. Recall that the top-order term in $\VV$ is $\p_*^I v=\p_*^I\p_t\eta$ which yields
\begin{align*}
K_1=&-\io (\p_*^I(J^{-1}\bp\eta^i))(\bp\p_*^I\p_t\eta_i)\dy+\cdots\\
=&-\frac12\ddt\io J|\p_*^I(J^{-1}\bp\eta)|^2\dy-\io\p_*^I(J^{-1}\bp\eta^i)\,(J^{-1}\bp\eta_i)\p_t\p_*^I J\dy+\cdots\\
=&-\frac12\ddt\io J|\p_*^I(J^{-1}\bp\eta)|^2\dy\underbrace{-\io J\p_*^I(J^{-1}\bp\eta^i)\,(J^{-1}\bp\eta_i)\p_*^I(\diva v)\dy}_{=:K_{11}}+\cdots,
\end{align*}where we use $b=J^{-1}\bp\eta$ and $\p_t J=J\diva v$. Note that $K_{11}$ cannot be directly controlled due to the presence of $\p_*^I(\diva v)$. Instead, it will be exactly cancelled by another term produced by $I_1$.

The term $I_1$ gives the energy of the fluid pressure $q$ and the cancellation structure with $K_{11}$. Recall that $\pa\cdot\VV=\p_*^I(\diva v)-C(v)$ and $Q=q+\frac12 |J^{-1}\bp\eta|^2$. We get
\begin{align*}
I_1=&\io J(\p_*^I q)\p_*^I(\diva v)+\io J\left(\p_*^I\left(\frac12|J^{-1}\bp\eta|^2\right)\right)\p_*^I(\diva v)\dy+\cdots\\
=&-\io J(\p_*^I q)\p_*^I\left(\frac{JR'(q)}{\rho_0}\p_t q\right)+\io J\p_*^I(J^{-1}\bp\eta^i)(J^{-1}\bp\eta_i)\p_*^I(\diva v)\dy+\cdots\\
=&-\frac{1}{2}\ddt\io\frac{J^2R'(q)}{\rho_0}|\p_*^I q|^2\dy+(-K_{11})+\cdots
\end{align*} Then using $Q=q+\frac12|J^{-1}\bp\eta|^2$, we also get the control of the total pressure $Q$.

\subsection{Modified Alinhac good unknowns}\label{stat4}

Before analyzing the boundary integral $IB$, we have to figure out the precise expressions of the Alinhac good unknowns $\VV,\QQ$ which can be derived by analyzing $\p_*^I(\pa f)$ for $f=v_i$ and $Q$. We will repeatedly use \eqref{da} in the analysis of commutators. First, for any multi-index $I'$ with $\lee I'\ree=1$, we have with the notation $[T,f,g]:=T(fg)-T(f)g-fT(g)$
\begin{align*}
\p_*^I(\pa^i f)=&\pa^i(\p_*^I f)+(\p_*^I A^{li})\,\p_l f+[\p_*^I,A^{li},\p_l f]\\
\overset{\eqref{da}}{=}&\pa^i(\p_*^I f)-\p_*^{I-I'}(A^{lr}\,\p_*^{I'}\p_m\eta_r\,A^{mi})\,\p_l f+[\p_*^I,A^{li},\p_l f]\\
=&\pa^i(\underbrace{\p_*^I f-\p_*^I\eta_rA^{lr}\p_l f}_{=:\p_*^I f-\p_*^I\eta\cdot\pa f})+\p_*^I\eta_r\,\pa^i(\pa^r f)-([\p_*^{I-I'},A^{lr}A^{mi}]\p_*^{I'}\p_m\eta_r)\p_l f+[\p_*^I,A^{li},\p_l f].
\end{align*}

Under the setting of standard Sobolev spaces, the term $\p_*^I f-\p_*^I\eta\cdot\pa f$ is already the standard Alinhac good unknown of $f$ (with respect to $\p_*^I$). 
See also \cite{MRgood2017,WangXingood,GuWang2016LWP,LuoZhang2020CWW,Zhang2020CRMHD,GL2021rel}. However, under the setting of anisotropic Sobolev spaces, we still need to analyze the commutators $-([\p_*^{I-I'},A^{lr}A^{mi}]\p_*^{I'}\p_m\eta_r)\p_l f$ and $[\p_*^I,A^{li},\p_l f]$ whose $L^2(\Omega)$ norms may not be directly controlled due to the anisotropy of $H_*^m$.

In particular, as long as $\p_*^I$ is not the purely non-weighted normal derivative $\p_3^4$, the commutator $[\p_*^I,A^{li},\p_l f]$ always contains the term $(\p_*^{I'}A^{li})(\p_*^{I-I'}\p_l f)$ whose $L^2(\Omega)$ norm cannot be controlled when $l=3$ due to the anisotropy of $H_*^m$. In fact, we should use different methods to analyze this term for $f=Q$ and $f=v_i$ respectively.
\begin{itemize}
\item When $f=v_i$, by using \eqref{da}, we can rewrite this term to be
\begin{align*}
(\p_*^{I'}A^{li})(\p_*^{I-I'}\p_l v_i)=&-(A^{lp}\,\p_*^{I'}\p_m\eta_p\,A^{mi})\p_*^{I-I'}\p_l v_i=-A^{li}\,\p_*^{I'}\p_m\eta_i\,A^{mp}\,\p_*^{I-I'}\p_l v_p\\
=&-\pa^i(\p_*^{I-I'}v_p~A^{mp}~\p_*^{I'}\p_m\eta_i)+\pa^i(A^{mp}\p_*^{I'}\p_m\eta_i)\,\p_*^{I-I'}v_p.
\end{align*} Then we can merge $-\p_*^{I-I'}v_p ~A^{mp}~\p_*^{I'}\p_m\eta_i$ into the good unknown of $v$, i.e., the covariant derivative part in \eqref{goodv}, and merge $\pa^i(A^{mp}\p_*^{I'}\p_m\eta_i)\p_*^{I-I'}v_p$ into the commutator part $C(v)$ in \eqref{goodv} because its $L^2$ norm can be directly controlled.

\item When $f=Q$, we invoke \eqref{da} and the MHD equation $-\pA Q=\rho_0\p_t v-\bp(J^{-1}\bp\eta)$ to get the following reduction. Here $\hat{A}=JA$.
\begin{align*}
J(\p_*^{I'}A^{li})(\p_*^{I-I'}\p_l Q)=&-(\hat{A}^{lp}\,\p_*^{I'}\p_m\eta_p\,A^{mi})(\p_*^{I-I'}\p_l Q)\\
=&-(\p_*^{I'}\p_m\eta_p~A^{mi})\,\p_*^{I-I'}\underbrace{(\hat{A}^{lp}\p_l Q)}_{=\nabla_{\hat{A}}^p Q}+(\p_*^{I-I'}\hat{A}^{lp})(\p_l Q)(\p_*^{I'}\p_m\eta_p~A^{mi})+[\p_*^{I-I'},\hat{A}^{lp},\p_l Q]\\
=&(\p_*^{I'}\p_m\eta_p~A^{mi})\p_*^{I-I'}\left(\rho_0\p_tv^p-\bp(J^{-1}\bp\eta^p)\right)\\
&+(\p_*^{I-I'}\hat{A}^{lp})(\p_l Q)(\p_*^{I'}\p_m\eta_p~A^{mi})+[\p_*^{I-I'},\hat{A}^{lp},\p_l Q]
\end{align*}
\begin{rmk}
Note that $\p_t$ and $\bp$ are both tangential derivatives while $\pa Q$ always contains a normal derivative. Such substitution actually makes the order of the derivatives lower with the help of the anisotropy of $H_*^m$.
\end{rmk}

The last term above is directly controlled. Since $\lee I-I'\ree=7$, we have
\begin{equation}
\|(\p_*^{I'}\p_m\eta_p~A^{mi})\p_*^{I-I'}(\rho_0\p_tv^p)\|_0\lesssim\|\p_*^{I'}\p_m\eta_p~A^{mi}\|_{L^{\infty}}\|\p_*^{I-I'}(\rho_0\p_tv^p)\|_0\lesssim P(\|\eta\|_{7,*})\|\rho_0\|_{7,*}\|v\|_{8,*}.
\end{equation}

For the term $-\p_*^{I'}\p_m\eta_p~A^{mi}~\p_*^{I-I'}\left(\bp(J^{-1}\bp\eta_p)\right)$, we need to use $b_0^3|_{\Gamma}=0$ to produce a weight function to make $b_0^3\p_3$ become a weighted normal derivative. By the fundamental theorem of calculus, we know (suppose $y_3>0$ without loss of generality)
\[
|b_0^3(t,y_3)|_{L^{\infty}(\T^2)}=\left|0+\int_1^{y_3}\p_3b_0^3(t,\zeta_3)d\zeta_3\right|_{L^{\infty}(\T^2)}\leq(1-y_3)\|\p_3b_0\|_{L^{\infty}}\lesssim \sigma(y_3)\|\p_3b_0\|_{L^{\infty}},
\]and thus
\begin{equation}\label{weight1}
\begin{aligned}
&\left\|(\p_*^{I'}\p_m\eta_p~A^{mi})\,\p_*^{I-I'}\left(\bp(J^{-1}\bp\eta_p\right)\right\|_0\\
\lesssim &P(\|\eta\|_{7,*})\left(\|b_0\|_{7,*}\|J^{-1}\bp\eta\|_{8,*}+\|\p_3b_0\|_{L^{\infty}}\|(\sigma\p_3)\p_*^{I-I'}(J^{-1}\bp\eta)\|_0\right)\\
\lesssim &P(\|\eta\|_{7,*})\left(\|b_0\|_{7,*}\|J^{-1}\bp\eta\|_{8,*}\right).
\end{aligned}
\end{equation}
In addition, the term $(\p_*^{I-I'}\hat{A}^{lp})(\p_l Q)(\p_*^{I'}\p_m\eta_pA^{mi})$ can be directly controlled when $l=3$ since $\hat{A}^{3p}$ consists of $(\TP\eta)(\TP\eta)$ (cf. \eqref{A}). When $l=1,2$, one should again invoke \eqref{da} to compute the highest order term and use $\TP Q|_{\Gamma}=0$ to produce a weight function as in \eqref{weight1}.
\begin{rmk}
From \eqref{weight1}, the definition of $H_*^m$ and the fact $\sigma|_{\Gamma}=0$, the weighted derivative $(\sigma\p_3)$ plays a similar role as a tangential derivative. In fact, one should consider the weighted derivative $(\sigma\p_3)$ as a tangential derivative throughout this manuscript.
\end{rmk}
\end{itemize}

There are three other terms which need further analysis:
\begin{itemize}
\setlength{\itemsep}{0pt}
\setlength{\parsep}{0pt}
\setlength{\parskip}{0pt}
\item $e_1:=-\p_*^{I-I'}(A^{lr}A^{mi})~(\p_*^{I'}\p_m\eta_r\,\p_l f)$. When $\p_*^{I-I'}$ does not contain time derivative, the term $\p_*^{I-I'}(A^{lr}A^{mi})$ cannot be controlled since both $A^{1i}$ and $A^{2i}$ contain $\p_3\eta$.

\item $e_2:=-\p_*^{I'}(A^{lr}A^{mi})\,(\p_*^{I-I'}\p_m\eta_r\,\p_l f)$. When $\p_*^{I-I'}$ does not contain time derivative, the term $\p_*^{I-I'}\p_m\eta_r$ cannot be controlled when $m=3$ since $\p_*^{I-I'}\p_3\eta$ should be controlled by $\|\eta\|_{9,*}$.

\item $e_3:=(\p_*^{I-I'}A^{li})\,\,(\p_*^{I'}\p_l f)$. When $\p_*^{I-I'}$ does not contain time derivative, the term $\p_*^{I-I'}A^{li}$ cannot be controlled when $l=1,2$ since $A^{1i}$ and $A^{2i}$ contains $\p_3\eta$.
\begin{rmk}
Since $\p_t \eta$ (resp. $\p_t A$) has the same spatial regularity as $\eta$ (resp. $A$), the $L^2(\Omega)$-norms of $e_1,e_2,e_3$ can be directly controlled when $\p_*^{I-I'}$ contains at least one time derivative.
\end{rmk}

\item When $\p_*^I$ contains the weighted normal derivative $(\sigma\p_3)$, we need to analyze the extra terms which are produced when $\p_3$ falls on $\sigma(y_3)$. This appears when we commute $\bp$ or $\pa$ with $\p_*^I$.
\end{itemize}

We note that these terms can be controlled by similar arguments as in the analysis of $(\p_*^{I'}A^{li})(\p_*^{I-I'}\p_l f)$. In other words, the following three techniques are enough for us to control the remaining terms.
\begin{itemize}
\setlength{\itemsep}{0pt}
\setlength{\parsep}{0pt}
\setlength{\parskip}{0pt}
\item Modify the definition of Alinhac good unknowns by rewriting the higher order terms to be a covariant derivative plus $L^2(\Omega)$-bounded terms.
\item Produce a weight function by using $b_0^3|_{\Gamma}=0$ and $\TP Q|_{\Gamma}=0$ in order to replace one $\p_3$ by $(\sigma\p_3)$.
\item Replace $\pA Q$ by $-\rho_0\p_t v+\bp(J^{-1}\bp\eta)$ in order to make the order of the derivatives lower thanks to the anisotropy of $H_*^m$.
\end{itemize}See Section \ref{ts8AGU} for detailed derivation of the modified Alinhac good unknowns and Section \ref{normal8} for the analysis of weighted derivatives. Therefore, we can write
\begin{align}
\label{goodqq} \QQ=\p_*^I Q-\p_*^I\eta\cdot\pa Q+\Delta_Q,\\
\label{goodvv} \VV_i=\p_*^I v_i-\p_*^I\eta\cdot\pa v_i+(\Delta_v)_i,
\end{align}where $\|\Delta_f\|_{1,*}\lesssim P(\EE(t))$ and the properties \eqref{goodq}-\eqref{goodv} still hold.

\subsection{Boundary estimates and necessity of anisotropy}\label{stat5}
Invoking \eqref{goodqq}-\eqref{goodvv}, we have
\begin{equation}\label{bdry1}
\begin{aligned}
IB=&-\ig JA^{3i}N_3\QQ\VV_i\dyy\\
=&\underbrace{-\ig JA^{3i}N_3(\p_*^I Q)\VV_i\dyy}_{=:IB_0}+\underbrace{\ig JA^{3i}N_3(\p_*^I\eta\cdot\pa Q)\VV_i\dyy}_{=:IB_1}+\cdots,
\end{aligned}
\end{equation} modulo the terms involving $\Delta_Q$ and $\Delta_v$ which can be controlled either by using trace lemma for anisotropic Sobolev space or using the trick of divergence theroem as in \eqref{IBIB7}. The detailed analysis can be found in Section \ref{sect n4bdry}, \ref{sect t8bdry} and \ref{sect tnbdry}.

\paragraph*{Regularity of the free surface and standard cancellation structure}

First, $IB_1$ in \eqref{bdry1} gives the boundary energy and a cancellation structure enjoyed by the standard Alinhac good unknown arguments as in \cite{MRgood2017,WangXingood,GuWang2016LWP,LuoZhang2020CWW,Zhang2020CRMHD,GL2021rel}. In specific, since $\TP_1 Q=\TP_2 Q=0$ on $\Gamma$, we have
\begin{equation}\label{IBIB1}
\begin{aligned}
IB_1=&\ig J\frac{\p Q}{\p N} \p_*^I\eta_k\,A^{3k}\,A^{3i}\,(\p_*^I \p_t\eta_i-\p_*^I \eta_r\,A^{lr}\,\p_l v_i+\Delta_{v_i})\dyy\\
=&-\frac12\ddt\ig \left(-J\frac{\p Q}{\p N}\right)\left|A^{3i}\p_*^I\eta_i\right|^2\dyy+\frac12\ig \p_t\left(-J\frac{\p Q}{\p N}\right)\left|A^{3i}\p_*^I\eta_i\right|^2\dyy\\
&-\ig J\frac{\p Q}{\p N}(A^{3k}\p_*^I\eta_k)\,\p_tA^{3i}\,\p_*^I\eta_i\dyy -\ig J\frac{\p Q}{\p N}(A^{3k}\p_*^I\eta_k)A^{3i}\p_*^I\eta_r\,A^{lr}\p_l v_i\dyy\\
&+\ig J\frac{\p Q}{\p N} \p_*^I\eta_k\,A^{3k}\,A^{3i}\,\Delta_{v_i}\dyy.
\end{aligned}
\end{equation}Invoking the Rayleigh-Taylor sign condition \eqref{sign2}, we get the boundary energy $\left|A^{3i}\p_*^I\eta_i\right|_0^2$ which exactly controls the second fundamental form of the free surface. The second term can be directly controlled thanks to the boundary energy. Then plugging $\p_tA^{3i}=-A^{3r}~\p_l v_r~A^{li}$ into the third term yields the cancellation with the fourth term. The last term can be controlled directly by using the boundary energy and trace lemma for anisotropic Sobolev space.
\begin{rmk}
The cancellation structure above, enjoyed by the Alinhac good unknown, relies on the fact that $A^{3i}(\p_*^I\p_t\eta_i-\p_*^I\eta_r~A^{lr}\p_l\p_t\eta_i)=\p_t(A^{3i}\p_*^I\p_t\eta_i)$ which can be proved by using \eqref{da} with $D=\p_t$. This identity will be repeatedly used to derive similar cancellation structure in the boundary estimates.
\end{rmk}

\paragraph*{Reduction of the normal derivatives and the advantage of the anisotropy}

When $\p_*^I$ contains normal derivative, $\p_*^I Q$ no longer vanishes on $\Gamma$. In this case we write $\p_*^I=\p_*^{I-e_3}\p_3$ where the multi-index $e_3$ is defined by $(i_0,i_1,i_2,i_3,i_4)=(0,0,0,1,0)$ and $\langle I-e_3\rangle=6$. We shall analyze
\begin{equation}\label{IBIB0}
IB_0=\ig N_3 J(\p_*^{I-e_3}\p_3Q)(A^{3i}\,\p_*^{I-e_3}\p_3 v_i)\dyy+\ig N_3 J(\p_*^{I-e_3}\p_3Q)(\p_*^{I-e_3}\p_3\eta_p\,A^{lp}\p_l v_i)\dyy=:IB_{01}+IB_{02}.
\end{equation}

First, for $IB_{01}$, we invoke the third equation in \eqref{CMHDL} to replace the normal derivative in $A^{3i}\p_3 v_i$ by tangential derivative
\begin{align}
\label{relationv} A^{3i}\,\,\p_*^{I-e_3}\p_3v_i=&\p_*^{I-e_3}(A^{3i}\p_3v_i)-[\p_*^I,A^{3i}]\p_3 v_i=-\p_*^{I-e_3}\left(\frac{JR'(q)}{\rho_0}\p_t q\right)-\sum_{L=1}^2\p_*^{I-e_3}(A^{Li}\TP_Lv_i)-[\p_*^{I-e_3},A^{3i}]\p_3 v_i.
\end{align}Note that we replace a normal derivative by a tangential derivative in the first term on the right side. The highest order terms in the last commutator are $\p_*^{I-e_3}A^{3i}\p_3 v_i$ and $\p_*^{I'}A^{3i}\p_*^{I-I'}v_i$ with $\langle I\rangle=1$. Since $A^{3i}$ consists of $\TP\eta\times\TP\eta$, we know the highest order of derivatives in either of these two terms is 7 (in the sense of anisotropy, that is, $\langle I\rangle=i_0+i_1+i_2+2i_3+i_4)$.

The most difficult term is $\p_*^{I-e_3}A^{Li}=-A^{Lp}(\p_*^{I-e_3}\p_m\eta_p)A^{mi}=-A^{Lp}(\p_*^{I-e_3}\p_3\eta_p)A^{3i}-\sum\limits_{M=1}^2A^{Lp}(\p_*^{I-e_3}\TP_M\eta_p)A^{Mi}$, in which the contribution of $-A^{Lp}(\p_*^{I-e_3}\p_3\eta_p)A^{3i}$ in $IB_{01}$ exactly cancels with the contribution of $l=1,2$ in $IB_{02}$. The highest order term in $\p_*^{I-e_3}\p_3\eta_p\,A^{lp}\p_l v_i$ corresponding to $l=3$ in $IB_{02}$ is actually $\p_3\eta_p\,\p_*^{I-e_3}A^{3p}~\p_l v_i=\p_3\eta_p\,\p_*^{I-e_3}(J^{-1}\TP_1\eta\times\TP_2\eta)_p~\p_l v_i$ thanks to the identity $A^{3p}\p_3\eta_p=1$.

Next we replace $\p_3 Q$ by tangential derivative of $v$ and $\bp\eta$. Since $A^{3i}\p_3\eta_i=1$, we have
\begin{equation}\label{relationq}
\begin{aligned}
J(\p_*^{I-e_3}\p_3Q)=&\p_3\eta_i\,\hat{A}^{3i}\,(\p_*^{I-e_3}\p_3Q)=\p_3\eta_i\,\p_*^{I-e_3}(\hat{A}^{3i}\p_3 Q)-\p_3\eta_i[\p_*^{I-e_3},\hat{A}^{3i}]\p_3 Q\\
=&\p_3\eta_i\,\p_*^{I-e_3}(\underbrace{\hat{A}^{li}\p_l Q}_{=\pA^i Q})-\sum_{L=1}^2\p_3\eta_i\,\p_*^{I-e_3}(\hat{A}^{Li}\TP_L Q)-\p_3\eta_i[\p_*^{I-e_3},\hat{A}^{3i}]\p_3 Q\\
=&\p_3\eta_i\,\p_*^{I-e_3}\left(-\rho_0\p_t v^i-\bp(J^{-1}\bp\eta^i)\right)-\sum_{L=1}^2\p_3\eta_i\,\p_*^{I-e_3}(\hat{A}^{Li}\TP_L Q)-\p_3\eta_i[\p_*^{I-e_3},\hat{A}^{3i}]\p_3 Q
\end{aligned}
\end{equation}

Note that $\TP_L Q|_{\Gamma}=0$ for $L=1,2$ eliminates the highest order term $\p_3\eta_i\,(\p_*^{I-e_3}\hat{A}^{Li})\,\TP_L Q$. And $b_0^3|_{\Gamma}=0$ implies that $\bp|_{\Gamma}=b_0^1\TP_1+b_0^2\TP_2$ is a tangential derivative on $\Gamma$. The last commutator can be controlled in the same way as \eqref{relationv}. Combining \eqref{IBIB0}-\eqref{relationq}, the highest order terms in $IB_0$ can all be written as the following form
\[
\ig N_3(\p_*^{I-e_3}\dd f)(\p_*^{I-e_3} \dd g)h\dyy,
\]where $\dd$ can be $(b_0\cdot\TP),\TP,\p_t$(tangential), and $f,g$ can be $\eta,v,q,J^{-1}\bp\eta$, and $h$ consists of the the terms containing at most first-order derivative of $\eta$ and $v$. To control such boundary integral, we first rewrite it to the interior thanks to the divergence theorem in $y$-coordinates, and then integrate $\dd$ by parts
\begin{equation}\label{IBIB7}
\begin{aligned}
&\ig N_3(\p_*^{I-e_3}\dd f)(\p_*^{I-e_3} \dd g)h\dyy\\
=&\io(\p_3\p_*^{I-e_3}\dd f)(\p_*^{I-e_3} \dd g)h\dy+\io(\p_*^{I-e_3}\dd f)(\p_3\p_*^{I-e_3} \dd g)h\dy+\io(\p_*^{I-e_3}\dd f)(\p_*^{I-e_3} \dd g)\p_3h\dy\\
\overset{\dd}{=}&-\io(\p_3\p_*^{I-e_3}f)(\p_*^{I-e_3} \dd^2 g)h\dy+\io(\p_3\p_*^{I-e_3}f)(\p_*^{I-e_3} \dd g) \dd h\dy\\
&-\io(\p_*^{I-e_3}\dd^2 f)(\p_3\p_*^{I-e_3} g)h\dy+\io(\p_*^{I-e_3}\dd f)(\p_3\p_*^{I-e_3} g) \dd h\dy +\io(\p_*^{I-e_3}\dd f)(\p_*^{I-e_3} \dd g)\p_3h\dy\\
\lesssim&\|f\|_{8,*}\|g\|_{8,*}\|h\|_{3},
\end{aligned}
\end{equation}where the anisotropy of the function space $H_*^8$ is crucial in the last step because $\lee I-e_3\ree=6$ allows us to have two more tangential derivatives $\dd^2$. When $\dd$ in \eqref{IBIB7} is $\p_t$, this step should be done under time integral. See \eqref{n3bdry22} for example.

The analysis of $IB_0$ above also shows the advantage of using anisotropic Sobolev space as pointed out as an important conclusion in the survey article \cite{ChenSX2013} by Chen who first introduced the anisotropic Sobolev spaces in \cite{ChenSX1982}
\begin{quote}``For the nonlinear hyperbolic system with characteristic boundary conditions, the growth of one normal derivative on the boundary should be compensated by the decrease in regularity of two tangential derivatives. This is one of the advantages of the anisotropic Sobolev space that the standard Sobolev space fails to carry."
\end{quote}
In specific, if we start with the estimates of $\p_3^4$, then by \eqref{IBIB7} we need the control of $\p_3^3\dd^2$ where $\dd$ is a tangential derivative. To control the latter one, we need the control of $\p_3^2\dd^4$ again due to \eqref{IBIB7}. Repeatedly, we finally need to derive the estimates of $\dd^8$. In addition, the weighted derivative $(\sigma\p_3)$ is necessary in the interior estimates, e.g., in \eqref{weight1}. On the other hand, we also need the control of 4 normal derivatives in order to close the energy estimates of 8 tangential derivatives. So we find that the anisotropic Sobolev space exactly meets all of these requirements in our mechanism of reducing normal derivatives on the boundary.

Finally, the contribution of $\Delta_Q$ and $\Delta_v$ in $IB$ can be controlled by using the boundary energy $|A^{3i}\p_*^I\eta_i|_0$ together with either the trace lemma for anisotropic Sobolev spaces (cf. Lemma \ref{trace}) or similar technique as in \eqref{IBIB7}. Hence, the control of boundary integral $IB$ is finished.

\subsection{Strategy to prove the existence results}\label{statLWP}

It is natural to ask if a local existence result (without loss of regularity) can be proved in $H_*^m(\Omega)$ by using the energy $\EE_m(t)$. However, we find it difficult to find a straightforward proof as in the case of compressible Euler equations \cite{LuoZhang2020CWW,GL2021rel}, elastodynamics \cite{Zhang2021elasto} or incompressible MHD \cite{GuWang2016LWP}. Briefly speaking, this is due to the fact that the magnetic field is involved in the pressure part for compressible MHD. \textbf{The simultaenous appearance of magnetic field and compressibility leads to a mismatched term in the linearized equation and causes a loss of derivative in Picard iteration. Such difficulty never appears in either case of Euler equations, elastodynamics, or incompressible MHD.}

Alternatively, we may try to prove the local existence in anisotropic Sobolev spaces by using the existing local existence result obtained by Nash-Moser iteration. The idea is to approximate the given data, say $U_0$, by a sequence of ``sufficiently nice" data, say $\{U_0^{(n)}\}$, in some Sobolev space. Once we can do this\footnote{Indeed, as stated before Theorem \ref{CMHDLWPm}, it is still unknown how to approximate the given data by a sequence of smooth data satisfying the compatibility conditions up to infinite order. This may be postpone to a future work.}, the solutions corresponding to the ``sufficiently nice" data, say $\{U^{(n)}(t)\}$, may have convergence in some anisotropic Sobolev space by using the continuous dependence on data. The limit, say $U(t)$, is expected to be solution corresponding to the given data $U_0$.

However, the lifespan of $U^{(n)}(t)$, say $T^{(n)}$, may depend on $n$, so we need to continue the solution at $t=T^{(n)}$ and use the energy bounds in Theorem \ref{CMHDEE} to obtain a positive lower bound for the lifespans of $\{U^{(n)}(t)\}$. The continuation process requires a local existence result where the solution and the data lie in the same space. Since \cite{TW2020MHDLWP} only shows the existence theorem for the data in anisotropic Sobolev spaces and has a loss of regularity, we have to improve the result in \cite{TW2020MHDLWP} such that a unique $C^{\infty}$ solution exists if the initial data is $C^{\infty}$ and satisfies the compatibility conditions up to infinite order. Such improvement can be achieved, because there are two extra error terms $e_n'''$ and $D_{n+\frac12}\delta\Psi_n$ (cf. \cite[(4.26)-(4.27)]{TW2020MHDLWP}) that can be avoided in Lagrangian coordinates. See Section \ref{sect Nash} for detailed explanations for the Nash-Moser iteration, Section \ref{sect continue} for the proof of continuation criterion, Section \ref{sect limit} for the proof of Theorem \ref{CMHDLWPm} via a limit process.

\section{Preliminary lemmas}\label{sect lemma}

\subsection{Some geometric identities}

We record the explicit form of the matrix $A$ which will be repeatedly used.
\begin{align}
A= J^{-1}\begin{pmatrix}
\TP_2 \eta^2\p_3\eta^3-\p_3\eta^2\TP_2\eta^3\q \p_3 \eta^1\TP_2\eta^3-\TP_2\eta^1\p_3\eta^3\q \TP_2 \eta^1\p_3\eta^2-\p_3\eta^1\TP_2\eta^2 \\
\p_3 \eta^2\TP_1\eta^3-\TP_1\eta^2\p_3\eta^3\q \TP_1 \eta^1\p_3\eta^3-\p_3\eta^1\TP_1\eta^3\q \TP_1 \eta^1\TP_1\eta^2-\TP_1\eta^1\p_3\eta^2\\
\TP_1 \eta^2\TP_2\eta^3-\TP_2\eta^2\TP_1\eta^3\q \TP_2 \eta^1\TP_1\eta^3-\TP_1\eta^1\TP_2\eta^3\q \TP_1 \eta^1\TP_2\eta^2-\TP_2\eta^1\TP_1\eta^2
\end{pmatrix} \label{a}
\end{align}
Moreover, since $\hat{A}=JA$, and in view of \eqref{a}, we can write
\begin{align}
\hat{A}^{1i}=\epsilon^{ijk}\TP_2\eta_j\p_3\eta_k,~~ \hat{A}^{2i}=-\epsilon^{ijk}\TP_1\eta_j\p_3\eta_k,~~
\hat{A}^{3i}=\epsilon^{ijk}\TP_1\eta_j\TP_2\eta_k.\label{A}
\end{align}
Here, $\epsilon^{ijk}$ is the sign of the 3-permutation $(ijk)\in S_3$. We will repeatedly use that fact that $\hat{A}^{1,\cdot},\hat{A}^{2,\cdot}$ consist of the linear combination of $\pm\TP\eta\times\p_3\eta$ and $\hat{A}^{3\cdot}$ consists of $\TP\eta\times\TP\eta$.

We also record the following identity: Suppose $D$ is the derivative $\p$ or $\p_t$, then
\begin{equation}\label{Da}
DA^{li}=-A^{lr}~\p_kD\eta_r~A^{ki}.
\end{equation}

\subsection{Anisotropic Sobolev space}\label{anisotropic}

We list two preliminary lemmas on the basic properties of anisotropic Sobolev space.

\begin{lem}[Trace lemma for anisotropic Sobolev space]\label{trace}
Let $m\geq 1,~m\in\N^*$, then we have the following trace lemma for the anisotropic Sobolev space.
\begin{enumerate}
\item If $f\in H_*^{m+1}(\Omega)$, then its trace $f|_{\Gamma}$ belongs to $H^{m}(\Gamma)$ and satisfies
\[
|f|_m\lesssim\|f\|_{H_*^{m+1}(\Omega)}.
\]

\item There exists a linear continuous operator $\mathfrak{R}_{T}:H^{m}(\Gamma)\to H_*^{m+1}(\Omega)$ such that $(\mathfrak{R}_{T} g)|_{\Gamma}=g$ and \[
\|\mathfrak{R}_{T} g\|_{H_*^{m+1}(\Omega)}\lesssim|g|_m.
\]
\end{enumerate}
\begin{proof} See Ohno-Shizuta-Yanagisawa \cite[Theorem 1]{anisotropictrace}.
\end{proof}
\end{lem}
\begin{rmk}
The condition $m\geq 1$ is necessary and analogous result may not hold when $m=0$. One can see the importance of $m\geq 1$ from \eqref{IBIB7}, as a special case, where we need to integrate one tangential derivative by part and thus $m\geq 1$ is necessary.
\end{rmk}

\begin{lem}[Sobolev embedding lemma for anisotropic Sobolev space]\label{embedding}
We have the following inequalities
\begin{align*}
H^m(\Omega)\hookrightarrow H_*^m(\Omega)\hookrightarrow& H^{\lfloor m/2\rfloor}(\Omega),~~\forall m\in\N^*\\
\|u\|_{L^{\infty}}\lesssim\|u\|_{H_*^3(\Omega)},~~&\|u\|_{W^{1,\infty}}\lesssim\|u\|_{H_*^5(\Omega)},~~|u|_{W^{1,\infty}}\lesssim\|u\|_{H_*^{5}(\Omega)}.
\end{align*}
\begin{proof} See Trakhinin-Wang \cite[Lemma 3.3]{TW2020MHDLWP}. For the last inequality, using trace lemma, we have $|\TP u|_{L^{\infty}}\lesssim|\TP u|_{1.5}\lesssim\|\TP u\|_2\leq\|u\|_{H_*^5(\Omega)}.$
\end{proof}
\end{lem}

\section{Control of purely non-weighted normal derivatives}\label{sect n4}

In this section, we prove the following estimates by the standard Alinhac good unknown argument.
\begin{prop}\label{prop n4}
The following energy inequality holds
\begin{equation}\label{n4n4}
\|\p_3^4 v\|_0^2+\left\|\p_3^4\left(J^{-1}\bp\eta\right)\right\|_0^2+\|\p_3^4 q\|_0^2+\frac{c_0}{4}\left|A^{3i}\p_3^4\eta_i\right|_0^2\bigg|_{t=T}\lesssim \PP_0+ P(\EE(T))\int_0^T P(\EE(t))\dt.
\end{equation}
\end{prop}

\subsection{Evolution equation of Alinhac good unknowns}\label{n4AGU}

We first compute the estimates of purely normal derivatives. When $\lee I\ree=8$, the purely non-weighted normal derivative should be $\p_*^I=\p_3^4$. First we introduce the following Alinhac good unknowns of $v$ and $Q$ with respect to $\p_3^4$
\begin{equation}\label{n4good}
\VV_i:=\p_3^4 v_i-\p_3^4\eta_p\,A^{lp}\,\p_l v_i,~~\QQ:=\p_3^4 Q-\p_3^4\eta_p\,A^{lp}\,\p_l Q.
\end{equation}
Then we have that for any function $f$
\begin{equation}\label{n4proof}
\begin{aligned}
\p_3^4(\pa^{i}f)&=\pa^i(\p_3^4 f)+(\p_3^4A^{li})\p_l f+[\p_3^4,A^{li},\p_l f] \\
&=\pa^{i}(\p_3^4 f)-\p_3^3(A^{lp}\,\p_3\p_{m}\eta_{p}\,A^{mi})\p_l f+[\p_3^4,A^{li},\p_{l} f] \\
&=\pa^{i}(\underbrace{\p_3^4 f-\p_3^4\eta_{p}\,A^{lp}\,\p_l f}_{\text{good unknowns}})+\underbrace{\p_3^4 \eta_{p}\pa^{i}(\pa^{p} f)-([\p_3^3,A^{lp}A^{mi}]\p_3\p_{m}\eta_{p})\p_l f+[\p_3^4,A^{li},\p_{l} f] }_{=:C^{i}(f)},
\end{aligned}
\end{equation} and thus
\begin{equation}\label{n4grad}
\nabla_A \cdot\VV=\p_3^4(\diva v)-C^i(v_i),~~~\nabla_A \QQ=\p_3^4(\nabla_A Q)-C(Q),
\end{equation}where the commutator satisfies the estimate
\begin{equation}\label{n4error}
\|C(f)\|_4\lesssim P(\|\eta\|_4)\|f\|_4.
\end{equation}

Now taking $\p_3^4$ in the second equation of \eqref{CMHDL} and invoking \eqref{n4good} and \eqref{n4grad}, we get the evolution equation of the Alinhac good unknowns
\begin{equation}\label{n4eq}
R\p_t\VV-J^{-1}\bp\p_3^4\left(J^{-1}\bp\eta\right)+\nabla_A \QQ=\underbrace{\left[R,\p_3^4\right]\p_t v+\left[\p_3^4,J^{-1}\bp\right] b-C(Q)-R\p_t(\p_3^4\eta\cdot\nabla_A v)}_{=:\FF}.
\end{equation}

Taking $L^2(\Omega)$-inner product of \eqref{n4eq} and $J\VV$ and using $\rho_0=RJ$, we get the energy identity
\begin{equation}\label{n4energy0}
\frac{1}{2}\ddt\io\rho_0\left|\VV\right|^2\dy=\io\bp\p_3^4(J^{-1}\bp\eta)\cdot\VV-\io(\nabla_{\hat{A}}\QQ)\cdot\VV+\io J\FF\cdot\VV.
\end{equation}

\subsection{Interior estimates}

The third integral on the RHS of \eqref{n4energy0} can be directly controlled
\begin{equation}\label{n43}
\io J\FF\cdot\VV\lesssim\|J\FF\|_0\|\VV\|_0\lesssim P(\|\rho_0\|_4,\|b_0\|_4,\|\eta\|_4,\|J^{-1}\bp\eta\|_4,\|Q\|_4,\|v\|_4,\|\p_t v\|_3)\|\VV\|_0.
\end{equation}

The first integral on the RHS of \eqref{n4energy0} gives the energy of magnetic field $b=J^{-1}\bp\eta$ after integrating $\bp$ by parts. Note that $b_0^3|_{\Gamma}=0$ and $\dive b_0=0$, there will be no boundary integral. In specific, we have
\begin{equation}\label{n410}
\begin{aligned}
&\io\bp\p_3^4(J^{-1}\bp\eta)\cdot\VV\dy=-\io\p_3^4(J^{-1}\bp\eta)\cdot\bp\VV\dy\\
=&-\io\p_3^4(J^{-1}\bp\eta)\cdot\bp\p_3^4 v\dy+\underbrace{\io\p_3^4(J^{-1}\bp\eta)\cdot\bp(\p_3^4\eta\cdot\pa v)\dy}_{=:L_1}\\
=&-\io J\p_3^4(J^{-1}\bp\eta)\cdot \left(J^{-1}\bp\p_3^4 \p_t\eta\right)\dy+L_1\\
=&-\io J\p_3^4(J^{-1}\bp\eta)\cdot \p_3^4 \p_t(J^{-1}\bp\eta)\dy\underbrace{-\io J\p_3^4(J^{-1}\bp\eta)\cdot \left[J^{-1}\bp,\p_3^4\p_t \right]\eta\dy}_{K_1}+L_1\\
=&-\frac12\ddt\io J\left|\p_3^4(J^{-1}\bp\eta)\right|^2\dy+\frac12\io\p_t J\left|\p_3^4(J^{-1}\bp\eta)\right|^2\dy+K_1+L_1.
\end{aligned}
\end{equation}

The term $L_1$ can be directly controlled
\begin{equation}\label{n4L1}
L_1\lesssim P\left(\|\bp\eta\|_4,\|\eta\|_4,\|b_0\|_4,\|v\|_4\right).
\end{equation}

The term $K_1$ produces a higher order term when $\p_3^4\p_t$ falls on $J^{-1}$. We invoke $\p_t J=J\diva v$ to get
\begin{equation}\label{n4K10}
\begin{aligned}
&-\left[J^{-1}\bp,\p_3^4 \p_t\right]\eta\\
=&\p_3^4\p_t(J^{-1})\,\bp\eta+\sum_{N=0}^3\p_3^N\p_t(J^{-1})\,\p_3^{4-N}b_0\cdot\p\eta+\sum_{M=0}^{3}\p_t\left(\p_3^M(J^{-1}b_0^l)\,\p_l\p_3^{4-M}\eta\right)\\
=&-J^{-1}\p_3^4(\diva v)\,\bp\eta\\
&+\underbrace{\left([\p_3^4, J^{-1}]\diva v\right)\bp\eta+\sum_{N=0}^3\p_3^N\p_t(J^{-1})(\p_3^{4-N}b_0^l)(\p_l\eta)+\sum_{M=0}^{3}\p_t\left(\p_3^M(J^{-1}b_0^l)\p_l\p_3^{4-M}\eta\right)}_{KL_1},
\end{aligned}
\end{equation} and thus
\begin{equation}\label{n4K1}
\begin{aligned}
K_1=&\underbrace{-\io  J\p_3^4(J^{-1}\bp\eta)\cdot\left(J^{-1}\bp\eta\right)\p_3^4(\diva v)\dy}_{K_{11}}+\io J\p_3^4(J^{-1}\bp\eta)\cdot(KL_1)\\
\lesssim& K_{11}+\|J\|_{L^{\infty}}\|J^{-1}\bp\eta\|_4\|KL_1\|_0\\
\lesssim& K_{11}+P\left(\|\bp\eta\|_4,\|\eta\|_4,\|b_0\|_4\right).
\end{aligned}
\end{equation}

Summarizing \eqref{n410}-\eqref{n4K1}, we get the following estimates
\begin{equation}\label{n41}
\io\bp\p_3^4(J^{-1}\bp\eta)\cdot\VV\dy\lesssim-\frac12\ddt\io J\left|\p_3^4(J^{-1}\bp\eta)\right|^2\dy+K_{11}+P\left(\|\bp\eta\|_4,\|\eta\|_4,\|b_0\|_4,\|v\|_4\right).
\end{equation}
We note that the term $K_{11}$ cannot be directly controlled, but will be cancelled by another term produced by $-\io(\nabla_{\hat{A}} \QQ)\cdot\VV$.

Next we analyze the second integral on the RHS of \eqref{n4energy0}. Integraing by parts and invoking Piola's identity $\p_l \hat{A}^{li}=0$, we get
\begin{equation}\label{n420}
\begin{aligned}
-\io(\nabla_{\hat{A}}\QQ)\cdot\VV\dy=\io J\QQ(\nabla_A\cdot\VV)\dy-\ig J\QQ A^{li}N_l\VV_i\dyy=:I+IB.
\end{aligned}
\end{equation}

Plugging \eqref{n4good} and \eqref{n4grad} as well as $Q=q+\frac12|b|^2$ into $I$, we get
\begin{equation}\label{I0}
\begin{aligned}
I=&\io J\p_3^4 q\,\p_3^4(\diva v)\dy+\io J\p_3^4\left(\frac12\left|J^{-1}\bp\eta\right|^2\right)\p_3^4(\diva v)\dy\\
&-\io \p_3^4\eta_p\,\hat{A}^{lp}\p_lQ\,\p_3^4(\diva v)\dy-\io\p_3^4 Q\,C(v)\dy\\
=:&I_1+I_2+I_3+I_4.
\end{aligned}
\end{equation}

The term $I_4$ can be directly controlled by using \eqref{n4error}
\begin{equation}\label{I4}
I_4\lesssim\|Q\|_4\|C(v)\|_0\lesssim P(\|\eta\|_4)\|Q\|_4\|v\|_4.
\end{equation}

The term $I_1$ gives the energy of $q$ by invoking $\diva v=-\dfrac{\p_t R}{R}=-\dfrac{JR'(q)}{\rho_0}\p_t q$
\begin{equation}\label{I1}
\begin{aligned}
I_1=&-\io J\p_3^4 q\,\p_3^4\left(\frac{JR'(q)}{\rho_0}\p_t q\right)\dy=-\io\frac{J^2R'(q)}{\rho_0}\p_3^4 q\,\p_3^4\p_t q\dy-\io J\p_3^4 q\left(\left[\p_3^4,\frac{JR'(q)}{\rho_0}\right]\p_t q\right)\dy\\
=&-\frac{1}{2}\ddt\io\frac{J^2R'(q)}{\rho_0}\left|\p_3^4 q\right|^2\dy+\frac12\io\p_t\left(\frac{J^2R'(q)}{\rho_0}\right)\left|\p_3^4 q\right|^2\dy-\io J\p_3^4 q\left(\left[\p_3^4,\frac{JR'(q)}{\rho_0}\right]\p_t q\right)\dy\\
\lesssim&-\frac{1}{2}\ddt\io\frac{J^2R'(q)}{\rho_0}\left|\p_3^4 q\right|^2\dy+ P(\|q\|_{8,*},\|\rho_0\|_4,\|\eta\|_4).
\end{aligned}
\end{equation}

The term $I_2$ will produce another higher order term to cancel with $K_{11}$
\begin{equation}\label{I2}
\begin{aligned}
I_2=&\underbrace{\io J\p_3^4\left(J^{-1}\bp\eta\right) \cdot\left(J^{-1}\bp\eta\right)\p_3^4(\diva v)}_{\text{exactly cancel with }K_{11}}\dy\\
&+\io \sum_{N=1}^3\binom{4}{N}J\p_3^N\left(J^{-1}\bp\eta\right) \cdot\p_3^{4-N}\left(J^{-1}\bp\eta\right)\p_3^4(\diva v)\dy\\
=&-K_{11}-\io \sum_{N=1}^3\binom{4}{N}J\p_3^N\left(J^{-1}\bp\eta\right) \cdot\p_3^{4-N}\left(J^{-1}\bp\eta\right)\p_3^4\left(\frac{JR'(q)}{\rho_0}\p_tq\right)\dy\\
=&-K_{11}-\io \sum_{N=1}^3\binom{4}{N}\left(\frac{J^2R'(q)}{\rho_0}\right)\p_3^N\left(J^{-1}\bp\eta\right) \cdot\p_3^{4-N}\left(J^{-1}\bp\eta\right)\p_3^4\p_tq\dy\\
&+\io \sum_{N=1}^3\binom{4}{N}J\p_3^N\left(J^{-1}\bp\eta\right) \cdot\p_3^{4-N}\left(J^{-1}\bp\eta\right)\left(\left[\p_3^4,\frac{JR'(q)}{\rho_0}\right]\p_tq\right)\dy\\
=:&-K_{11}+I_{21}+I_{22}.
\end{aligned}
\end{equation}

We should control $I_{21}$ by integrating $\p_t$ by parts under time integral
\begin{equation}\label{I21}
\begin{aligned}
\int_0^TI_{21}\overset{\p_t}{=}&\int_0^T\io \sum_{N=1}^3\binom{4}{N}\p_t\left(\frac{J^2R'(q)}{\rho_0}\right)\p_3^N\left(J^{-1}\bp\eta\right) \cdot\p_3^{4-N}\left(J^{-1}\bp\eta\right)\p_3^4q\dy\\
&+\int_0^T\io \sum_{N=1}^3\binom{4}{N}\left(\frac{J^2R'(q)}{\rho_0}\right)\p_t\p_3^N\left(J^{-1}\bp\eta\right) \cdot\p_3^{4-N}\left(J^{-1}\bp\eta\right)\p_3^4q\dy\\
&-\io\sum_{N=1}^3\binom{4}{N}\left(\frac{J^2R'(q)}{\rho_0}\right)\p_3^N\left(J^{-1}\bp\eta\right) \cdot\p_3^{4-N}\left(J^{-1}\bp\eta\right)\p_3^4q\dy\bigg|^T_0\\
\lesssim&\int_0^T P(\|J^{-1}\bp\eta\|_{4},\|\p_t(J^{-1}\bp\eta)\|_{3},\|q\|_4)+\PP_0+\|J^{-1}\bp\eta\|_{3}^2\|\p_3^4q\|_0\\
\lesssim&\PP_0+\int_0^TP(\EE(t))\dt+\varepsilon\|\p_3^4 q\|_0^2+\int_0^T\|\p_t(J^{-1}\bp\eta)\|_{3}^2\leq\PP_0+\int_0^TP(\EE(t))\dt+\varepsilon\|\p_3^4 q\|_0^2.
\end{aligned}
\end{equation} Then $I_{22}$ can be directly controlled since at most three $\p_3$'s fall on $\p_t q$.
\begin{equation}\label{I22}
I_{22}\lesssim \|J^{-1}\bp\eta\|_3^2\|q\|_{7,*}.
\end{equation}

The term $I_3$ should also be controlled under time integral. We have
\begin{equation}\label{I3}
\begin{aligned}
\int_0^TI_3=&\int_0^T\io \frac{JR'(q)}{\rho_0}\p_3^4\eta_p\,\hat{A}^{lp}\p_l Q\,\p_3^4\p_t q\dy+\underbrace{\int_0^T\io \p_3^4\eta_p\,\hat{A}^{lp}\p_l Q\,\left[\p_3^4, \frac{JR'(q)}{\rho_0}\right]\p_t q}_{L_2}\dy\\
\overset{\p_t}{=}&-\int_0^T\io\p_t\left( \frac{JR'(q)}{\rho_0}\,\p_3^4\eta_p\,\hat{A}^{lp}\p_l Q\right)\p_3^4 q\dy+\io\frac{JR'(q)}{\rho_0}\p_3^4\eta_p\,\hat{A}^{lp}\p_l Q\,\p_3^4 q\dy\bigg|^T_0+L_2\\
\lesssim&\PP_0+\left\|\frac{JR'(q)}{\rho_0}\,A\,\p Q\right\|_{L^{\infty}}\|\p_3^4 q\|_0\|\p_3^4\eta\|_0+\int_0^TP\left(\|q\|_{8,*},\|\eta\|_4,\|v\|_4,\|\rho_0\|_4\right)\dt.\\
\lesssim&\PP_0+\int_0^T P(\EE(t))\dt+\left\|\frac{JR'(q)}{\rho_0}\,A\,\p Q\right\|_{L^{\infty}}\|\p_3^4 q\|_0\int_0^T\|\p_3^4 v(t)\|_0\dt\\
\lesssim&\PP_0+P(\EE(t))\int_0^T P(\EE(t))\dt,
\end{aligned}
\end{equation}where we use $\p^4\eta|_{t=0}=0$ in the last step. Summarizing \eqref{I4}-\eqref{I3} and choosing $\eps>0$ suitably small, we get the estimates of $I$ under time integral
\begin{equation}\label{I}
\int_0^TI\dt\lesssim-\frac{1}{2}\io\frac{J^2R'(q)}{\rho_0}\left|\p_3^4 q\right|^2\dy\bigg|^T_0+\PP_0+P(\EE(t))\int_0^T P(\EE(t))\dt.
\end{equation}

\subsection{Boundary estimates}\label{sect n4bdry}

To finish the estimates of purely non-weighted normal derivative, it remains to control the boundary integral $IB$ in \eqref{n420} which reads
\begin{equation}\label{IBB}
\begin{aligned}
-\ig J\QQ A^{li}N_l\VV_i\dyy=&-\ig \hat{A}^{3i}N_3\,\p_3^4Q\,\VV_i\dyy\\
&+\ig \hat{A}^{3i}N_3\p_3^4\eta_p\,A^{3p}\p_3Q\,\p_3^4 v_i\dyy-\ig \hat{A}^{3i}N_3\p_3^4\eta_p\,A^{3p}\p_3Q\,(\p_3^4 \eta_r~A^{mr}~\p_m v_i)\dyy\\
=:&IB_0+IB_1+IB_2.
\end{aligned}
\end{equation}

First, $IB_1$ will produce the boundary energy with the help of Rayleigh-Taylor sign condition \eqref{sign} and the error terms will be cancelled with $IB_2$. In specific, we have
\begin{equation}\label{IB10}
\begin{aligned}
IB_1=&-\ig \left(-\frac{\p Q}{\p N}\right)JA^{3i}\p_3^4\eta_p\,A^{3p}\p_3^4 \p_t\eta_i\dyy\\
=&-\frac{1}{2}\ddt\ig\left(-J\frac{\p Q}{\p N}\right)\left|A^{3i}\p_3^4\eta_i\right|^2\dyy\\
&-\frac12\ig\p_t\left(J\frac{\p Q}{\p N}\right)\left|A^{3i}\p_3^4\eta_i\right|^2\dyy+\ig\left(-J\frac{\p Q}{\p N}\right)\p_tA^{3i}\,\p_3^4\eta_p\,A^{3p}\p_3^4\eta_i\dyy\\
=:&IB_{11}+IB_{12}+IB_{13}.
\end{aligned}
\end{equation}

Invoking Rayleigh-Taylor sign condition, we get
\begin{equation}\label{IB11}
\int_0^TIB_{11}\dt\lesssim-\frac{c_0}{4}\ig\left|A^{3i}\p_3^4\eta_i\right|^2\dyy\bigg|^T_0,
\end{equation}and thus the term $IB_{12}$ can be directly controlled by the boundary energy
\begin{equation}\label{IB12}
IB_{12}\lesssim|\p_t(J\p_3 Q)|_{L^{\infty}}\left|A^{3i}\p_3^4\eta_i\right|_0^2\lesssim P(\EE(t)).
\end{equation}

Then we plug $\p_tA^{3i}=-A^{3r}~\p_m v_r~A^{mi}$ into $IB_{13}$ to get
\begin{align}
\label{IB13} IB_{13}=&\ig\left(\frac{\p Q}{\p N}\right)A^{3r}\p_m v_r\,\hat{A}^{mi}\p_3^4\eta_p\, A^{3p}\p_3^4\eta_i\dyy,
\end{align}and this term exactly cancel with $IB_2$ if we replace the indices $(r,i)$ by $(i,r)$.

It now remains to control $IB_0$. We have
\begin{equation}\label{IB00}
IB_0=-\ig N_3J~\p_3^4Q~(A^{3i}\p_3^4v_i)\dyy+\ig \hat{A}^{3i}N_3\p_3^4 Q\,\p_3^4\eta_p\, A^{lp}\p_l v_i\dyy=:IB_{01}+IB_{02}.
\end{equation}

To control $IB_0$, we shall differentiate the following relations
\begin{align}
\label{vbdry} A^{3i}\p_3v_i=&\diva v-A^{1i}\TP_1 v_i-A^{2i}\TP_2v_i=-\frac{JR'(q)}{\rho_0}\p_t q-A^{1i}\TP_1 v_i-A^{2i}\TP_2v_i.
\end{align}

In $IB_{01}$, we use the relation \eqref{vbdry} to get
\begin{equation}\label{n4vbdry}
\begin{aligned}
A^{3i}\p_3^4v_i=&\p_3^3(A^{3i}\p_3v_i)-\p_3^3A^{3i}\,\p_3 v_i-3\p_3^2A^{3i}\,\p_3^2 v_i-3\p_3 A^{3i}\,\p_3^3 v_i\\
=&-\p_3^3\left(\frac{JR'(q)}{\rho_0}\p_t q\right)-\sum_{L=1}^2\p_3^3(A^{Li}\TP_Lv_i)-\p_3^3A^{3i}\,\p_3 v_i-3\p_3^2A^{3i}\,\p_3^2 v_i-3\p_3 A^{3i}\,\p_3^3 v_i,
\end{aligned}
\end{equation}and thus $IB_{01}$ becomes
\begin{equation}\label{IB010}
\begin{aligned}
IB_{01}=&\ig N_3J~\p_3^4Q~\p_3^3\left(\frac{JR'(q)}{\rho_0}\p_t q\right)+\sum_{L=1}^2\ig N_3J~\p_3^4Q~\p_3^3(A^{Li}\TP_Lv_i)\\
&+\ig N_3J\p_3^4Q~\left(\p_3^3A^{3i}\,\p_3 v_i+3\p_3^2A^{3i}\,\p_3^2 v_i+3\p_3 A^{3i}\,\p_3^3 v_i\right)\\
=:&IB_{011}+IB_{012}+IB_{013}.
\end{aligned}
\end{equation}

In $IB_{012}$, since $A^{Li}$ has the form $\p_3\eta\times\TP\eta$, the highest order term contains $\p_3^3 A^{Li}=\p_3^4\eta\times\TP\eta+\cdots$ which cannot be directly controlled. However, this term can produce cancellation with $IB_{02}$. We have
\begin{equation}
\begin{aligned}
\p_3^3A^{Li}=&-\p_3^2(A^{Lp}\,\p_3\p_m\eta_p\,A^{mi})\\
=&-A^{Lp}\p_3^4\eta_p A^{3i}-\sum_{M=1}^2A^{Lp}\p_3^3\TP_M\eta_p A^{Mi}-[\p_3^2,A^{Lp}A^{mi}]\p_3\p_m\eta_p,
\end{aligned}
\end{equation}and thus $IB_{012}$ can be written as
\begin{align}
\label{IB0121} IB_{012}=&-\sum_{L=1}^2\ig \hat{A}^{3i} N_3\p_3^4Q\,\p_3^4\eta_p\,A^{Lp}\TP_L v_i\\
\label{IB0122} &-\sum_{L=1}^2 \ig N_3J\p_3^4Q\left(\sum_{M=1}^2A^{Lp}\p_3^3\TP_M\eta_p\,A^{Mi}+[\p_3^2,A^{Lp}A^{mi}]\p_3\p_m\eta_p\right).
\end{align}

On the other hand, we write $IB_{02}$ as
\begin{align}
\label{IB021} IB_{02}=&\ig \hat{A}^{3i}N_3\p_3^4 Q\,\p_3^4\eta_p\, A^{3p}\p_3 v_i\dyy\\
\label{IB022} &+\sum_{L=1}^2\ig \hat{A}^{3i}N_3\p_3^4 Q\,\p_3^4\eta_p\, A^{Lp}\TP_L v_i\dyy.
\end{align} Therefore, \eqref{IB022} exactly cancels with the main term \eqref{IB0121} in $IB_{012}$.

Now it remains to control $IB_{011},IB_{013}$ and \eqref{IB0122}, \eqref{IB021}. Invoking the relation
\begin{align}\label{qbdry}
\hat{A}^{3i}\p_3 Q=&-\sum_{L=1}^2\hat{A}^{Li}\TP_L Q-\rho_0\p_tv^i+\bp(J^{-1}\bp\eta^i),
\end{align} we get
\begin{equation}\label{n4qbdry1}
\begin{aligned}
\hat{A}^{3i}\p_3^4 Q=&\p_3^3(\hat{A}^{3i}\p_3 Q)-\p_3^3\hat{A}^{3i}\,\p_3Q-3\p_3^2\hat{A}^{3i}\,\p_3^2Q-3\p_3\hat{A}^{3i}\,\p_3^3Q\\
=&\p_3^3\left(-\rho_0\p_tv^i+\bp(J^{-1}\bp\eta)\right)-\sum_{L=1}^2\p_3^3(\hat{A}^{Li}\TP_L Q)\\
&-\p_3^3\hat{A}^{3i}\,\p_3Q-3\p_3^2\hat{A}^{3i}\,\p_3^2Q-3\p_3\hat{A}^{3i}\,\p_3^3Q.
\end{aligned}
\end{equation}Note that
\begin{itemize}
\setlength{\itemsep}{0pt}
\setlength{\parsep}{0pt}
\setlength{\parskip}{0pt}
\item The term $\hat{A}^{3i}$ is of the form $\TP\eta\times\TP\eta$, so the leading order term in $\p_3^3 A^{3i}$ should be $(\p_3^3\TP\eta)(\TP\eta)$.
\item The highest order term in $\p_3^3(\hat{A}^{Li}\TP_L Q)$ is $\p_3^3\hat{A}^{Li}\,\TP_L Q=0$ due to $\TP_LQ|_{\Gamma=0}$.
\item The highest order term in $\p_3^3(\bp(J^{-1}\bp\eta))$ is $(b_0\cdot\TP)\p_3^3(J^{-1}\bp\eta)$ because $b_0^3|_{\Gamma}=0$ makes $\bp$ tangential on the boundary.
\end{itemize}
Therefore, we can rewrite $\p_3^4 Q$ to be the terms of at most 3 normal derivatives and one tangential derivative:
\begin{equation}\label{n4qbdry2}
\begin{aligned}
\p_3^4 Q=&\underbrace{J^{-1}\hat{A}^{3i}\p_3\eta_i}_{=1} \p_3^4Q=J^{-1}\p_3\eta_i(\hat{A}^{3i}\p_3^4 Q)\\
=&J^{-1}\p_3\eta_i\bigg(\p_3^3\left(-\rho_0\p_tv^i+(b_0\cdot\TP)(J^{-1}\bp\eta)\right)-\sum_{L=1}^2\sum_{N=0}^2\binom{3}{N}(\p_3^{N}\hat{A}^{Li})(\p_3^{3-N}\TP_L Q)\\
&~~~~~~~~~~~~~~~~-\p_3^3\hat{A}^{3i}\,\p_3Q-3\p_3^2\hat{A}^{3i}\,\p_3^2Q-3\p_3\hat{A}^{3i}\,\p_3^3Q\bigg).
\end{aligned}
\end{equation}

In \eqref{IB021}, we need to rewrite $A^{3p}\p_3^4\eta_p$ by using $A^{3p}\p_3\eta_p=1$ in $\bar{\Omega}$ (and thus $\p_3^3(A^{3p}\p_3\eta_p)=0$)
\begin{equation}\label{n4etabdry}
\begin{aligned}
A^{3p}\p_3^4\eta_p=-\p_3^3A^{3p}\,\p_3\eta_p-3\p_3^2A^{3p}\,\p_3^2\eta_p-3\p_3A^{3p}\,\p_3^3\eta_p.
\end{aligned}
\end{equation}

In the light of \eqref{n4qbdry1}-\eqref{n4etabdry}, we are able to write  $IB_{011},IB_{013}$ and \eqref{IB0122}, \eqref{IB021} in the form of
\begin{equation}\label{n3bdry1}
\ig N_3(\p_3^3\dd f)(\p_3^3\dd g)h\dyy+\text{lower order terms},
\end{equation}where $\dd=\TP$ or $\p_t$ or $b_0\cdot\TP$, and $f,g$ can be $\eta,v,q,J^{-1}\bp\eta$, and $h$ contains at most first order derivative of $\eta,v$. Then \eqref{n3bdry1} can be controlled in the following way
\begin{equation}\label{n3bdry2}
\begin{aligned}
\ig N_3(\p_3^3\dd f)(\p_3^3\dd g)h\dyy=&\left(\io (\p_3^4\dd f)(\p_3^3 \dd g)h-\io(\p_3^3\dd f)(\p_3^4 \dd g) h-\io(\p_3^3\dd f)(\p_3^3\dd g)(\p_3 h)\right)\\
\overset{\dd}{=}&-\io(\p_3^4 f)(\p_3^3 \dd^2 g)h-\io(\p_3^4 f)(\p_3^3 \dd g)(\dd h)\\
&+\io(\p_3^3\dd^2 f)(\p_3^4 g) h+\io(\p_3^3\dd f)(\p_3^4 g)(\dd h)-\io(\p_3^3\dd f)(\p_3^3\dd g)(\p_3 h)\\
\lesssim&(\|\p_3^4 f\|_{0}+\|\p_3^3\dd^2 f\|_0)(\|\p_3^4 g\|_{0}+\|\p_3^3\dd^2 g\|_0)\|\p h\|_{L^{\infty}}\lesssim \|f\|_{8,*}\|g\|_{8,*}\|h\|_{3},
\end{aligned}
\end{equation}which gives the control of  $IB_{011},IB_{013}$ and \eqref{IB0122}, \eqref{IB021}.
\begin{rmk}
If we integrate $\dd=\p_t$ by parts in \eqref{n3bdry2} (such term appears in a leading order term $\p_3^3\p_t v$ in $\p_3^4Q$), then we should proceed the estimate under time integral and also consider the terms like $\io(\p_3^4 f)(\p_3^3\dd g)h$ which can be controlled by
\begin{equation}\label{n3bdry22}
\begin{aligned}
&\io(\p_3^4 v)(\p_3^3\dd g)h\lesssim\eps\|\p_3^4 v\|_0^2+\frac{1}{8\eps}\|\p_3^3\dd g\|_0^4+\frac{1}{8\eps}\|h\|_{L^{\infty}}^4\\
\lesssim& \eps\|\p_3^4 v\|_0^2+\frac{1}{8\eps}\left(\|g(0)\|_{7,*}^4+\|h(0)\|_{2}^4+\int_0^T\|\p_3^3\dd\p_t g(t)\|_0^4+\|\p_t h(t)\|_2^4\right)\\
\lesssim& \eps\|\p_3^4 v\|_0^2+\PP_0+\int_0^T P(\|g\|_{8,*},\|h\|_{5,*})\dt.
\end{aligned}
\end{equation}
\end{rmk}
According to \eqref{n3bdry2}-\eqref{n3bdry22}, we can finalize the estimates of the boundary integral $IB$ as follows
\begin{equation}\label{IB}
IB\lesssim\eps\|\p_3^4 v\|_0^2-\frac{c_0}{4}\ddt\ig\left|A^{3i}\p_3^4\eta_i\right|^2\dyy+P(\EE(t)).
\end{equation}

\subsection{Energy estimates of purely normal derivatives}

Now, \eqref{IB} together with \eqref{n4energy0}, \eqref{n43}, \eqref{n41}, \eqref{I} gives the estimates of Alinhac good unknowns of $v,Q$ in the case of purely non-weighted normal derivatives
\begin{equation}\label{n4energy}
\|\VV\|_0^2+\left\|\p_3^4\left(J^{-1}\bp\eta\right)\right\|_0^2+\|\p_3^4 q\|_0^2+\frac{c_0}{4}\left|A^{3i}\p_3^4\eta_i\right|_0^2\bigg|_{t=T}\lesssim\eps\|\p_3^4 v\|_0^2+ \PP_0+ P(\EE(T))\int_0^T P(\EE(t))\dt.
\end{equation} Finally, by the definition of Alinhac good unknown \eqref{n4good} and $\p_3^4\eta|_{t=0}=\mathbf{0}$, $\p_3^4 v$ is controlled by
\begin{equation}
\|\p_3^4 v\|_0^2\lesssim\|\VV\|_0^2+\|a\p v\|_{L^{\infty}}^2\int_0^T\|\p_3^4 v\|_0^2\dt\lesssim\|\VV\|_0+P(\EE(T))\int_0^T P(\EE(t))\dt,
\end{equation}and thus by choosing $\eps>0$ sufficiently small, we get
\begin{equation}\label{n4}
\|\p_3^4 v\|_0^2+\left\|\p_3^4\left(J^{-1}\bp\eta\right)\right\|_0^2+\|\p_3^4 q\|_0^2+\frac{c_0}{4}\left|A^{3i}\p_3^4\eta_i\right|_0^2\bigg|_{t=T}\lesssim \PP_0+ P(\EE(T))\int_0^T P(\EE(t))\dt.
\end{equation}

\section{Control of purely tangential derivatives}\label{sect t8}

Now we consider the purely tangential derivatives. In this case, the top order derivative becomes $\p_*^I=\p_t^{i_0}\TP_1^{i_1}\TP_2^{i_2}$ with $i_0+i_1+i_2=8$. We will prove the following estimates by a modified Alinhac good unknown method.

\begin{prop}\label{prop t8}
The following energy inequality holds for any sufficiently small $\eps>0$
\begin{equation}\label{t8t8}
\sum_{i_3=i_4=0}\|\p_*^I v\|_0^2+\left\|\p_*^I\left(J^{-1}\bp\eta\right)\right\|_0^2+\|\p_*^I q\|_0^2+\frac{c_0}{4}\left|A^{3i}\p_*^I\eta_i\right|_0^2\bigg|_{t=T}\lesssim \varepsilon\|\p_3\p_t^6 v\|_0^2+\PP_0+ P(\EE(T))\int_0^T P(\EE(t))\dt.
\end{equation}
\end{prop}

For simplicity, we mainly study the case $i_0=0$, i.e., $\p_*^I=\TP_1^{i_1}\TP_2^{i_2}$ with $i_1+i_2=8$. For sake of clean notations, we denote $\TP^8=\TP_1^{i_1}\TP_2^{i_2}$. In fact, most of the steps of the proof in this section are completely applicable to the case of $i_0>0$.

\subsection{The case of full spatial derivatives}\label{ts8}

\subsubsection{Derivation of ``modified Alinhac good unknowns" in anisotropic Sobolev space}\label{ts8AGU}

We still use Alinhac good unknowns to control the tangential derivatives. However, we cannot directly replace $\p_3^4$ by $\TP^8$ in \eqref{n4good} because the commutator contains the terms like $\TP^7\p\eta$, $\TP^7\p v$ and $\TP^7\p Q$ whose $L^2$-norm cannot be controlled in $H_*^8$. In specific, we have
\begin{equation}\label{t8proof1}
\begin{aligned}
\TP^8(\pa^{i}f)=&\pa^i(\TP^8 f)+(\TP^8A^{li})\p_l f+[\TP^8,A^{li},\p_l f] \\
=&\pa^{i}(\TP^8 f)-\TP^7(A^{lr}\,\TP\p_{m}\eta_{r}\,A^{mi})\p_l f+[\TP^8,A^{li},\p_{l} f] \\
=&\pa^{i}(\TP^8 f-\TP^8\eta_{r}\,A^{lr}\,\p_l f)+\TP^8\eta_{r}\,\pa^{i}(\pa^{r} f)-([\TP^7,A^{lr}A^{mi}]\TP\p_{m}\eta_{r})\p_l f+[\TP^8,A^{li},\p_{l} f] .
\end{aligned}
\end{equation}
We notice that the $L^2(\Omega)$-norm of the following quantities coming from the last two terms of \eqref{t8proof1} cannot be controlled because $\TP^7$ may fall on $A=\p\eta\times\p\eta$ and $\p f$.
\begin{equation}\label{t8bad}
\begin{aligned}
e_1:=-\TP^7(A^{lr}A^{mi})\,\TP\p_m\eta_r\,\p_l f,&~~~e_2:=-7\TP(A^{lr}A^{mi})\,\TP^7\p_m\eta_r\,\p_l f\\
e_3:=8(\TP^7A^{li})(\TP\p_l f),&~~~e_4:=8(\TP A^{li})(\TP^7\p_l f).
\end{aligned}
\end{equation}
Here $8\TP^7$ means there are 8 terms of the form $\TP_1^{i_1}\TP_2^{i_2}$ with $i_1+i_2=7$. We will repeatedly use similar notations throughout the manuscript.

Our idea to overcome this difficulty is mainly based on the following three techniques:
\begin{enumerate}
\setlength{\itemsep}{0pt}
\setlength{\parsep}{0pt}
\setlength{\parskip}{0pt}
\item Modify the definition of ``Alinhac good unknowns": Rewrite these quantities in terms of $\nabla_A^i(\cdots)+~L^2$-bounded terms, and then merge the terms inside the covariant derivative $\pa^i$ into the ``Alinhac good unknowns".

\item Produce a weighted normal derivative to replace a non-weighted one: There are terms like $(\TP^7\p_3\eta)(\TP Q)$. Since $Q|_{\Gamma}=0$, we know $\TP Q|_{\Gamma}=0$. Therefore, we can estimate the $L^{\infty}$-norm of $\TP Q$ by fundamental theorem of calculus: (Suppose $y_3>0$ without loss of generality)
\[
|\TP Q(t,y_3)|_{L^{\infty}(\T^2)}=\left|0+\int_1^{y_3}\TP\p_3 Q(t,\zeta_3)d\zeta_3\right|_{L^{\infty}(\T^2)}\leq (1-y_3)\|\TP\p_3 Q\|_{L^{\infty}}\leq \sigma(y_3)\|\TP\p_3 Q\|_{L^{\infty}},
\] then we move the $\sigma(y_3)$ to $\TP^7\p_3\eta$ to get a weighted normal derivative $(\sigma\p_3)^1\TP^7 \eta$ whose $L^2$-norm can be directly bounded in $H_*^8$.
\item Replace $\pA Q$ (contains a normal derivative) by $-\rho_0\p_t v+\bp(J^{-1}\bp\eta)$ (only contains tangential derivative) in order to make the order of the derivatives lower thanks to the anisotropy of $H_*^m$.
\end{enumerate}

Now we analyze these extra terms from the commutator. We start with $8(\TP^7A^{li})(\TP\p_l f)$ and $8(\TP A^{li})(\TP^7\p_l f)$ coming from $[\TP^8,A^{li},\p_{l} f]$ in \eqref{t8proof1}. Since $\TP A^{li}=-A^{lp}\,\TP\p_m\eta_p\,A^{mi}$, we have
\[
\TP^7A^{li}=-A^{lp}\,\TP^7\p_m\eta_p\,A^{mi}-[\TP^6,A^{lp}A^{mi}]\p_m\eta_p,
\]where the highest order term in $[\TP^6,A^{lp}A^{mi}]\p_m\eta_p$ is $\TP^6\p_m\eta_p$ whose $L^2$-norm can be directly bounded by $\|\eta\|_{8,*}$. Therefore, we have
\begin{equation}\label{t8C1}
\begin{aligned}
8(\TP^7A^{li})(\TP\p_l f)=&-8(A^{mi}\,\p_m\TP^7\eta_p\,A^{lp})\TP\p_l f-8([\TP^6,A^{lp}A^{mi}]\p_m\eta_p)\TP\p_l f\\
=&-8\nabla_A^i(\TP^7\eta_p\,A^{lp}\TP\p_l f)\underbrace{+8\nabla_A^i(\nabla_A^p \TP f)\TP^7\eta_p-8([\TP^6,A^{lp}A^{mi}]\p_m\eta_p)\TP\p_l f}_{=:C_1(f)}\\
=:&-8\nabla_A^i(\TP^7\eta\cdot\nabla_A\TP f)+C_1(f),
\end{aligned}
\end{equation}where $C_1(f)$ can be controlled by using $H^{1/2}\hookrightarrow L^3$ and $H^{1}\hookrightarrow L^6$ in 3D domain
\begin{align*}
C_1(f)\lesssim&\|A\|_{L^{\infty}}^2\|\p^2\TP f\|_{L^6}\|\TP^7\eta\|_{L^3}+\|A\,\TP f\,\p A\|_{L^{\infty}}\|\TP^7\eta\|_{L^2}+P(\|\eta\|_{8,*})\|\TP\p f\|_{L^{\infty}}\\
\lesssim& \|A\|_{L^{\infty}}^2\|\p^2 \TP f\|_{1}\|\lee\TP\ree^{1/2}\TP^7\eta\|_{0}+P(\|\eta\|_{8,*})\|\TP \p f\|_{L^{\infty}}\\
\lesssim&\|A\|_{L^{\infty}}^2\|\p^2 \TP f\|_{1}\|\TP^7\eta\|_{0}^{1/2}\|\TP^8\eta\|_{0}^{1/2}+P(\|\eta\|_{8,*})\|\TP\p f\|_{L^{\infty}}\\
\lesssim&P(\|\eta\|_3)\|f\|_{7,*}\|\eta\|_{8,*}+P(\|\eta\|_{8,*})\|\TP\p f\|_{L^{\infty}}.
\end{align*}

The term $8(\TP A^{li})(\TP^7\p_l f)$ should be treated differently in the case of $f=v_i$ and $f=Q$ respectively.
\begin{itemize}
\item When $f=v_i$, then this term becomes
\begin{equation}\label{t8C2v}
\begin{aligned}
8(\TP A^{li})(\TP^7\p_l v_i)=&-8A^{lp}\,\TP\p_m\eta_p\,A^{mi}\TP^7\p_lv_i=-8A^{li}\,\TP\p_m\eta_i\,A^{mp}\,\TP^7\p_lv_p\\
=&-8\nabla_A^i(\TP^7 v_p\,A^{mp}\,\TP\p_m\eta_i)+\underbrace{8\nabla_A^i(\TP\p_m\eta_i\,A^{mp})\TP^7v_p}_{=:C_2(v)}\\
=:&-8\nabla_A^i(\TP^7 v\cdot\pa\TP\eta_i)+C_2(v),
\end{aligned}
\end{equation}and similarly we have $\|C_2(v)\|_0\lesssim P(\|\eta\|_{7,*})\|v\|_{8,*}$.

\item When $f=Q$, we cannot mimic the simplification as above. Instead, we need to invoke the MHD equation to replace $\pa Q$ by tangential derivatives. We consider
\begin{equation}\label{t8C2Q}
\begin{aligned}
8(J\TP A^{li})(\TP^7\p_l Q)=&-8(\hat{A}^{lp}\,\TP\p_m\eta_p\,A^{mi})\TP^7\p_lQ\\
=&-8\TP^7(\hat{A}^{lp}\p_l Q)\,A^{mi}\TP\p_m\eta_p+8(\TP^7 \hat{A}^{lp})(\p_l Q)(\TP\p_m\eta_p\,A^{mi})\\
&+8\sum_{N=1}^6 \binom{7}{N}(\TP^N \hat{A}^{lp})(\TP^{7-N}\p_l Q)(\TP\p_m\eta_p\,A^{mi})\\
=&8\TP^7\left(\rho_0\p_tv^p-\bp(J^{-1}\bp\eta^p)\right)A^{mi}\TP\p_m\eta_p+8(\TP^7 \hat{A}^{lp})(\p_l Q)(\TP\p_m\eta_p\,A^{mi})\\
&+8\sum_{N=1}^6 \binom{7}{N}(\TP^N \hat{A}^{lp})(\TP^{7-N}\p_l Q)(\TP\p_m\eta_p\,A^{mi})\\
=:&C_{21}+C_{22}+C_{23}.
\end{aligned}
\end{equation}

The $L^2$-norm of $C_{23}$ can be directly controlled since the top order derivative is $\TP^6\p$
\begin{equation}\label{t8C23}
\|C_{23}\|_0\lesssim \|\eta\|_{8,*}\|Q\|_{8,*} P(\|\eta\|_{7,*}).
\end{equation}

The $L^2$-norm of $C_{22}$ can be directly controlled when $l=3$ because $\hat{A}^{3p}$ consists of $\TP\eta\times\TP\eta$. When $l=1,2$, we need to invoke the second technique above, i.e., using $\TP Q|_{\Gamma}=0$ to produce a weight function $\sigma(y_3)$.
\begin{equation}\label{t8C22}
\begin{aligned}
\|C_{22}\|_0\lesssim&\|\TP^7\hat{A}^{3p}\|_0\|\p_3 Q\,\TP\p\eta\, a\|_{L^{\infty}}+\sum_{L=1}^2\|(\TP^7\hat{A}^{Lp})(\TP_L Q)(\TP\p_m\eta_p\,A^{mi})\|_{0}\\
\lesssim& P(\|\eta\|_{7,*})\|\TP^8\eta\|_{0}\|Q\|_3+\sum_{L=1}^2\|(\TP^7\hat{A}^{Lp})(\sigma(y_3)\p_3\TP_L Q)(\TP\p_m\eta_p\,A^{mi})\|_{0}\\
\lesssim& P(\|\eta\|_{7,*})\|\TP^8\eta\|_{0}\|Q\|_3+\sum_{L=1}^2\|\sigma\TP^7\hat{A}^{Lp}\|_{0}\|(\p_3\TP_L Q)(\TP\p_m\eta_p\,A^{mi})\|_{L^{\infty}}\\
\lesssim& P(\|\eta\|_{7,*})\|Q\|_{7,*}\left(\|\TP^8\eta\|_{0}+\|(\sigma\p_3)\TP^7 \eta\|_0\right),
\end{aligned}
\end{equation}where we use the fact that $\hat{A}^{Lp}$ consists of $(\p_3\eta)(\TP\eta)$ in the last step.

Finally, $C_{21}$ can also be directly bounded because the top order derivatives are $\TP^7\p_t$ and $\TP^7\bp$. Note that $b_0^3|_{\Gamma}=0$ yields the following estimates by using the second technique mentioned above.
\[
\|b_0^3\p_3\TP^7(J^{-1}\bp\eta)\|_{0}\lesssim\|\p b_0\|_2\|(\sigma\p_3)\TP^7(J^{-1}\bp\eta)\|_{0},
\]and thus
\begin{equation}\label{t8C22}
C_{21}\lesssim P(\|\eta\|_{7,*})(\|\rho_0\|_{7,*}\|v\|_{8,*}+\|b_0\|_{7,*}\|\bp\eta\|_{8,*}).
\end{equation}
Therefore, we have the estimates for $C_2(Q):=8\TP A^{li}\TP^7\p_l Q$
\begin{equation}\label{t8C2q}
\|C_2(Q)\|_{0}\lesssim P(\|b_0\|_{7,*},\|\rho_0\|_{7,*},\|\eta\|_{7,*})(\|\eta\|_{8,*}+\|v\|_{8,*}+\|J^{-1}\bp\eta\|_{8,*}+\|Q\|_{8,*}).
\end{equation}
\end{itemize}

Next we analyze $-(\TP^7(A^{lr}A^{mi})~\TP\p_m\eta_r)\p_l f$ coming from $-([\TP^7, A^{lr}A^{mi}]\TP\p_m\eta_r)\p_lf$. There are two terms of top order derivatives:
\[
-\TP^7(A^{lr}A^{mi})\,\TP\p_m\eta_r\,\p_l f=-(\TP^7A^{lr})A^{mi}\,\TP\p_m\eta_r\,\p_l f-A^{lr}(\TP^7A^{mi})\,\TP\p_m\eta_r\,\p_l f-\sum_{N=1}^6\binom{7}{N}(\TP^{N}A^{lr})(\TP^{6-N}A^{mi})\TP\p_m\eta_r\,\p_l f,
\]where the $L^2$-norm of the last term can be directly controlled
\[
\left\|\sum_{N=1}^6\binom{7}{N}(\TP^{N}A^{lr})(\TP^{6-N}A^{mi})\TP\p_m\eta_r\,\p_l f\right\|_0\lesssim P(\|\eta\|_{8,*})\|f\|_3.
\]

Similarly as \eqref{t8C1}, the term $-A^{lr}(\TP^7A^{mi})\TP\p_m\eta_r\,\p_l f$ can be written as the covariant derivatives plus $L^2$-bounded terms
\begin{equation}\label{t8C3}
\begin{aligned}
-A^{lr}(\TP^7A^{mi})\TP\p_m\eta_r\,\p_l f=&A^{lr}A^{mp}(\p_k\TP^7\eta_p)A^{ki}\,\TP\p_m\eta_r\,\p_lf+([\TP^6,A^{mp}A^{ki}]\TP\p_k\eta_p)A^{lr}\,\TP\p_m\eta_r\,\p_lf\\
=&\nabla_A^i(\TP^7\eta_p\,A^{mp}\,\TP\p_m\eta_r\,A^{lr}\,\p_l f)\\
&\underbrace{-\TP^7\eta_p\,\nabla_A^i(A^{mp}\,\TP\p_m\eta_r\,A^{lr}\,\p_l f)+([\TP^6,A^{mp}A^{ki}]\TP\p_k\eta_p)A^{lr}\,\TP\p_m\eta_r\,\p_lf}_{=:C_3(f)}\\
=:&\nabla_A^i(\TP^7\eta\cdot\pa \TP\eta\cdot\pa f)+C_3(f),
\end{aligned}
\end{equation}where $C_3(f)$ can be directly controlled similarly as $C_1(f)$
\[
\|C_3(f)\|_0\lesssim P(\|\eta\|_{8,*})\|\p f\|_2.
\]

We then compute $-(\TP^7A^{lr})A^{mi}\,\TP\p_m\eta_r\,\p_l f$.

\begin{itemize}
\item When $f=v_i$: Similarly as in \eqref{t8C3}, we have
\begin{equation}\label{t8C4v}
\begin{aligned}
-(\TP^7A^{lr})A^{mi}\,\TP\p_m\eta_r\,\p_l v_i=&A^{lp}(\TP^7\p_k\eta_p)A^{kr}A^{mi}\,\TP\p_m\eta_r\,\p_l v_i-([\TP^6,A^{lp}A^{kr}]\TP\p_k\eta_p)A^{mi}\,\TP\p_m\eta_r\,\p_l v_i\\
=&A^{lp}(\TP^7\p_k\eta_p)A^{ki}A^{mr}\,\TP\p_m\eta_i\,\p_l v_r-([\TP^6,A^{lp}A^{kr}]\TP\p_k\eta_p)A^{mi}\,\TP\p_m\eta_r\,\p_l v_i\\
=&\nabla_A^i(\TP^7\eta_p\,A^{lp}\,\p_l v_r\,A^{mr}\,\TP\p_m\eta_i)\\
&\underbrace{-\nabla_A^i(A^{lp}A^{mr}\TP\p_m\eta_i\,\p_lv_r)\TP^7\eta_p-([\TP^6,A^{lp}A^{kr}]\TP\p_k\eta_p)A^{mi}\,\TP\p_m\eta_r\,\p_l v_i}_{=:C_4(v)}\\
=:&\nabla_A^i(\TP^7\eta\cdot\pa v\cdot\pa\TP\eta_i)+C_4(v),
\end{aligned}
\end{equation}where $C_4(v)$ can be directly controlled similarly as $C_1(f)$
\[
\|C_4(v)\|_0\lesssim P(\|\eta\|_{8,*})\|\p v\|_2.
\]

\item When $f=Q$: If $l=3$, then this term can be directly controlled since $A^{3r}=J^{-1}\TP\eta\times\TP\eta$ only contains first-order tangential derivatives. If $l=1,2$, then we can mimic the treatment of $C_{22}$, i.e., using $\TP_L Q|_{\Gamma}=0$ and fundamental theorem of calculus to produce a weight function $\sigma(y_3)$ and move that to $\TP^7 A^{lr}$. Define $C_4(Q):=-(\TP^7A^{lr})A^{mi}\,\TP\p_m\eta_r\,\p_lQ$, then
\begin{equation}\label{t8C4q}
\begin{aligned}
\|C_4(Q)\|_0\lesssim&\|(\TP^7A^{3r})A^{mi}\,\TP\p_m\eta_r\,\p_3Q\|_0+\sum_{L=1}^2\|(\TP^7A^{Lr})A^{mi}\,\TP\p_m\eta_r\,\TP_LQ\|_0\\
\lesssim&\|\TP^8\eta\|_0\|Q\|_3 P(\|\p\eta\|_2,\|\TP\p\eta\|_2)+\sum_{L=1}^2\|\sigma\TP^7A^{Lr}\|_0\|A^{mi}\,\TP\p_m\eta_r\,\TP_L\p_3Q\|_{L^{\infty}}\\
\lesssim&\left(\|\TP^8\eta\|_{0}+\|(\sigma\p_3)\TP^7 \eta\|_0\right)P(\|Q\|_3,\|\TP Q\|_3,\|\eta\|_{7,*}).
\end{aligned}
\end{equation}
\end{itemize}

Next we analyze $-7\TP(A^{lr}A^{mi})\TP^7\p_m\eta_r~\p_l f$ coming from $-[\TP^7, A^{lr}A^{mi}]\TP\p_m\eta_r~\p_lf$. This term cannot be directly controlled when $m=3$. We should analyze it term by term. First we have
\begin{align*}
-7\TP(A^{lr}A^{mi})\TP^7\p_m\eta_r~\p_l f=&-7\TP A^{lr}\,A^{mi}\,\TP^7\p_m\eta_r\,\p_l f-7A^{lr}\,\TP A^{mi}\,\TP^7\p_m\eta_r\,\p_l f\\
=&7A^{lp}\,\p_k\TP\eta_p\,A^{kr}A^{mi}\,\TP^7\p_m\eta_r\,\p_l f+7A^{lr}A^{mp}\,\p_k\TP\eta_p\,A^{ki}\,\TP^7\p_m\eta_r\,\p_l f.
\end{align*}

The first term can be directly rewritten as follows
\begin{equation}\label{t8C5}
\begin{aligned}
7A^{lp}~\p_k\TP\eta_p~A^{kr}A^{mi}~\TP^7\p_m\eta_r~\p_l f=&7\nabla_i(\TP^7\eta_r~A^{kr}~\p_k\TP\eta_p~A^{lp}~\p_l f)\underbrace{-7\nabla_A^i(A^{kr}~\p_k\TP\eta_p~A^{lp}~\p_l f)\TP^7\eta_r}_{C_5(f)}\\
=:&7\nabla_A^i(\TP^7\eta\cdot\pa\TP\eta\cdot\pa f)+C_5(f),
\end{aligned}
\end{equation}where $C_5(f)$ can be similarly controlled as $C_1(f)$
\[
\|C_5(f)\|_0\lesssim P(\|\eta\|_{8,*})\|\p f\|_{3}.
\]

Then we analyze $7A^{lr}A^{mp}(\p_k\TP\eta_p)A^{ki}(\TP^7\p_m\eta_r)\p_l f$, which needs different treatment for $f=v_i$ and $f=Q$ respectively.
\begin{itemize}
\item When $f=v_i$, we have the following simplification
\begin{equation}\label{t8C6v}
\begin{aligned}
&7A^{lr}A^{mp}\p_k\TP\eta_p\,A^{ki}\,\TP^7\p_m\eta_r\,\p_l v_i=7A^{lr}A^{mi}\,\p_k\TP\eta_i\,A^{kp}\,\TP^7\p_m\eta_r\,\p_l v_p\\
=&7\nabla_A^i(\TP^7\eta_r\,A^{lr}\,\p_lv_p\,A^{kp}\,\TP\p_k\eta_i)+\underbrace{7\nabla_A^i(A^{lr}\,\p_lv_p\,A^{kp}\,\TP\p_k\eta_i)\TP^7\eta_r}_{C_6(v)}\\
=:&7\nabla_A^i(\TP^7\eta\cdot\pa v\cdot\pa\TP\eta_i)+C_6(v),
\end{aligned}
\end{equation}and $\|C_6(v)\|_0\lesssim P(\|\eta\|_{7,*})\|\eta\|_{8,*}\|v\|_{3}$ follows from direct computation, analogous to the analysis of the first term in $C_1(f)$.

\item When $f=Q$, this term becomes
\begin{equation}
\begin{aligned}
C_6(Q):=&-7A^{lr}(\TP A^{mi})(\TP^7\p_m\eta_r)\p_l Q\\
=&-7\left(\underbrace{\TP^7(A^{lr}\p_m\eta_r)}_{=\TP^7\delta^l_m=0}-(\TP^7A^{lr})(\p_m\eta_r)-\sum_{N=1}^6\binom{7}{N} (\TP^NA^{lr})(\TP^{7-N}\p_m\eta_r)\right)\TP A^{mi}\,\p_l Q\\
=&7(\TP^7A^{3r})\p_m\eta_r\,\TP A^{mi}\,\p_3 Q+\sum_{L=1}^2(\TP^7A^{Lr})\p_m\eta_r\,\TP A^{mi}\,\TP_L Q\\
&+7\sum_{N=1}^6\binom{7}{N} (\TP^NA^{lr})(\TP^{7-N}\p_m\eta_r)\TP A^{mi}\,\p_l Q\\
=:&C_{61}+C_{62}+C_{63}.
\end{aligned}
\end{equation}

Since $A^{3r}=J^{-1}\TP\eta\times\TP\eta$, we know the top order term is of the form $\TP^8\eta\times\TP\eta$ and thus $C_{61}$ can be directly controlled
\[
\|C_{61}\|_0\lesssim P(\|\eta\|_{3})\|\eta\|_{8,*}\|\p_3 Q\|_{2}.
\]

The term $C_{62}$ can be treated in the same way as $C_4(Q)$ in \eqref{t8C4q} by using $\TP_LQ|_{\Gamma}=0$ to produce a weight function $\sigma$
\[
\|C_{62}\|_0\lesssim (\|(\sigma\p_3)\TP^7\eta\|+\|\TP^8\eta\|_0)P(\|\eta\|_{7,*})\|\TP\p_3 Q\|_2\lesssim P(\|\eta\|_{7,*})\| Q\|_{7,*}\|\eta\|_{8,*}.
\]

Finally, $C_{63}$ can be directly controlled
\[
\|C_{63}\|_0\lesssim  P(\|\eta\|_{7,*})\|\eta\|_{8,*}\|\p Q\|_{2},
\] and thus
\begin{equation}\label{t8C6q}
\|C_6(Q)\|_{0}\lesssim P(\|\eta\|_{7,*})\|\eta\|_{8,*}\| Q\|_{7,*}.
\end{equation}
\end{itemize}

Now we plug \eqref{t8C1}-\eqref{t8C2v}, \eqref{t8C2q}-\eqref{t8C6q} into \eqref{t8proof1} and define the ``modified Alinhac good unknowns" of $v$ and $Q$ with respect to $\TP^8$ as
\begin{equation}\label{t8goodv}
\begin{aligned}
\VV_i^*:=\TP^8v_i&-\TP^8\eta\cdot\pa v_i\\
&-8\TP^7\eta\cdot\pa\TP v_i-8\TP^7v\cdot\pa\TP\eta_i\\
&+\TP^7\eta\cdot\pa\TP\eta\cdot\pa v_i+\TP^7\eta\cdot\pa v\cdot\pa \TP\eta_i\\
&+7\TP^7\eta\cdot\pa\TP\eta\cdot\pa v_i+7\TP^7\eta\cdot\pa v\cdot\pa\TP\eta_i\\
=\TP^8v_i&-\TP^8\eta\cdot\pa v_i-8\TP^7\eta\cdot\pa\TP v_i-8\TP^7v\cdot\pa\TP\eta_i+8\TP^7\eta\cdot\pa\TP\eta\cdot\pa v_i+8\TP^7\eta\cdot\pa v\cdot\pa \TP\eta_i,
\end{aligned}
\end{equation}and
\begin{equation}\label{t8goodq}
\begin{aligned}
\QQ^*:=\TP^8Q&-\TP^8\eta\cdot\pa Q-8\TP^7\eta\cdot\pa\TP Q+8\TP^7\eta\cdot\pa\TP\eta\cdot\pa Q.
\end{aligned}
\end{equation}Then the modified good unknowns satisfy the following relations
\begin{equation}\label{t8goodgrad}
\TP^8(\diva v)=\pa\cdot \VV^*+\sum_{M=0}^6 C_M(v),~~~\TP^8(\pa Q)=\pa\QQ^*+\sum_{M=0}^6 C_M(Q),
\end{equation}where $C_0(f)$ comes from the directly controllable terms in the RHS of \eqref{t8proof1}
\begin{equation}\label{t8C0}
C_0(f):=\TP^8\eta_r~\pa^i(\pa^r f)-\sum_{N=2}^{6}\binom{7}{N}\TP^{N}(A^{lr}A^{mi})\TP^{7-N}(\TP\p_{m}\eta_{r})\p_l f+\sum_{N=2}^{6}\binom{8}{N}(\TP^NA^{li})(\TP^{8-N}\p_{l} f),
\end{equation}satisfies
\[
\|C_0(f)\|_0\lesssim P(\|\eta\|_{8,*})\|f\|_{8,*},
\] and $C_1\sim C_6$ are constructed in \eqref{t8C1}-\eqref{t8C2v}, \eqref{t8C2q}-\eqref{t8C6q}.

\subsection{Energy estimates of purely tangential derivatives}\label{ts8energy}

We denote $C^*(f):=C_0(f)+C_1(f)+\cdots+C_6(f)$ and the ``extra modification terms" in the modified Alinhac good unknowns by
\begin{align*}
(\Delta_v^*)_i:=&-8\TP^7\eta\cdot\pa\TP v_i-8\TP^7v\cdot\pa\TP\eta_i+8\TP^7\eta\cdot\pa\TP\eta\cdot\pa v_i+8\TP^7\eta\cdot\pa v\cdot\pa \TP\eta_i,\\
\Delta_Q^*:=&-8\TP^7\eta\cdot\pa\TP Q+8\TP^7\eta\cdot\pa\TP\eta\cdot\pa Q.
\end{align*} Then the modified Alinhac good unknowns become
\[
\VV^*=\TP^8 v-\TP^8\eta\cdot\pa v+\Delta_v^*,~~\QQ^*=\TP^8 Q-\TP^8\eta\cdot\pa Q+\Delta_Q^*.
\]
\begin{rmk}
There are more modification terms in $\VV^*$ than in $\QQ^*$. The reason is that we can replace $\pa Q$ which contains a normal derivative with tangential derivative ($\p_t v$ and $\bp(J^{-1}\bp\eta)$) by invoking the MHD equation. However, similar relation only holds for $\diva v$ instead of $\pa v$. Therefore, for those terms in the commutators containing $v$, we have to rewrite them to be the covariant derivatives of the modifition terms plus $L^2(\Omega)$-bounded terms.
\end{rmk}
It is straightforward to see that the $L^2(\Omega)$ norms of $\Delta_v^*,\Delta_Q^*,\p_t(\Delta_v^*)$ and $\p_t(\Delta_Q)$ can be controlled by $P(\EE(t))$
\begin{equation}\label{t8errorvt}
\begin{aligned}
\|\p_t(\Delta_v^*)\|_0\lesssim&\|\TP^7 v\|_0(\|\pa\TP v\|_{2}+\|\pa\TP\eta\|_2\|\pa v\|_2)+\|\TP^7\p_t v\|_0\|\pa\TP\eta\|_2\\
&+\|\TP^7\eta\|_0(\|\pa\TP\p_t v\|_2+\|\pa\TP\eta\|_2\|\pa\p_tv\|_2+\|\pa \TP v\|_2\|\pa v\|_2)\\
\lesssim&P(\|\eta\|_{8,*},\|v\|_{8,*}),
\end{aligned}
\end{equation}

\begin{equation}\label{t8errorqt}
\begin{aligned}
\|\p_t(\Delta_Q^*)\|_0\lesssim&\|\TP^7 v\|_0(\|\pa\TP Q\|_{2}+\|\pa\TP\eta\|_2\|\pa Q\|_2)\\
&+\|\TP^7\eta\|_0(\|\pa\TP\p_t Q\|_2+\|\pa\TP\eta\|_2\|\pa\p_tQ\|_2+\|\pa \TP v\|_2\|\pa Q\|_2)\\
\lesssim&P(\|\eta\|_{8,*},\|v\|_{7,*},\|Q\|_{8,*}),
\end{aligned}
\end{equation}

\begin{equation}\label{t8errorvq}
\|\Delta_Q^*\|_0+\|\Delta_v^*\|_0\lesssim P(\|\eta\|_{7,*},\|v\|_{7,*},\|Q\|_{7,*}).
\end{equation}

Now we take $\TP^8$ in the second equation of compressible MHD system \eqref{CMHDL} to get
\[
R\p_t(\TP^8 v)-J^{-1}\bp\TP^8\left(J^{-1}\bp\eta\right)+\TP^8(\pa Q)=\left[R,\TP^8 \right]\p_t v+\left[\TP^8,J^{-1}\bp\right]\left(J^{-1}\bp\eta\right).
\] Then invoking \eqref{t8goodgrad} to get
\[
R\p_t(\TP^8 v)-J^{-1}\bp\TP^8\left(J^{-1}\bp\eta\right)+\pa\QQ^*=\left[R,\TP^8 \right]\p_t v+\left[\TP^8,J^{-1}\bp\right]\left(J^{-1}\bp\eta\right)-C^*(Q).
\]Finally, plugging $\VV^*=\TP^8 v-\TP^8\eta\cdot\pa v+\Delta_v^*$ yields the evolution equation of $\VV^*$ and $\QQ^*$
\begin{equation}\label{t8goodeq}
\begin{aligned}
R\p_t\VV^* -J^{-1}\bp\TP^8\left(J^{-1}\bp\eta\right)+\pa\QQ^*=&\left[R,\TP^8 \right]\p_t v+\left[\TP^8,J^{-1}\bp\right]\left(J^{-1}\bp\eta\right)\\
&-C^*(Q)+R\p_t(-\TP^8\eta\cdot\pa v+\Delta_v^*)
\end{aligned}
\end{equation}

We denote the RHS of \eqref{t8goodeq} by $\FF^*$. Similarly as in Section \ref{sect n4}, we compute the $L^2$-inner product of \eqref{t8goodeq} and $J\VV^*$ to get the energy identity
\begin{equation}\label{t8energy0}
\frac{1}{2}\ddt\io\rho_0\left|\VV^*\right|^2\dy=\io\bp\TP^8(J^{-1}\bp\eta)\cdot\VV^*-\io(\nabla_{\hat{A}}\QQ^*)\cdot\VV^*+\io J\FF^*\cdot\VV^*.
\end{equation}

\subsubsection{Interior estimates}

Using \eqref{t8errorvt}, the third integral on RHS of \eqref{t8energy0} is controlled by direct computation
\begin{equation}\label{t83}
\io J\FF^*\cdot\VV^*\lesssim\|J\FF^*\|_0\|\VV^*\|_0\lesssim P\left(\left\|\left(\rho_0,\eta,v,Q,b_0,\bp\eta\right)\right\|_{8,*}\right)\|\VV^*\|_0.
\end{equation}

The first integral on RHS of \eqref{t8energy0} can be similarly treated as \eqref{n410}-\eqref{n41} by replacing $\p_3^4$ by $\TP^8$ and $\|\cdot\|_4$-norm by $\|\cdot\|_{8,*}$-norm. We omit the details and list the result
\begin{equation}\label{t81}
\io\bp\TP^8(J^{-1}\bp\eta)\cdot\VV^*\dy\lesssim-\frac12\ddt\io J\left|\TP^8(J^{-1}\bp\eta)\right|^2\dy+K_{11}^*+P\left(\|(\eta,v,b_0,\bp\eta)\|_{8,*}\right),
\end{equation}where $K_{11}^*$ is defined to be
\begin{equation}\label{t8K1}
K_{11}^*:=-\io  J\TP^8(J^{-1}\bp\eta)\cdot\left(J^{-1}\bp\eta\right)\TP^8(\diva v)\dy.
\end{equation}

Next we analyze the term $-\io J\pa \QQ\cdot\VV$. Integrating by parts and using Piola's identity, we get
\begin{equation}\label{t820}
\begin{aligned}
-\io(\nabla_{\hat{A}}\QQ^*)\cdot\VV^*=\io J\QQ(\nabla_A\cdot\VV^*)-\ig J\QQ A^{li}N_l\VV_i^*\dyy=:I^*+IB^*.
\end{aligned}
\end{equation}

Invoking \eqref{t8goodv}, \eqref{t8goodgrad} and $Q=q+\frac12|J^{-1}\bp\eta|^2$, we get
\begin{equation}\label{I*0}
\begin{aligned}
I^*=&\io J\TP^8 q\, \TP^8(\diva v)\dy+\io J\TP^8\left(\frac12\left|J^{-1}\bp\eta\right|^2\right)\TP^8(\diva v)\dy\\
&+\io (-\TP^8\eta_p\,\hat{A}^{lp}\,\p_l Q+\Delta_Q^*)\TP^8(\diva v)\dy-\io\TP^8 Q~C^*(v)\dy\\
=:&I_1^*+I_2^*+I_3^*+I_4^*,
\end{aligned}
\end{equation}where $I_4^*$ can be directly controlled by using the estimates of $C^*(v)$
\begin{equation}\label{I*4}
I_4^*\lesssim\|\TP^8 Q\|_0\|C^*(v)\|_0\lesssim P(\|\eta\|_{8,*})\|\TP^8Q\|_0\|v\|_{8,*}.
\end{equation}

Similarly, $I_2^*$ produces another higher order term to cancel with $K_{11}^*$
\begin{equation}\label{I*2}
\begin{aligned}
I_2=&\underbrace{\io J\TP^8\left(J^{-1}\bp\eta\right) \cdot\left(J^{-1}\bp\eta\right)\TP^8(\diva v)}_{\text{exactly cancel with }K_{11}^*}\dy\\
&+\sum_{N=1}^{7}\binom{8}{N}\io J\TP^N\left(J^{-1}\bp\eta\right) \cdot\TP^{8-N}\left(J^{-1}\bp\eta\right)\TP^8(\diva v)\dy\\
=&-K_{11}^*-\sum_{N=1}^{7}\binom{8}{N}\io\frac{J^2R'(q)}{\rho_0}\TP^N\left(J^{-1}\bp\eta\right) \cdot\TP^{8-N}\left(J^{-1}\bp\eta\right)\TP^8\p_t q\dy\\
&-\sum_{N=1}^{7}\binom{8}{N}\io J\TP^N\left(J^{-1}\bp\eta\right) \cdot\TP^{8-N}\left(J^{-1}\bp\eta\right)\left(\left[\TP^8,\frac{JR'(q)}{\rho_0}\right]\p_t q\right)\dy\\
=:&-K_{11}^*+I_{21}^*+I_{22}^*
\end{aligned}
\end{equation}

Similarly as in \eqref{I21}-\eqref{I22}, the term $I_{21}^*$ should be controlled by integrating $\p_t$ by parts under time integral and $I_{22}^*$ can be directly controlled. We omit the details
\begin{align}
\label{I*21} \int_0^T I_{21}^*\lesssim&\varepsilon\|\TP^8 q\|_0^2+\PP_0+\int_0^TP(\EE(t))\dt\\
\label{I*22} I_{22}^*\lesssim&\|J^{-1}\bp\eta\|_{7,*}^2\|q\|_{8,*}.
\end{align}

The term $I_1^*$ produces the energy term $\|\TP^8 q\|_0^2$ as in \eqref{I1}.
\begin{equation}\label{I*1}
I_1^*\lesssim-\frac12\ddt\io\frac{J^2R'(q)}{\rho_0}|\TP^8 q|^2+P(\|q\|_{8,*},\|\rho_0\|_{8,*},\|\eta\|_{8,*}).
\end{equation}

$I_3^*$ can be controlled by integrating $\p_t$ by parts under time integral after invoking $\diva v=-\frac{JR'(q)}{\rho_0}\p_t q$ and \eqref{t8errorqt}-\eqref{t8errorvq}.
\begin{equation}\label{I*3}
\begin{aligned}
\int_0^TI_3^*=&\int_0^T\io \frac{JR'(q)}{\rho_0}(\TP^8\eta_p\,\hat{A}^{lp}\,\p_l Q-\Delta_Q^*)\TP^8\p_t q\dy+\underbrace{\int_0^T\io (\TP^8\eta_p\,\hat{A}^{lp}\,\p_l Q-\Delta_Q^*)\left(\left[\TP^8, \frac{JR'(q)}{\rho_0}\right]\p_t q\right)}_{L_2^*}\dy\\
\overset{\p_t}{=}&-\int_0^T\io\p_t\left( \frac{JR'(q)}{\rho_0}\,\TP^8\eta_p\,\hat{A}^{lp}\,\p_l Q-\Delta_Q^*\right)\TP^8q\dy+\io\frac{JR'(q)}{\rho_0}(\TP^8\eta_p\,\hat{A}^{lp}\,\p_l Q-\Delta_Q^*)\TP^8q\dy\bigg|^T_0+L_2^*\\
\lesssim&\PP_0+\left(\left\|\frac{JR'(q)}{\rho_0}\,A\,\p Q\right\|_{L^{\infty}}\|\TP^8\eta\|_0+\|\Delta_Q^*\|_0\right)\|\TP^8 q\|_0+\int_0^TP\left(\|(\eta,v,q,\rho_0)\|_{8,*}\right)\dt.\\
\lesssim&\PP_0+\eps\|\TP^8 q\|_0^2+\left\|\frac{JR'(q)}{\rho_0}\,A\,\p Q\right\|_{L^{\infty}}^4+\|\TP^8\eta\|_0^4+\|\Delta_Q^*\|_0^2+\int_0^T P(\EE(t))\dt\\
\lesssim&\PP_0+\eps\|\TP^8 q\|_0^2+\int_0^T\left\|\p_t\left(\frac{JR'(q)}{\rho_0}\,A\,\p Q\right)\right\|_{L^{\infty}}^4+\|\TP^8 v(t)\|_0^4+\|\p_t(\Delta_Q^*)\|_0^2\dt+\int_0^T P(\EE(t))\dt\\
\lesssim&\eps\|\TP^8 q\|_0^2+\PP_0+\int_0^T P(\EE(t))\dt,
\end{aligned}
\end{equation}

Summarizing \eqref{I*0}-\eqref{I*3} and choosing $\eps>0$ to be sufficiently small, we get the estimates of $I^*$ under time integral
\begin{equation}\label{I*}
\int_0^TI^*\dt\lesssim-\frac{1}{2}\io\frac{J^2R'(q)}{\rho_0}\left|\TP^8 q\right|^2\dy\bigg|^T_0+\PP_0+\int_0^T P(\EE(t))\dt.
\end{equation}

\subsubsection{Boundary estimates}\label{sect t8bdry}

Now it remains to deal with the boundary integral $IB^*$. Since $Q|_{\Gamma}=0$, we know $$\QQ^*|_{\Gamma}=-\TP^8\eta_p\,A^{3p}\,\p_3 Q+\Delta_Q^*,$$ and $$\Delta_Q^*|_{\Gamma}=-8\TP^7\eta_p\,A^{3p}\,\TP\p_3 Q+8\TP^7\eta\cdot\pa\TP\eta_r\, A^{3r}\,\p_3Q.$$ Then the boundary integral $IB^*$ reads
\begin{equation}\label{IB*B}
\begin{aligned}
IB^*=&\ig \hat{A}^{3i}N_3\TP^8\eta_p\,A^{3p}\,\p_3 Q\,\TP^8 v_i\dyy-\ig \hat{A}^{3i}N_3(\TP^8\eta_pA^{3p}\p_3 Q)(\TP^8\eta\cdot\pa v_i)\dyy\\
&-\ig \hat{A}^{3i}N_3\Delta_Q^*\,\TP^8 v_i\dyy+\ig \hat{A}^{3i}N_3\Delta_Q^*\,\TP^8\eta\cdot\pa v_i\dyy\\
&-\ig  \hat{A}^{3i}N_3\Delta_Q^*(\Delta_v^*)_i\dyy+\ig \hat{A}^{3i}N_3(\TP^8\eta_p\,A^{3p}\,\p_3 Q)(\Delta_v^*)_i\dyy\\
=:&IB_1^*+IB_2^*+IB_3^*+IB_4^*+IB_5^*+IB_6^*.
\end{aligned}
\end{equation}

Before going to the proof, we would like to state our basic strategy to deal with the boundary control
\begin{itemize}
\setlength{\itemsep}{0pt}
\setlength{\parsep}{0pt}
\setlength{\parskip}{0pt}
\item $IB_1^*$ together with the Raylor-Taylor sign condition gives the boundary energy $|A^{3i}\TP^8\eta_i|_0^2$ and the extra terms can be cancelled by $IB_2^*$. This step also appears in the study of Euler equations \cite{CL2000priori,CS2010LWP,LL2018priori,Luo2018CWW,LuoZhang2020CWW} and incompressible MHD \cite{HaoLuo2014priori,GuWang2016LWP,Guaxi1,Guaxi2} and compressible resistive MHD \cite{Zhang2020CRMHD}. It actually gives the control of the second fundamental form of the free surface \cite{CL2000priori}.

\item $IB_3^*$: We can write $\TP^8 v_i=\TP^8\p_t\eta_i$ and integrate $\p_t$ by parts. When $\p_t$ falls on $\Delta_Q^*$, the boundary integral can be directly controlled by using trace lemma. When $\p_t$ falls on $\hat{A}^{3i}$, such terms exactly cancel with the top order term in $IB_4^*$.

\item $IB_5^*$ and $IB_6^*$: Direct computation together with the trace lemma gives the control.
\end{itemize}

We first compute $IB_1^*$. Similarly as \eqref{IB10}, we have
\begin{equation}\label{IB*10}
\begin{aligned}
IB_1^*=&-\ig \left(-\frac{\p Q}{\p N}\right)JA^{3i}\TP^8\eta_p\,A^{3p}\,\TP^8 \p_t\eta_i\dyy\\
=&-\frac{1}{2}\ddt\ig\left(-J\frac{\p Q}{\p N}\right)\left|A^{3i}\TP^8\eta_i\right|^2\dyy\\
&-\frac12\ig\p_t\left(J\frac{\p Q}{\p N}\right)\left|A^{3i}\TP^8\eta_i\right|^2\dyy+\ig\left(-J\frac{\p Q}{\p N}\right)\p_tA^{3i}\,\TP^8\eta_p\, A^{3p}\,\TP^8\eta_i\dyy\\
=:&IB_{11}^*+IB_{12}^*+IB_{13}^*,
\end{aligned}
\end{equation}

The term $IB_{11}^*$ together with the Rayleigh-Taylor sign condition gives the boundary energy
\begin{equation}
\int_0^TIB_{11}^*\leq-\frac{c_0}{4}\left|A^{3i}\TP^8\eta_i\right|_0^2\bigg|^T_0,
\end{equation} and $IB_{12}^*$ can be directly controlled by the boundary energy
\begin{equation}\label{IB*12}
IB_{12}^*\lesssim\left|A^{3i}\TP^8\eta_i\right|_0^2\left|\p_t\left(J\frac{\p Q}{\p N}\right)\right|_{L^{\infty}}\lesssim P(\EE(t)).
\end{equation}

Then we plug $\p_tA^{3i}=-A^{3r}\p_k v_rA^{ki}$ into $IB_{13}^*$ to get the cancellation structure
\begin{equation}\label{IB*13}
\begin{aligned}
IB_{13}^*=&\ig J\frac{\p Q}{\p N} A^{3r}\,\p_k v_rA^{ki}\,\TP^8\eta_p\,A^{3p}\TP^8\eta_i\\
=&\ig J\frac{\p Q}{\p N} A^{3i}\,\p_k v_i\,A^{kr}\,\TP^8\eta_p\, A^{3p}\,\TP^8\eta_r=-IB_2^*
\end{aligned}
\end{equation}

Next we analyze $IB_3^*$. We write $v_i=\p_t\eta_i$ and integrate this $\p_t$ by parts
\begin{equation}\label{IB*30}
\begin{aligned}
\int_0^TIB_3^*=&-\int_0^T\ig JA^{3i}N_3\Delta_Q^*~\TP^8\p_t\eta_i\dyy\dt\\
\overset{\p_t}{=}&\int_0^T\ig J~\p_tA^{3i}~N_3\Delta_Q^*~\TP^8\eta_i\dyy\dt\\
&-\int_0^T\ig A^{3i}N_3~\p_t(J\Delta_Q^*)~\TP^8\eta_i\dyy\dt-\ig JA^{3i}N_3~\Delta_Q^*~\TP^8\eta_i\dyy\bigg|^T_0\\
=:&IB_{31}^*+IB_{32}^*+IB_{33}^*.
\end{aligned}
\end{equation}

Again, plug $\p_tA^{3i}=-A^{3r}\p_k v_rA^{ki}$ into $IB_{31}^*$ to get the cancellation with $IB_4^*$
\begin{equation}\label{IB*31}
\begin{aligned}
IB_{31}^*=&-\int_0^T\ig JA^{3r}~\p_k v_r~A^{ki}~N_3~\Delta_Q^*~\TP^8\eta_i\dyy\dt\\
=&-\int_0^T\ig JA^{3i}~\p_k v_i~A^{kr}~N_3~\Delta_Q^*~\TP^8\eta_r\dyy\dt=-IB_4^*.
\end{aligned}
\end{equation}

For $IB_{33}^*$, we use the fact that $\TP^7\eta|_{t=0}=0$ (and thus $\Delta_Q^*|_{\Gamma}=0$ when $t=0$) together with Lemma \ref{trace} to get
\begin{equation}\label{IB*33}
\begin{aligned}
&\ig JA^{3i}N_3\Delta_Q^*~\TP^8\eta_i\dyy\bigg|_{t=T}=-\ig JA^{3i}N_3(8\TP^7\eta_p~A^{3p}~\TP\p_3 Q-8\TP^7\eta\cdot\pa\TP\eta_r ~A^{3r}~\p_3Q)\TP^8\eta_i\dyy\\
\lesssim&\left|A^{3i}\TP^8\eta_i\right|_0|J|_{L^{\infty}}(|A^{3p}\TP\p_3 Q|_{L^{\infty}}+|(\pa\TP\eta_r) A^{3r}\p_3Q|_{L^{\infty}})\int_0^T|\TP^7v(t)|_0\dt\\
\lesssim&\left|A^{3i}\TP^8\eta_i\right|_0P(\|\eta\|_{8,*},\|Q\|_{8,*})\int_0^T\|v(t)\|_{8,*}\dt.
\end{aligned}
\end{equation}

In $IB_{32}^*$, we invoke the relation \eqref{qbdry} to get
\begin{align*}
\p_t(J\Delta_Q^*)|_{\Gamma}=&-8\TP^7 v_p\,\hat{A}^{3p}\,\TP\p_3 Q+8\TP^7v \cdot\pa\TP\eta_r~ \hat{A}^{3r}\p_3Q\\
&-8\TP^7\eta_p~\p_t(\hat{A}^{3p}\TP\p_3 Q)+8\TP^7\eta\cdot\p_t(\pa\TP\eta_r~\hat{A}^{3r}\p_3Q)\\
=&-8\TP^7 v_p~\hat{A}^{3p}~\TP\p_3 Q+8\TP^7v \cdot\pa\TP\eta_r~ \hat{A}^{3r}\p_3Q\\
&-8\TP^7\eta_p\p_t\TP(\hat{A}^{3p}\p_3 Q)+8\TP^7\eta_p\p_t(\TP \hat{A}^{3p}~\p_3 Q)+8\TP^7\eta\cdot\p_t(\pa\TP\eta_r~\hat{A}^{3r}~\p_3Q)\\
\overset{\eqref{qbdry}}{=}&-8\TP^7 v_p~\hat{A}^{3p}~\TP\p_3 Q+8\TP^7v \cdot\pa\TP\eta_r~\hat{A}^{3r}\p_3Q\\
&+8\TP^7\eta_p\p_t\TP\left(\rho_0\p_tv^p-(b_0\cdot\TP)(J^{-1}\bp\eta)^{p}\right)+8\TP^7\eta_p~\p_t(\TP \hat{A}^{3p}~\p_3 Q)+8\TP^7\eta\cdot\p_t(\pa\TP\eta_r~\hat{A}^{3r}\p_3Q).
\end{align*}Then we use $H^{\frac32}(\T^2)\hookrightarrow L^{\infty}(\T^2)$, Lemma \ref{trace} and standard Sobolev trace lemma to get
\begin{equation}\label{modifyt}
\begin{aligned}
\left|\p_t(J\Delta_Q^*)|_{\Gamma}\right|_{0}\lesssim&|\TP^7 v|_0\left(|A~\TP\p Q|_{L^{\infty}}+|a~\TP\p\eta~ A~\p Q|_{L^{\infty}}\right)\\
&+|\TP^7\eta_p|_0\left(\left|\rho_0\p_t^2 v^p+\p_t(b_0\cdot\TP)b^{p}\right|_{W^{1,\infty}(\T^2)}+\left|\p_t(\TP \hat{A}^{3p}~\p_3 Q)+\p_t(\pa\TP\eta_r ~\hat{A}^{3r}~\p_3Q)\right|_{L^{\infty}}\right)\\
\lesssim&\|v\|_{8,*}\|\p Q\|_{5,*}P(\|\p\eta\|_3)\\
&+\|\eta\|_{8,*}\left(\|\rho_0\p_t^2 v+\p_t(b_0\cdot\TP)b^{p}\|_{5,*}+\left\|\p_t(\TP \hat{A}^{3p}~\p_3 Q)+\p_t(\pa\TP\eta_r ~\hat{A}^{3r}~\p_3Q)\right\|_2\right)\\
\lesssim& (\|\eta\|_{8,*}+\|v\|_{8,*}+\|b\|_{7,*})(\|Q,\eta,v\|_{7,*},\|b_0,\rho_0\|_3),
\end{aligned}
\end{equation}and thus
\begin{equation}\label{IB*32}
IB_{32}^*\lesssim\int_0^T \left|A^{3i}\TP^8\eta_i\right|_0 P(\|\eta\|_{8,*},\|v\|_{8,*},\|Q\|_{8,*},\|b\|_{8,*},\|\rho_0\|_3)\dt.
\end{equation}

From \eqref{IB*B}, we know it suffices to control the product of ``error part" $IB_5^*$
\[
IB_5^*\lesssim |\hat{A}^{3i}|_{L^{\infty}}|\Delta_Q^*|_{\Gamma}|_0|(\Delta_v^*)_i|_0,
\]and the RHS can be directly controlled by Lemma \ref{trace} and standard trace lemma
\begin{align*}
\left|\Delta_Q^*|_{\Gamma}\right|_0\lesssim&\left|\TP^7\eta_p\right|_0\left(|A^{3p}\TP\p_3 Q|_{L^{\infty}}+|\pa^p\TP\eta_r~A^{3r}\p_3Q|_{L^{\infty}}\right)\lesssim P(\|Q\|_{7,*},\|\eta\|_{7,*})\|\eta\|_{8,*},
\end{align*}
\begin{align*}
\left|\Delta_v^*\right|_0\lesssim&|\TP^7\eta|_0\left(|\pa\TP v|_{L^{\infty}}+|\pa\TP\eta\cdot\pa v|_{L^{\infty}}+|\pa v\cdot\pa \TP\eta|_{L^{\infty}}\right)+|\TP^7v|_0|\pa\TP\eta|_{L^{\infty}}\\
\lesssim & P(\|\eta\|_{7,*})(\|v\|_{8,*}+\|v\|_{7,*}\|\eta\|_{8,*}).
\end{align*}

Therefore,
\begin{equation}\label{IB*5}
IB_{5}^*\lesssim P(\|\eta\|_{8,*},\|v\|_{8,*},\|Q\|_{7,*}),
\end{equation} and similarly
\begin{equation}\label{IB*6}
IB_6^*\lesssim |\hat{A}^{3i}\p_3 Q|_{L^{\infty}}|A^{3p}\TP^8\eta_p|_0|(\Delta_v^*)_i|_0.
\end{equation}

Summarizing \eqref{IB*B}-\eqref{IB*6} gives the control of the boundary integral
\begin{equation}\label{IB*}
\int_0^T IB^*\lesssim -\frac{c_0}{4}\left|A^{3i}\TP^8\eta_i\right|_0^2+\PP_0+ P(\EE(T))\int_0^T P(\EE(t))\dt.
\end{equation}

Combining \eqref{t8energy0}, \eqref{t83}, \eqref{t81}, \eqref{I*} and \eqref{IB*} and choosing $\varepsilon>0$ in \eqref{I*21} to be suitably small, we get the following energy inequality
\begin{equation}\label{t8energy}
\|\VV^*\|_0^2+\left\|\TP^8\left(J^{-1}\bp\eta\right)\right\|_0^2+\|\TP^8 q\|_0^2+\frac{c_0}{4}\left|A^{3i}\TP^8\eta_i\right|_0^2\bigg|_{t=T}\lesssim \PP_0+ P(\EE(T))\int_0^T P(\EE(t))\dt.
\end{equation} Finally, invoking \eqref{t8goodv}, we get the $\TP^8$-estimates of $v$
\begin{align*}
\|\TP^8v\|_0\lesssim&\|\VV^*\|_0+\|\TP^8\eta\|_0\|\pa v_i\|_{L^{\infty}}+\|\TP^7\eta\|_0\left(\|\pa\TP v\|_{L^{\infty}}+\|\pa\TP \eta\cdot\pa v\|_{L^{\infty}}\right)+\|\TP^7v\|_0\|\pa\TP\eta\|_{L^{\infty}}.
\end{align*}
Since $\p^m\eta|_{t=0}=0$ for any $m\geq 2, m\in \N^*$, we know
\begin{equation}\label{t8gap}
\|\TP^8v\|_0\lesssim\|\VV^*\|_0+P(\|v\|_{7,*},\|\eta\|_{7,*})\int_0^T P(\|v\|_{8,*}),
\end{equation}and thus
\begin{equation}\label{t8}
\|\TP^8 v\|_0^2+\left\|\TP^8\left(J^{-1}\bp\eta\right)\right\|_0^2+\|\TP^8 q\|_0^2+\frac{c_0}{4}\left|A^{3i}\TP^8\eta_i\right|_0^2\bigg|_{t=T}\lesssim \PP_0+ P(\EE(T))\int_0^T P(\EE(t))\dt.
\end{equation}

\subsection{The case of one time derivative $\TP^7\p_t$}\label{ts7}
If we replace $\p_*^I=\TP^8$ by $\TP^7\p_t$, then most of steps in the proof above are still applicable because we do not integrate the derivative(s) in $\dd^8$ by parts. However, we still need to do the following modifications due to the presence of time derivative.

\subsubsection{Extra difficulty: non-vanishing initial data of $\p_*^I \eta$}

If $\p_*^I=\TP^7\p_t$, then we can no longer derive $\TP^7\p_t\eta|_{t=0}=\mathbf{0}$ from $\eta|_{t=0}=\text{Id}$ due to the presence of time derivative and $\p_t \eta=v$. This property is used in the analysis of $IB_{33}^*$ and the control of the difference between $\VV^*$ and $\p_*^I v$. Before we analyze the analogues of $IB_{33}^*$ and \eqref{t8gap} in the case of $\p_*^I=\TP^7\p_t$, we have to find out the precise form of the modified Alinhac good unknowns when $\p_*^I=\TP^7\p_t$.

\subsubsection{The modified Alinhac good unknowns}\label{ts7AGU}

Recall the ``extra modification terms" $\Delta_Q^*,\Delta_v^*$ in \eqref{t8goodeq} come from the bad terms \eqref{t8bad}. Now we replace $\TP^8$ by $\TP^7\p_t$. In $e_1,e_2,e_3$ in \eqref{t8bad}, if we replace $\TP^7$ by $\TP^6\p_t$ (i.e., the time derivative falls on the higher order term), then their $L^2$ norms can be directly controlled since $\p_t a$ has the same spatial regularity as $a$. Therefore, the remaining quantities whose $L^2$-norms cannot be directly controlled in the case of $\p_*^I=\TP^7\p_t$ are
\begin{equation}\label{t7bad}
\begin{aligned}
e_1:=-\TP^7(A^{lr}A^{mi})\p_t\p_m\eta_r~\p_l f,&~~~e_2:=-7\p_t(A^{lr}A^{mi})~\TP^7\p_m\eta_r~\p_l f\\
e_3:=8(\TP^7A^{li})\p_t\p_l f,&~~~e_4:=(\p_tA^{li})(\TP^7\p_l f)+7(\TP A^{li})(\TP^6\p_t\p_l f).
\end{aligned}
\end{equation} Then the corresponding Alinhac good unknowns now becomes (with the abuse of terminology)
\begin{equation}\label{t7goodeq}
\VV^*=\TP^7\p_t v-\TP^7\p_t\eta\cdot\pa v+\Delta_v^*,~~\QQ^*=\TP^7\p_t Q-\TP^7\p_t\eta\cdot\pa Q+\Delta_Q^*.,
\end{equation}where
\begin{equation}\label{t7error}
\begin{aligned}
(\Delta_v^*)_i:=&-8\TP^7\eta\cdot\pa\p_t v_i-8\TP^7v\cdot\pa v_i+16\TP^7\eta\cdot\pa v\cdot\pa v_i,\\
\Delta_Q^*:=&-8\TP^7\eta\cdot\pa\p_t Q+8\TP^7\eta\cdot\pa v\cdot\pa Q,
\end{aligned}
\end{equation}and
\begin{equation}\label{t7goodgrad}
\TP^7\p_t(\diva v)=\pa\cdot \VV^*+C^*(v),~~~\TP^7\p_t(\pa Q)=\pa\QQ^*+ C^*(Q),
\end{equation}with
\[
\|C^*(f)\|_0\lesssim P(\EE(t))\|f\|_{8,*}.
\]

Now, the analogue of $IB_{33}^*$ becomes the following quantity (recall such term comes from the product of $\Delta_Q$ and $\TP^7\p_t v$
\begin{equation}\label{IB*330}
\ig JA^{3i}N_3(8\TP^7\eta_p~A^{3p}~\p_t\p_3 Q-8\TP^7\eta\cdot\pa\p_t\eta_r~A^{3r}~\p_3Q)\TP^7\p_t\eta_i\dyy,
\end{equation}and we can still use $\TP^7\eta|_{t=0}=\mathbf{0}$.

The analogue of \eqref{t8gap} now needs some small modifications
\begin{equation}\label{t7gap}
\begin{aligned}
\|\TP^7\p_tv\|_0\lesssim&\|\VV^*\|_0+\|\TP^7 v\|_0\|\pa v\|_{L^{\infty}}+\|\TP^7\eta\|_0\left(8\|\pa\p_t v\|_{L^{\infty}}+16\|\pa v\|_{L^{\infty}}^2\right)\\
\lesssim&\|\VV^*\|_0+\|\TP^7 v\|_0^2+\|\pa v\|_{2}^2+\|\pa\TP v\|_{2}^2+\left(8\|\pa\p_t v\|_{L^{\infty}}+16\|\pa v\|_{L^{\infty}}^2\right)\int_0^T\|\TP^7 v\|_0\dt\\
\lesssim&\|\VV^*\|_0^2+\PP_0+\int_0^TP\left(\|\TP^7\p_t v\|_0,\|\p_t\p v\|_{2},\|\p_t\p\TP v\|_{2},\|\p\TP\eta\|_{2}^2\right)+P(\EE(T))\int_0^T P(\EE(t))\dt\\
\lesssim&\|\VV^*\|_0^2+\PP_0+P(\EE(T))\int_0^T P(\EE(t))\dt.
\end{aligned}
\end{equation}

The remaining analysis should follow in the same way as in Section \ref{ts8energy}, so we omit those details. The result is
\begin{equation}\label{t7}
\|\TP^7\p_t v\|_0^2+\left\|\TP^7\p_t\left(J^{-1}\bp\eta\right)\right\|_0^2+\|\TP^7\p_t q\|_0^2+\frac{c_0}{4}\left|A^{3i}\TP^7\p_t\eta_i\right|_0^2\bigg|_{t=T}\lesssim \PP_0+ P(\EE(T))\int_0^T P(\EE(t))\dt.
\end{equation}

\subsection{The case of 2$\sim$7 time derivatives}\label{ts6}

If the number of time derivatives in $\p_*^I$ is between 2 and 7, i.e,. $\p_*^I$ contains at least one spatial and two time derivatives, we can still mimic most steps in Section \ref{ts7}. In this case we write $\p_*^I=\dd^6\p_t^2$ where $\dd=\TP$ or $\p_t$ and $\dd^6$ contains at least one $\TP$.

The extra time derivatives allow us to eliminate most of the ``extra modification terms" in the modified Alinhac good unknowns as in \eqref{t8goodeq}, \eqref{t7goodeq}-\eqref{t7error} and thus much simplify the analysis of Alinhac good unknowns and tne boundary control. The reason is that the $L^2$-norm of the analogues of $e_1\sim e_3$ in \eqref{t8bad} can be directly controlled in the case of $\dd^8=\dd^6\p_t^2$. In specific, we have
\begin{equation}\label{t6bad}
\begin{aligned}
\dd^6\p_t^2(\pa^{i}f)=&\pa^i(\dd^6\p_t^2 f)+(\dd^6\p_t^2A^{li})\p_l f+[\dd^6\p_t^2,A^{li},\p_l f] \\
=&\pa^{i}(\dd^6\p_t^2f)-\dd^6\p_t(A^{lr}~\p_t\p_{m}\eta_{r}~A^{mi})\p_l f+[\dd^6\p_t^2,A^{li},\p_{l} f] \\
=&\pa^{i}(\dd^6\p_t^2 f-\dd^6\p_t^2\eta_{r}~A^{lr}\p_l f)+\underbrace{\dd^6\p_t^2\eta_{r}~\pa^{i}(\pa^{r} f)-([\dd^6\p_t,A^{lr}A^{mi}]\p_t\p_{m}\eta_{r})\p_l f}_{C_0(f)}+[\dd^6\p_t^2,A^{li},\p_{l} f]
\end{aligned}
\end{equation}and
\[
\|C_0(f)\|_0\lesssim P(\|\eta\|_{8,*},\|v\|_{8,*})\|f\|_{8,*}.
\]

Therefore, the analogous analysis of $C_1,C_3\sim C_6$ in Section \ref{ts8} are no longer needed here. The only problematic term is $-2(\p_tA^{li})(\dd^6\p_t\p_l f)-6(\dd A^{li})(\dd^5\p_t^2\p_l f)$ which comes from $[\dd^6\p_t^2,A^{li},\p_{l} f]$. By mimicing the treatment of $C_2(Q)$ and $C_2(v)$ in \eqref{t8C2v}-\eqref{t8C2Q}, we can define the modified Alinhac good unknowns in the case of $\p_*^I=\p_t^N\TP^{8-N}~(2\leq N\leq 7)$ as the following
\begin{equation}\label{t6goodeq}
\VV^*=\dd^6\p_t^2 v-\dd^6\p_t^2\eta\cdot\pa v+\Delta_v^*,~~\QQ^*=\dd^6\p_t^2 Q-\dd^6\p_t^2\eta\cdot\pa Q,
\end{equation}where
\begin{equation}\label{t6error}
(\Delta_v^*)_i:=-6\dd^5\p_t^2v\cdot\pa\dd\eta_i-2\dd^6\p_tv\cdot\pa v_i
\end{equation}and
\begin{equation}\label{t6goodgrad}
\dd^6\p_t^2(\diva v)=\pa\cdot \VV^*+C^*(v),~~~\dd^6\p_t^2(\pa Q)=\pa\QQ^*+C^*(Q),
\end{equation}with
\[
\|C^*(f)\|_0\lesssim P(\EE(t))\|f\|_{8,*}.
\]

In this case, $\Delta_Q^*=0$, and thus the boundary integrals $IB_3^*,IB_4^*,IB_5^*$ all vanish. The analogues of $IB_1^*,IB_2^*,IB_6^*$ in this case can still be controlled in the same way as in Section \ref{ts8energy}. In the control of the difference between $\VV^*$ and $\dd^6\p_t^2$, we have by \eqref{t6goodeq}-\eqref{t6error} that
\begin{equation}\label{t6gap}
\begin{aligned}
\|\dd^6\p_t^2 v\|_0\lesssim&\|\VV^*\|_0+\|\dd^6\p_t v\|_0\|\pa v\|_{L^{\infty}}+\|\dd^5\p_t^2 v\|_0\|\pa\dd\eta\|_{L^{\infty}}\\
\lesssim&\|\VV^*\|_0+\|\dd^6\p_t v\|_0^2+\|\pa v\|_{2}^2+\|\dd^5\p_t^2 v\|_0^2+\|\pa\dd\eta\|_{2}^2\\
\lesssim&\|\VV^*\|_0+\PP_0+\int_0^TP\left(\|\dd^6\p_t^2 v\|_0,\|\dd^5\p_t^3 v\|_0,\|\p\dd v\|_2\right)\dt\lesssim\|\VV^*\|_0+\PP_0+\int_0^T P(\EE(t))\dt
\end{aligned}
\end{equation}

The remaining analysis should follow in the same way as in Section \ref{ts8energy} and \ref{ts7} so we omit the details. The result is
\begin{equation}\label{t6}
\|\dd^6\p_t^2 v\|_0^2+\left\|\dd^6\p_t^2\left(J^{-1}\bp\eta\right)\right\|_0^2+\|\dd^6\p_t^2 q\|_0^2+\frac{c_0}{4}\left|A^{3i}\dd^6\p_t^2\eta_i\right|_0^2\bigg|_{t=T}\lesssim \PP_0+ P(\EE(T))\int_0^T P(\EE(t))\dt,
\end{equation}where $\dd^6$ contains at least one spatial derivative $\TP$.

\subsection{The case of full time derivatives}\label{ts0}

In the case of full time derivatives, the modified Alinhac good unknown is still defined similarly as in \eqref{t6goodeq}-\eqref{t6goodgrad}:\begin{equation}\label{t0goodeq}
\VV^*=\p_t^8 v-\p_t^8\eta\cdot\pa v+\Delta_v^*,~~\QQ^*=\p_t^8 Q-\p_t^8\eta\cdot\pa Q,
\end{equation}where
\begin{equation}\label{t0error}
(\Delta_v^*)_i:=-8\p_t^7v\cdot\pa v_i
\end{equation}and
\begin{equation}\label{t0goodgrad}
\p_t^8(\diva v)=\pa\cdot \VV^*+C^*(v),~~~\p_t^8(\pa Q)=\pa\QQ^*+C^*(Q),
\end{equation}with
\[
\|C^*(f)\|_0\lesssim P(\EE(t))\|f\|_{8,*}.
\]

\paragraph*{Extra difficulty: trace lemma is no longer applicable}

When $\p_*^I=\p_t^8$, there are terms of the form $\p_t^7 v$ in the boundary integrals. In the case of full time derivative, one cannot apply Lemma \ref{trace} to control $|\p_t^7 v|_0$. This difficulty appears in the estimates of the analogue of $IB_6^*$.
 Instead, we need to write $IB_6^*$ in terms of interior integrals by using a similar technique in \eqref{n3bdry2}.
\begin{equation}\label{IB*60}
\begin{aligned}
IB_6^*=&-8\ig \hat{A}^{3i}N_3~\p_t^8\eta_p~A^{3p}\p_3 Q~\p_t^7v_r~A^{lr}\p_l v_i\dyy=-8\ig \hat{A}^{3i}N_3~\p_t^7 v_p~A^{3p}\p_3 Q~\p_t^7v_r~A^{lr}\p_l v_i\dyy\\
=&-8\io \hat{A}^{3i}\p_3\p_t^7 v_p~A^{3p}\p_3 Q~\p_t^7v_r~A^{lr}\p_l v_i\dy-8\io \hat{A}^{3i}\p_t^7 v_p~A^{3p}\p_3 Q~\p_3\p_t^7v_r~A^{lr}\p_l v_i\dy\\
&-8\io\p_t^7 v_p~\p_t^7 v_r~\p_3(\hat{A}^{3i}A^{3p}\p_3 Q~A^{lr}\p_l v_i)\dy\\
=:&IB_{61}^*+IB_{62}^*+IB_{63}^*.
\end{aligned}
\end{equation}

The term $IB_{63}^*$ can be directly controlled
\begin{equation}\label{IB*63}
IB_{63}^*\lesssim P(\|\p_t^7 v\|_{0},\|\p v\|_3,\|\p Q\|_3,\|A\|_3)\lesssim P(\|v\|_{8,*},\|Q\|_{8,*},\|\eta\|_{4}).
\end{equation}

The term $IB_{61}^*,IB_{62}^*$ should be controlled by integrating $\p_t$ by parts under time integral.
\begin{equation}\label{IB*61}
\begin{aligned}
\int_0^TIB_{61}^*=&-8\int_0^T\io \hat{A}^{3i}\p_3\p_t^7 v_p~A^{3p}\p_3 Q~\p_t^7v_r~A^{lr}\p_l v_i\dy\dt\\
\overset{\p_t}{=}&-8\io \hat{A}^{3i}\p_3\p_t^6 v_p~A^{3p}\p_3 Q~\p_t^7v_r~A^{lr}\p_l v_i\dy\\
&+8\int_0^T\io \hat{A}^{3i}\p_3\p_t^6 v_p~A^{3p}\p_3 Q~\p_t^8v_r~A^{lr}\p_l v_i\dy\dt\\
&+8\int_0^T\io \hat{A}^{3i}\p_3\p_t^6 v_p~\p_t^7v_r~\p_t(A^{3p}\p_3 Q~A^{lr}\p_l v_i)\dy\dt\\
\lesssim &\|\p_3\p_t^6 v\|_0\|\p_t^7 v\|_0P(\|A\|_{L^{\infty}},\|\p Q\|_{L^{\infty}},\|\p v\|_{L^{\infty}})\\
&+\int_0^T \|\p_3\p_t^6 v\|_0\left(\|\p_t^8 v\|_0 P(\|A\|_{L^{\infty}},\|\p Q\|_{L^{\infty}},\|\p v\|_{L^{\infty}})+\|\p_t^7 v\|_0\|A\cdot\p_t(A\cdot\p Q\cdot A\cdot\p v)\|_{L^{\infty}}\right)\\
\lesssim&\varepsilon\|\p_3\p_t^6 v\|_0^2 +\PP_0+\int_0^T P\left(\|\p_t^8 v\|_0,\|(\p_t A,\p_t\p Q,\p_t\p v\|_{L^{\infty}}^2\right)+\int_0^T P(\|v\|_{8,*},\|Q\|_{7,*},\|\eta\|_{3})\dt\\
\leq&\varepsilon\|\p_3\p_t^6 v\|_0^2 +\PP_0+\int_0^T P(\EE(t))\dt,
\end{aligned}
\end{equation}

$IB_{62}^*$ can be controlled in the same way, so we omit the details. Summarizing the estimates above, we get the energy inequality of the full time derivatives
\begin{equation}\label{t0}
\|\p_t^8 v\|_0^2+\left\|\p_t^8 \left(J^{-1}\bp\eta\right)\right\|_0^2+\|\p_t^8  q\|_0^2+\frac{c_0}{4}\left|A^{3i}\p_t^8 \eta_i\right|_0^2\bigg|_{t=T}\lesssim \varepsilon\|\p_3\p_t^6 v\|_0^2+\PP_0+ P(\EE(T))\int_0^T P(\EE(t))\dt,
\end{equation}which together with \eqref{t8}, \eqref{t7}, \eqref{t6} concludes the proof of Proposition \ref{prop t8}.

\section{Control of mixed non-weighted derivatives}\label{mixed}

The case of mixed non-weighted derivatives correspond to $\p_*^I=\p_t^{i_0}(\sigma\p_3)^{i_4}\TP_{1}^{i_1}\TP_2^{i_2}\p_3^{i_3}$ with $1\leq i_3\leq 3,~i_4=0$. In this case, the modified Alinhac good unknowns introduced in Section \ref{sect t8} are still needed when commuting $\p_*^I$ with $\pa$. On the other hand, the highest order term $\p_*^I Q$ no longer vanishes on the boundary due to the presence of normal derivatives, so we need to use the method in Section \ref{sect n4} to deal with the boundary integral. Therefore, we should combine the methods in Section \ref{sect n4} and Section \ref{sect t8} to get the control of mixed non-weighted derivatives. The result of this section is
\begin{prop}\label{prop mixed}
The following energy inequality holds for sufficiently small $\eps>0$
\begin{equation}\label{tntn}
\sum_{1\leq i_3\leq 3,~i_4=0}\|\p_*^I v\|_0^2+\left\|\p_*^I\left(J^{-1}\bp\eta\right)\right\|_0^2+\|\p_*^I q\|_0^2+\frac{c_0}{4}\left|A^{3i}\p_*^I \eta_i\right|_0^2\bigg|_{t=T}\lesssim\varepsilon\|\p_3^4 v\|_0^2+\PP_0+ P(\EE(T))\int_0^T P(\EE(t))\dt.
\end{equation}
\end{prop}

\subsection{Purely spatial derivatives}\label{sect tn}

We still start with the control of purely spatial derivatives. Let $N=1,2,3$ and we consider $\p_*^I=\p_3^N\TP^{8-2N}$.

\subsubsection{The modified Alinhac good unknowns}

Similarly as in Section \ref{ts8AGU}, we have
\begin{equation}\label{tnproof1}
\begin{aligned}
\p_3^N\TP^{8-2N}(\pa^{i}f)=&\pa^i(\p_3^N\TP^{8-2N} f)+(\p_3^N\TP^{8-2N}A^{li})\p_l f+[\p_3^N\TP^{8-2N},A^{li},\p_l f] \\
=&\pa^{i}(\p_3^N\TP^{8-2N} f)-\p_3^N\TP^{7-2N}(A^{lr}~\TP\p_{m}\eta_{r}~A^{mi})\p_l f+[\p_3^N\TP^{8-2N},A^{li},\p_{l} f] \\
=&\pa^{i}(\p_3^N\TP^{8-2N} f-\p_3^N\TP^{8-2N}\eta_{r}~A^{lr}~\p_l f)+(\p_3^N\TP^{8-2N}\eta_{r})\pa^{i}(\pa^{r} f)\\
&-([\p_3^N\TP^{7-2N},A^{lr}A^{mi}]\TP\p_{m}\eta_{r})\p_l f+[\p_3^N\TP^{8-2N},A^{li},\p_{l} f],
\end{aligned}
\end{equation}
where the last line still contains the terms whose $L^2(\Omega)$-norms cannot be directly bounded under the setting of anisotropic Sobolev space $H_*^8(\Omega)$. The reason is that $\p_3^N\TP^{7-2N}$ may fall on $A=\p\eta\times\p\eta$ and $\p_l f$. The following quantities are exactly these terms.
\begin{equation}\label{tnbad}
\begin{aligned}
e_1^\sh:=-\p_3^N\TP^{7-2N}(A^{lr}A^{mi})(\TP\p_m\eta_r)\p_l f,&~~~e_2^\sh:=-(7-2N)\TP(A^{lr}A^{mi})(\p_3^N\TP^{7-2N}\p_m\eta_r)\p_l f,\\
e_3^\sh:=(8-2N)(\p_3^N\TP^{7-2N}A^{li})(\TP\p_l f),&~~~e_4^\sh:=(8-2N)(\TP A^{li})(\p_3^N\TP^{7-2N}\p_l f).
\end{aligned}
\end{equation}

One can mimic the derivation of \eqref{t8goodv} and \eqref{t8goodq} to define the ``modified Alinhac good unknowns" of $v$ and $Q$ with respect to $\p_3^N\TP^{8-2N}$ to be
\begin{equation}\label{tngoodv}
\begin{aligned}
\VV_i^\sh:=&\p_3^N\TP^{8-2N}v_i-\p_3^N\TP^{8-2N}\eta\cdot\pa v_i\\
&-(8-2N)\p_3^N\TP^{7-2N}\eta\cdot\pa\TP v_i-(8-2N)\p_3^N\TP^{7-2N}v\cdot\pa\TP\eta_i\\
&+(8-2N)\p_3^N\TP^{7-2N}\eta\cdot\pa\TP\eta\cdot\pa v_i+(8-2N)\p_3^N\TP^{7-2N}\eta\cdot\pa v\cdot\pa \TP\eta_i,
\end{aligned}
\end{equation}and
\begin{equation}\label{tngoodq}
\begin{aligned}
\QQ^\sh:=&\p_3^N\TP^{8-2N}Q-\p_3^N\TP^{8-2N}\eta\cdot\pa Q\\
&-(8-2N)\p_3^N\TP^{7-2N}\eta\cdot\pa\TP Q+(8-2N)\p_3^N\TP^{7-2N}\eta\cdot\pa\TP\eta\cdot\pa Q.
\end{aligned}
\end{equation} Then $\VV^\sh$ and $\QQ^\sh$ satisfy the following relations
\begin{equation}\label{tngoodgrad}
\p_3^N\TP^{8-2N}(\diva v)=\pa\cdot \VV^\sh+C^\sh(v),~~~\p_3^N\TP^{8-2N}(\pa Q)=\pa\QQ^*+ C^\sh(Q),
\end{equation}where the commutator $C^\sh$ satisfies
\begin{equation}\label{tnC0}
\|C^\sh(f)\|_{0}\lesssim P(\EE(t))\|f\|_{8,*}.
\end{equation}

Denote $\Delta_v^\sh$ and $\Delta_Q^\sh$ to be
\begin{align*}
(\Delta_v^\sh)_i:=&-(8-2N)\p_3^N\TP^{7-2N}\eta\cdot\pa\TP v_i-(8-2N)\p_3^N\TP^{7-2N}v\cdot\pa\TP\eta_i\\
&+(8-2N)\p_3^N\TP^{7-2N}\eta\cdot\pa\TP\eta\cdot\pa v_i+(8-2N)\p_3^N\TP^{7-2N}\eta\cdot\pa v\cdot\pa \TP\eta_i,\\
\Delta_Q^\sh:=&-(8-2N)\p_3^N\TP^{7-2N}\eta\cdot\pa\TP Q+(8-2N)\p_3^N\TP^{7-2N}\eta\cdot\pa\TP\eta\cdot\pa Q.
\end{align*} Then we can derive the evolution equation of $\VV^\sh$ and $\QQ^\sh$
\begin{equation}\label{tngoodeq}
\begin{aligned}
&R\p_t\VV^\sh-J^{-1}\bp\p_3^N\TP^{8-2N}\left(J^{-1}\bp\eta\right)+\pa \QQ^\sh\\
=& [R,\p_3^N\TP^{8-2N}]\p_t v+\left[J^{-1}\bp,\p_3^N\TP^{8-2N}\right]\left(J^{-1}\bp\eta\right)\\
&+C^\sh(Q)+R\p_t(-\p_3^N\TP^{8-2N}\eta\cdot\pa v+\Delta_v^\sh).
\end{aligned}
\end{equation}

Denote the RHS of \eqref{tngoodeq} to be $\FF^\sh$, then direct computation yields that
\[
\|\FF^\sh\|_0\lesssim P\left(\|\eta\|_{8,*},\|v\|_{8,*},\|Q\|_{8,*}\right).
\]

Now we take $L^2(\Omega)$ inner product of \eqref{tngoodeq} and $J\VV^\sh$ to get the following energy identity
\begin{equation}\label{tngoodenergy0}
\begin{aligned}
\frac{1}{2}\ddt\io\rho_0|\VV^\sh|^2\dy=\io\bp\p_3^N\TP^{8-2N}\left(J^{-1}\bp\eta\right)\cdot\VV^{\sh}-\io (\pA Q)\cdot\VV^\sh+\io J\FF^\sh\cdot\VV^\sh.
\end{aligned}
\end{equation}

\subsubsection{Interior estimates}

The last integral on RHS of \eqref{tngoodenergy0} is directly controlled
\begin{equation}\label{tn3}
\io J\FF^\sh\cdot\VV^\sh\lesssim\io \|\FF^\sh\|_0\|\VV^\sh\|_0.
\end{equation}

Then for the first term on RHS of \eqref{tngoodenergy0} we integrate $\bp$ by parts to produce the energy of magnetic field. Again, there is one term which cannot be directly controlled but will cancel with another term produced by $-\io (\pA Q)\cdot\VV^\sh$. The proof follows in the same way as \eqref{n41} so we omit the details.
\begin{equation}\label{tn1}
\begin{aligned}
&\io\bp\p_3^N\TP^{8-2N}\left(J^{-1}\bp\eta\right)\cdot\VV^{\sh}\dy\\
\lesssim&-\frac{1}{2}\ddt\io J\left|\p_3^N\TP^{8-2N}(J^{-1}\bp\eta)\right|^2\dy+K_{11}^\sh+P\left(\left\|(\eta,v,b_0,\bp)\right\|_{8,*}\right),
\end{aligned}
\end{equation}where
\begin{equation}\label{tnK1}
K_{11}^\sh:=-\io  J\p_3^N\TP^{8-2N}(J^{-1}\bp\eta)\cdot\left(J^{-1}\bp\eta\right)\p_3^N\TP^{8-2N}(\diva v)\dy.
\end{equation}

Next we analyze the term $-\io (\pA Q)\cdot\VV^\sh$. Integrating by parts and using Piola's identity $\p_l\hat{A}^{li}=0$, we get
\begin{equation}\label{tn20}
-\io (\pA Q)\cdot\VV^\sh\dy=\io J\QQ^\sh(\nabla_A\cdot \VV^\sh)\dy-\ig J\QQ^\sh A^{li}N_l\VV_i^\sh\dyy=:I^\sh+IB^\sh.
\end{equation}

Plugging \eqref{tngoodgrad} into $I^\sh$, we get
\begin{equation}\label{Ish0}
\begin{aligned}
I^\sh=&\io J\p_3^N\TP^{8-2N} q \p_3^N\TP^{8-2N}(\diva v)+\io J\p_3^N\TP^{8-2N}\left(\frac12\left|J^{-1}\bp\eta\right|^2\right)\p_3^N\TP^{8-2N}(\diva v)\\
&+\io \left(-(\p_3^N\TP^{8-2N}\eta_p)\hat{A}^{lp}~\p_l Q+\Delta_Q^\sh\right)~\p_3^N\TP^{8-2N}(\diva v)-\io(\p_3^N\TP^{8-2N} Q) C^\sh(v)\\
=:&I_1^\sh+I_2^\sh+I_3^\sh+I_4^\sh,
\end{aligned}
\end{equation}where $I_4^\sh$ can be directly controlled by using the estimates of $C^\sh(v)$
\begin{equation}\label{Ish4}
I_4^\sh\lesssim\|\p_3^N\TP^{8-2N} Q\|_0\|C^\sh(v)\|_0\lesssim P(\|\eta\|_{8,*})\|\p_3^N\TP^{8-2N} Q\|_0\|v\|_{8,*}.
\end{equation}

The term $I_1^\sh$ produces the energy of fluid pressure
\begin{equation}\label{Ish1}
I_1^\sh\lesssim-\frac12\ddt\io\frac{J^2R'(q)}{\rho_0}\left|\p_3^N\TP^{8-2N} q\right|^2\dy+P(\|q\|_{8,*},\|\rho_0\|_{8,*},\|\eta\|_{8,*}).
\end{equation}

Similarly as in \eqref{I*2}, the term $I_2^\sh$ produces the cancellation with $K_{11}^\sh$.
{\small\begin{equation}\label{Ish2}
\begin{aligned}
I_2^\sh=&\underbrace{\io J\p_3^N\TP^{8-2N}\left(J^{-1}\bp\eta\right) \cdot\left(J^{-1}\bp\eta\right)\p_3^N\TP^{8-2N}(\diva v)}_{\text{exactly cancel with }K_{11}^\sh}\dy\\
&+\sum_{1\leq N_1+N_2=8}\binom{N}{N_1}\binom{8-2N}{N_2}\io J\p_3^{N_1}\TP^{N_2}\left(J^{-1}\bp\eta\right) \cdot\p_3^{N-N_1}\TP^{8-2N-N_2}\left(J^{-1}\bp\eta\right)\p_3^N\TP^{8-2N}(\diva v)\dy\\
=&-K_{11}^\sh\\
&-\sum_{1\leq N_1+N_2=8}\binom{N}{N_1}\binom{8-2N}{N_2}\io\frac{J^2R'(q)}{\rho_0}\p_3^{N_1}\TP^{N_2}\left(J^{-1}\bp\eta\right) \cdot\p_3^{N-N_1}\TP^{8-2N-N_2}\left(J^{-1}\bp\eta\right)(\p_3^N\TP^{8-2N}\p_t q)\dy\\
&-\sum_{1\leq N_1+N_2=8}\binom{N}{N_1}\binom{8-2N}{N_2}\io J\p_3^{N_1}\TP^{N_2}\left(J^{-1}\bp\eta\right) \cdot\p_3^{N-N_1}\TP^{8-2N-N_2}\left(J^{-1}\bp\eta\right)\left(\left[\TP^8,\frac{JR'(q)}{\rho_0}\right]\p_t q\right)\dy\\
=:&-K_{11}^\sh+I_{21}^\sh+I_{22}^\sh
\end{aligned}
\end{equation}}and by direct computation we have
\begin{align}
\label{Ish21} \int_0^T I_{21}^\sh\lesssim&\varepsilon\|\p_3^N\TP^{8-2N} q\|_0^2+\PP_0+\int_0^TP(\EE(t))\dt\\
\label{Ish22} I_{22}^\sh\lesssim&\|J^{-1}\bp\eta\|_{7,*}^2\|q\|_{8,*}.
\end{align}

Then $I_3^\sh$ can be controlled by integrating $\p_t$ by parts under time integral after invoking $\text{div}_Av=-\frac{JR'(q)}{\rho_0}\p_t q$. The proof is similar to \eqref{I*3} so we do not repeat the proof.
\begin{equation}\label{Ish3}
\begin{aligned}
\int_0^TI_3^\sh\lesssim\eps\|\TP^{8-2N}\p_3^N q\|_0^2+\PP_0+\int_0^T P(\EE(t))\dt.
\end{aligned}
\end{equation}

Summarizing \eqref{Ish0}-\eqref{Ish3} and choosing $\eps>0$ sufficiently small, we get the interior estimates
\begin{equation}\label{Ish}
\int_0^TI^\sh\dt\lesssim-\frac{1}{2}\io\frac{J^2R'(q)}{\rho_0}\left|\p_3^N\TP^{8-2N} q\right|^2\dy\bigg|^T_0+\PP_0+\int_0^T P(\EE(t))\dt.
\end{equation}Therefore, it suffices to analyze the boundary integral $IB^\sh$.

\subsubsection{Boundary estimates}\label{sect tnbdry}

Invoking \eqref{tngoodv}-\eqref{tngoodq}, the boundary integral now reads
\begin{equation}\label{IBshB}
\begin{aligned}
IB^\sh=-\ig \QQ^\sh JA^{3i}N_3\VV_i^\sh\dyy=&-\ig JA^{3i}N_3(\p_3^N\TP^{8-2N}Q)\VV_i^\sh\dyy\\
&+\ig \hat{A}^{3i}N_3(\p_3^N\TP^{8-2N}\eta_p)A^{3p}\p_3 Q~\p_3^N\TP^{8-2N} v_i\dyy\\
&-\ig \hat{A}^{3i}N_3(\p_3^N\TP^{8-2N}\eta_p~A^{3p}\p_3 Q)(\p_3^N\TP^{8-2N}\eta\cdot\pa v_i)\dyy\\
&-\ig \hat{A}^{3i}N_3(\Delta_Q^\sh)(\p_3^N\TP^{8-2N} v_i)\dyy+\ig \hat{A}^{3i}N_3(\Delta_Q^\sh)\p_3^N\TP^{8-2N}\eta\cdot\pa v_i\dyy\\
&-\ig  \hat{A}^{3i}N_3\Delta_Q^\sh(\Delta_v^\sh)_i\dyy+\ig \hat{A}^{3i}N_3(\p_3^N\TP^{8-2N}\eta_p~A^{3p}~\p_3 Q)(\Delta_v^\sh)_i\dyy\\
=:&IB_0^\sh+IB_1^\sh+IB_2^\sh+IB_3^\sh+IB_4^\sh+IB_5^\sh+IB_6^\sh.
\end{aligned}
\end{equation}

To control $IB^\sh$, we only need to combine the techniques used in Section \ref{sect n4bdry} and Section \ref{sect t8bdry}:
\begin{itemize}
\setlength{\itemsep}{0pt}
\setlength{\parsep}{0pt}
\setlength{\parskip}{0pt}
\item  $IB_1^\sh,IB_2^\sh$ together with the Rayleigh-Taylor sign condition produces the boundary energy $|A^{3i}\p_3^N\TP^{8-2N}\eta_i|_0^2$, similarly as $IB_1+IB_2$ in Section \ref{sect n4bdry} and $IB_1^*+IB_2^*$ Section \ref{sect t8bdry}.
\item  The term $IB_0^\sh$ is the analogue of $IB_0$ in Section \ref{sect n4bdry} and can be controlled with similar method as in Section \ref{sect n4bdry}.
\item  $IB_3^\sh\sim IB_6^\sh$ are the analogues of $IB_3^*\sim IB_6^*$ in Section \ref{sect t8bdry}. These terms can be controlled exactly in the same way as $IB_3^*\sim IB_6^*$.
\end{itemize}

First, $IB_1^\sh$ and $IB_2^\sh$ give the boundary energy with the help of Rayleigh-Taylor sign condition. The proof is exactly the same as in Section \ref{sect n4bdry} and Section \ref{sect t8bdry} by merely replacing $\p_3^4$ or $\TP^8$ with $\p_3^N\TP^{8-2N}$, so we do not repeat the computations here.
\begin{equation}\label{IBsh12}
\begin{aligned}
\int_0^TIB_1^\sh+IB_2^{\sh}=&-\frac{c_0}{4}\left|A^{3i}\p_3^N\TP^{8-2N}\eta_i\right|_0^2\bigg|^T_0 +\int_0^T P(\EE(t))\dt.
\end{aligned}
\end{equation}


We then analyze $IB_0^\sh$. Invoking \eqref{tngoodv}, we have
\begin{equation}\label{IBsh00}
\begin{aligned}
IB_0^\sh=&-\ig N_3(J\p_3^N\TP^{8-2N}Q)(A^{3i}~\p_3^N\TP^{8-2N} v_i)\dyy+\ig JA^{3i}N_3(\p_3^N\TP^{8-2N}Q)(\p_3^N\TP^{8-2N}\eta\cdot\pa v_i)\dyy\\
&-\ig JA^{3i}N_3(\p_3^N\TP^{8-2N}Q)(\Delta_v^\sh)_i\dyy\\
=:&IB_{01}^\sh+IB_{02}^\sh+IB_{03}^\sh.
\end{aligned}
\end{equation} We note that $IB_{01}^\sh$ and $IB_{02}^{\sh}$ are the analogues of $IB_{01}$ and $IB_{02}$ in Section \ref{sect n4bdry}, so we do not repeat all the details here. The extra term $IB_{03}^{\sh}$ can be directly controlled (cf. \eqref{tnbdry2} below).

We differentiate the continuity equation \eqref{vbdry} by $\p_3^N\TP^{8-2N}$ to simplify the top order term containing $v$ in $IB_{01}^{\sh}$:
\begin{equation}\label{tnvbdry}
\begin{aligned}
A^{3i}\p_3^N\TP^{8-2N}v_i=&-\p_3^{N-1}\TP^{8-2N}\left(\frac{JR'(q)}{\rho_0}\p_t q\right)-\sum_{L=1}^2\p_3^{N-1}\TP^{8-2N}(A^{Li}\TP_Lv_i)\\
&-\sum_{N_1+N_2\geq 1,N_1\leq N-1}\binom{N-1}{N_1}\binom{8-2N}{N_2}\left(\p_3^{N_1}\TP^{N_2}A^{3i}\right)\left(\p_3^{N-N_1}\TP^{8-2N-N_2}v_i\right),
\end{aligned}
\end{equation}
where the term contains $\p_3^{N-1}\TP^{8-2N}A^{Li}=\p_3^{N}\TP^{8-2N}\eta\times\TP\eta+$ controllable terms, where $\p_3^{N}\TP^{8-2N}\eta$ cannot be controlled on the boundary. Invoking \eqref{Da} with $D=\TP$, we expand this problematic term to be
\begin{equation}\label{tncancel1}
\begin{aligned}
(\p_3^{N-1}\TP^{8-2N}A^{Li})\TP_Lv_i=&-\left(\p_3^{N-1}\TP^{7-2N}(A^{Lp}~\TP\p_m\eta_p~A^{mi})\right)\TP_Lv_i\\
=&-A^{Lp}~\p_3^{N}\TP^{8-2N}\eta_p~A^{3i}\TP_Lv_i-\sum_{M=1}^2A^{Lp}(\p_3^{N-1}\TP^{8-2N}\TP_M\eta_p) A^{Mi}\TP_Lv_i-([\p_3^{N-1}\TP^{7-2N},A^{Lp}A^{mi}]\TP\p_m\eta_p)\TP_Lv_i.
\end{aligned}
\end{equation}

On the other hand, in $IB_{02}^\sh$ we have
\begin{equation}\label{tncancel2}
\begin{aligned}
A^{3i}\p_3^{N}\TP^{8-2N}\eta\cdot\pa v_i=A^{3i}\sum_{L=1}^2\p_3^{N}\TP^{8-2N}\eta_p A^{Lp}\TP_L v_i+A^{3i}\p_3^{N}\TP^{8-2N}\eta_p A^{3p}\p_3 v_i,
\end{aligned}
\end{equation}where the first term exactly cancels with the first term in the RHS of \eqref{tncancel1}. In fact, this is the analogue of \eqref{IB0121}-\eqref{IB022} by merely replacing $\p_3^4$ with $\p_3^N\TP^{8-2N}$. Thus we get the cancellation of the top order terms in $IB_{01}^\sh$ and $IB_{02}^\sh$.

The second term in \eqref{tncancel2} could be treated similarly as in \eqref{n4etabdry} by invoking $A^{3p}\p_3\eta_p=1$
\begin{equation}\label{tncancel3}
\begin{aligned}
\p_3^{N}\TP^{8-2N}\eta_p A^{3p}=&\underbrace{\p_3^{N-1}\TP^{8-2N}(\p_3\eta_p A^{3p})}_{=0}-(\p_3^{N-1}\TP^{8-2N}A^{3p})\p_3\eta_p-(\TP A^{3p})(\p_3^N\TP^{7-2N}\eta_p)+\text{lower order terms}.
\end{aligned}
\end{equation}



To control $IB_0^\sh$, we still need to analyze $\p_3^{N}\TP^{8-2N}Q$. Following the aruments in \eqref{qbdry}-\eqref{n4qbdry2} and replacing $\p_3^4$ with $\p_3^N\TP^{8-2N}$, we can reduce one normal derivative of $Q$ to one tangential derivative of $v$ and $\bp\eta$ via
\begin{equation}\label{tnqbdry2}
\begin{aligned}
\p_3^N\TP^{8-2N}Q=&J^{-1}\p_3\eta_i\left(\rho_0\p_3^{N-1}\TP^{8-2N}\p_t v^i+(b_0\cdot\TP)\p_3^{N-1}\TP^{8-2N}(J^{-1}\bp\eta^i)\right)-\sum_{L=1}^2\hat{A}^{Li}(\p_3^{N-1}\TP^{8-2N}\TP_L Q)\\
&-(N-1)(\p_3^{N-1}\TP^{8-2N}\hat{A}^{3i})(\p_3 Q)+\text{ lower order terms. }
\end{aligned}
\end{equation}
Plugging the expression of $\Delta_v^\sh$ and \eqref{tnvbdry}-\eqref{tnqbdry2} into \eqref{IBsh00}, we find that every highest order term in $IB_0^\sh$ must be one of the following forms
\begin{align*}
K_1^{\sh}:=&\ig N_3(\p_3^{N-1}\TP^{8-2N}\dd f)(\p_3^{N-1}\TP^{9-2N} g)( \p h) r \dyy,\\
K_2^{\sh}:=&\ig N_3(\p_3^{N-1}\TP^{8-2N}\dd f)(\p_3^{N}\TP^{7-2N} g)(\p\TP h) r \dyy,\\
K_3^{\sh}:=&\ig N_3(\p_3^{N-1}\TP^{8-2N}\dd f)(\p_3^{N}\TP^{7-2N} g)(\p h) r \dyy,
\end{align*} where $\dd=\TP$ or $\p_t$ or $(b_0\cdot\TP)$, the functions $f,g,h$ can be $\eta,v,Q,J^{-1}\bp\eta$, and $r$ contains at most one derivative of $\eta,~v$ or $Q$. We note that the term $K_2^\sh$ comes from $IB_{03}^\sh$ where $\Delta_v^\sh$ contributes to $\p_3^{N}\TP^{7-2N} g\cdot \p\TP h\cdot r$.

Since $1\leq N\leq 3$, we know $7-2N\geq 1$ and thus we can directly apply lemma \ref{trace} to control $K_1^\sh\sim K_3^\sh$.
\begin{equation}\label{tnbdry1}
\begin{aligned}
K_1^\sh\lesssim& |\p_3^{N-1}\dd f|_{8-2N}|\p_3^{N-1}g|_{9-2N}|\p h~r|_{L^{\infty}}\lesssim\|\p_3^{N-1}\dd f\|_{H_*^{9-2N}}\|\p_3^{N-1}g\|_{H_*^{10-2N}}\|\p h~r\|_{H^2}\\
\lesssim&\|f\|_{2(N-1)+1+(9-2N),*}\|g\|_{2(N-1)+(10-2N),*}\|h\|_3\|r\|_2=\|f\|_{8,*}\|g\|_{8,*}\|h\|_3\|r\|_2.
\end{aligned}
\end{equation}
and
\begin{equation}\label{tnbdry2}
\begin{aligned}
K_2^\sh\lesssim&|\p_3^{N-1}\dd f|_{8-2N}|\p_3^N g|_{7-2N}|(\p\TP h)r|_{L^{\infty}}\lesssim|\p_3^{N-1}\dd f|_{8-2N}|\p_3^N g|_{7-2N}|(\p\TP h)r|_{1.5}\\
\lesssim&\|\p_3^{N-1}\dd f\|_{H_*^{9-2N}}\|\p_3^N g\|_{H_*^{8-2N}}\|(\p\TP h)r\|_{2}\lesssim\|f\|_{8,*}\|g\|_{8,*}\|h\|_{7,*}\|r\|_{2},
\end{aligned}
\end{equation}
and
\begin{equation}\label{tnbdry3}
\begin{aligned}
K_3^\sh\lesssim&|\p_3^{N-1}\dd f|_{8-2N}|\p_3^N g|_{7-2N}|(\p h)r|_{L^{\infty}}\lesssim|\p_3^{N-1}\dd f|_{8-2N}|\p_3^N g|_{7-2N}|(\p h)r|_{1.5}\\
\lesssim&\|\p_3^{N-1}\dd f\|_{H_*^{9-2N}}\|\p_3^N g\|_{H_*^{8-2N}}\|(\p h)r\|_{2}\lesssim\|f\|_{8,*}\|g\|_{8,*}\|h\|_{3}\|r\|_{2}.
\end{aligned}
\end{equation}

One can use either trace lemma or similar techniques as in \eqref{n3bdry1}-\eqref{n3bdry22} to analyze the remaining terms which are all of lower order than $K_1^\sh\sim K_3^\sh$. This completes the control of $IB_0^\sh$.

The analysis of $IB_3^\sh\sim IB_6^\sh$ can be proceeded exactly in the same way as $IB_3^*\sim IB_6^*$. Since these quantities involving the modification terms $\Delta_Q^\sh,\Delta_v^\sh$ are of lower order, we do not repeat the details again. We can finally prove that
%
%
\begin{align}
\label{IBsh34} \int_0^T IB_3^\sh+IB_4^\sh\dt\lesssim&\int_0^T\left|A^{3i}\p_3^N\TP^{8-2N}\eta_i\right|_0 P\left(\|(\eta,v,b)\|_{8,*},\|Q\|_{8,*},\|\rho_0\|_3\right)\dt,\\ \nonumber
&+\left|A^{3i}\p_3^N\TP^{8-2N}\eta_i\right|_0P(\|\eta\|_{8,*},\|Q\|_{8,*})\int_0^T \|v(t)\|_{8,*}\dt\\
\label{IBsh5} IB_5^\sh\lesssim& |\hat{A}^{3i}|_{L^{\infty}}|\Delta_Q^\sh|_0|(\Delta_v^\sh)_i|_0\lesssim P(\|\eta\|_{8,*},\|v\|_{8,*},\|Q\|_{7,*}),\\
\label{IBsh6} IB_6^\sh\lesssim& |\hat{A}^{3i}\p_3 Q|_{L^{\infty}}|A^{3p}\p_3^N\TP^{8-2N}\eta_p|_0|(\Delta_v^\sh)_i|_0\lesssim P(\|\eta\|_{8,*},\|v\|_{8,*},\|Q\|_{7,*}).
\end{align}

Summarizing the estimates above, we get the control of the boundary integral
\begin{equation}\label{IBsh}
\int_0^T IB^\sh\lesssim -\frac{c_0}{4}\left|A^{3i}\p_3^N\TP^{8-2N}\eta_i\right|_0^2+\PP_0+ P(\EE(T))\int_0^T P(\EE(t))\dt.
\end{equation}

Combining \eqref{tngoodenergy0}-\eqref{tn1}, \eqref{Ish}, \eqref{IBsh} and choosing $\varepsilon>0$ in \eqref{Ish21} to be suitably small, we get the following inequality
\begin{equation}\label{tnenergy}
\|\VV^\sh\|_0^2+\left\|\p_3^N\TP^{8-2N}\left(J^{-1}\bp\eta\right)\right\|_0^2+\|\p_3^N\TP^{8-2N} q\|_0^2+\frac{c_0}{4}\left|A^{3i}\p_3^N\TP^{8-2N}\eta_i\right|_0^2\bigg|_{t=T}\lesssim \PP_0+ P(\EE(T))\int_0^T P(\EE(t))\dt.
\end{equation} Finally, invoking \eqref{tngoodv}, we get the $\p_3^N\TP^{8-2N}~(N=1,2,3)$-estimates of $v$
\begin{align*}
\|\p_3^N\TP^{8-2N}v\|_0\lesssim&\|\VV^\sh\|_0+\|\p_3^N\TP^{8-2N}\eta\|_0\|\pa v_i\|_{L^{\infty}}+\|\p_3^N\TP^{7-2N}\eta\|_0\left(\|\pa\TP v\|_{L^{\infty}}+\|\pa\TP \eta\cdot\pa v\|_{L^{\infty}}\right)+\|\p_3^N\TP^{7-2N}v\|_0\|\pa\TP\eta\|_{L^{\infty}}.
\end{align*}
Since $\p^m\eta|_{t=0}=0$ for any $m\geq 2, m\in \N^*$, we know
\begin{equation}\label{tngap}
\|\p_3^N\TP^{8-2N}v\|_0\lesssim\|\VV^\sh\|_0+P(\|v\|_{7,*},\|\eta\|_{7,*})\int_0^T P(\|v\|_{8,*}),
\end{equation}and thus
\begin{equation}\label{tn}
\|\p_3^N\TP^{8-2N} v\|_0^2+\left\|\p_3^N\TP^{8-2N}\left(J^{-1}\bp\eta\right)\right\|_0^2+\|\p_3^N\TP^{8-2N} q\|_0^2+\frac{c_0}{4}\left|A^{3i}\p_3^N\TP^{8-2N}\eta_i\right|_0^2\bigg|_{t=T}\lesssim \PP_0+ P(\EE(T))\int_0^T P(\EE(t))\dt.
\end{equation}

\subsection{Control of time derivatives}

In the case of $\p_*^I=\p_3^N\TP^{8-2N-k}\p_t^k$ for $1\leq k\leq 8-2N$, most of steps in the proof are still applicable. However, the presence of time derivative(s) could simplify the ``modified Alinhac good unknowns". We note that most of the modifications are essentially similar to Section \eqref{ts7} $\sim$ Section \ref{ts0}, so we no longer repeat the details.

\subsubsection{One time derivative}\label{sect tn7}

When $k=1$, the modified Alinhac good unknowns can be defined by replacing $8\TP^7$ by $(8-2N)\p_3^N\TP^{7-2N}$ in Section \ref{ts7AGU}, i.e.,
\begin{equation}\label{tn7goodeq}
\VV^\sh=\p_3^N\TP^{7-2N}\p_t v-\p_3^N\TP^{7-2N}\p_t\eta\cdot\pa v+\Delta_v^\sh,~~\QQ^\sh=\p_3^N\TP^{7-2N}\p_t Q-\p_3^N\TP^{7-2N}\p_t\eta\cdot\pa Q+\Delta_Q^\sh.,
\end{equation}where
\begin{equation}\label{tn7error}
\begin{aligned}
(\Delta_v^\sh)_i:=&-(8-2N)\p_3^N\TP^{7-2N}\eta\cdot\pa\p_t v_i-(8-2N)\p_3^N\TP^{7-2N}v\cdot\pa v_i+(16-4N)\p_3^N\TP^{7-2N}\cdot\pa v\cdot\pa v_i,\\
\Delta_Q^\sh:=&-(8-2N)\p_3^N\TP^{7-2N}\eta\cdot\pa\p_t Q+(8-2N)\p_3^N\TP^{7-2N}\eta\cdot\pa v\cdot\pa Q,
\end{aligned}
\end{equation}and
\begin{equation}\label{tn7goodgrad}
\p_3^N\TP^{7-2N}\p_t(\diva v)=\pa\cdot \VV^\sh+C^\sh(v),~~~\p_3^N\TP^{7-2N}\p_t(\pa Q)=\pa\QQ^\sh+ C^\sh(Q),
\end{equation}with
\[
\|C^\sh(f)\|_0\lesssim P(\EE(t))\|f\|_{8,*}.
\]
%

The difference between $\p_3^N\TP^{7-2N} v$ and $\VV^\sh$ should be controlled in the same way as \eqref{t7gap} by replacing $\TP^7$ with $\p_3^N\TP^{7-2N}$
\begin{equation}\label{tn7gap}
\|\p_3^N\TP^{7-2N}\p_tv\|_0\lesssim\|\VV^*\|_0^2+\PP_0+P(\EE(T))\int_0^T P(\EE(t))\dt,
\end{equation}and thus
\begin{equation}\label{tn7}
\begin{aligned}
&\|\p_3^N\TP^{7-2N}\p_t v\|_0^2+\left\|\p_3^N\TP^{7-2N}\p_t\left(J^{-1}\bp\eta\right)\right\|_0^2+\|\p_3^N\TP^{7-2N}\p_t q\|_0^2+\frac{c_0}{4}\left|A^{3i}\p_3^N\TP^{7-2N}\p_t\eta_i\right|_0^2\bigg|_{t=T}\\
\lesssim& \PP_0+ P(\EE(T))\int_0^T P(\EE(t))\dt.
\end{aligned}
\end{equation}

\subsubsection{2$\sim$(7-2N) time derivatives}\label{sect tn6}

When $2\leq k\leq 7-2N$, we can mimic the analysis in Section \ref{ts6}: We just need to replace $\dd^6\p_t^2$ by $\p_3^{N}\dd^{6-2N}\p_t^2$ where $\dd$ denotes $\TP$ or $\p_t$ and $\dd^{6-2N}$ contains at least one $\TP$. The analogous problematic term becomes $-2(\p_tA^{li})(\p_3^{N}\dd^{6-2N}\p_t\p_l f)-(6-2N)(\dd A^{li})(\p_3^{N}\dd^{5-2N}\p_t^2 \p_l f)$ which comes from $[\p_3^{N}\dd^{6-2N}\p_t^2,A^{li},\p_l f]$. Following \eqref{t6goodeq}-\eqref{t6goodgrad}, we can similarly define
\begin{equation}\label{tn6goodeq}
\VV^\sh=\p_3^{N}\dd^{6-2N}\p_t^2 v-\p_3^{N}\dd^{6-2N}\p_t^2\eta\cdot\pa v+\Delta_v^\sh,~~\QQ^\sh=\p_3^{N}\dd^{6-2N}\p_t^2 Q-\p_3^{N}\dd^{6-2N}\p_t^2\eta\cdot\pa Q,
\end{equation}where
\begin{equation}\label{tn6error}
(\Delta_v^\sh)_i:=-(6-2N)\p_3^{N}\dd^{5-2N}\p_t^2v\cdot\pa\dd\eta_i-2\p_3^{N}\dd^{6-2N}\p_tv\cdot\pa v_i
\end{equation}and
\begin{equation}\label{tn6goodgrad}
\p_3^{N}\dd^{6-2N}\p_t^2(\diva v)=\pa\cdot \VV^\sh+C^\sh(v),~~~\p_3^{N}\dd^{6-2N}\p_t^2(\pa Q)=\pa\QQ^*+C^\sh(Q),
\end{equation}with
\[
\|C^\sh(f)\|_0\lesssim P(\EE(t))\|f\|_{8,*}.
\]

Again we have $\Delta_Q^\sh$ in this case, and thus the analogues of $IB_3^\sh\sim IB_5^\sh$ all vanish. The boundary integrals $IB_0^\sh, IB_1^\sh,IB_2^\sh,IB_6^\sh$ and the interior terms can be controlled in the same way as Section \ref{sect tn}. Finally, one has
\begin{equation}\label{tn6}
\begin{aligned}
&\|\p_3^{N}\dd^{6-2N}\p_t^2 v\|_0^2+\left\|\p_3^{N}\dd^{6-2N}\p_t^2\left(J^{-1}\bp\eta\right)\right\|_0^2+\|\p_3^{N}\dd^{6-2N}\p_t^2 q\|_0^2+\frac{c_0}{4}\left|A^{3i}\p_3^{N}\dd^{6-2N}\p_t^2\eta_i\right|_0^2\bigg|_{t=T}\\
\lesssim& \PP_0+ P(\EE(T))\int_0^T P(\EE(t))\dt,
\end{aligned}
\end{equation}where $\dd^{6-2N}$ contains at least one spatial derivative $\TP$.

\subsubsection{Full time derivatives}\label{sect tn0}

When $\p_*^I=\p_3^N\p_t^{8-2N}$ for $N=1,2,3$, there is not tangential spatial derivative on the boundary and thus Lemma \ref{trace} is no longer applicable. In this case, the modified Alinhac good unknowns become
\begin{equation}\label{ts0goodeq}
\VV^\sh=\p_3^N\p_t^{8-2N} v-\p_3^N\p_t^{8-2N}\eta\cdot\pa v+\Delta_v^\sh,~~\QQ^\sh=\p_3^N\p_t^{8-2N}Q-\p_3^N\p_t^{8-2N}\eta\cdot\pa Q,
\end{equation}where
\begin{equation}\label{ts0error}
(\Delta_v^\sh)_i:=-(8-2N)\p_3^N\p_t^{8-2N}v\cdot\pa v_i
\end{equation}and
\begin{equation}\label{ts0goodgrad}
\p_3^N\p_t^{8-2N}(\diva v)=\pa\cdot \VV^\sh+C^\sh(v),~~~\p_3^N\p_t^{8-2N}(\pa Q)=\pa\QQ^\sh+C^\sh(Q),
\end{equation}with
\[
\|C^\sh(f)\|_0\lesssim P(\EE(t))\|f\|_{8,*}.
\]

The proof follows in the same way as Section \ref{ts0} after replacing $\p_t^7$ by $\p_t^{7-2N}$ and the coefficient 8 by $(8-2N)$. So we no longer repeat the details. Finally, we can get
\begin{equation}\label{tn0}
\begin{aligned}
&\|\p_3^N\p_t^{8-2N} v\|_0^2+\left\|\p_3^N\p_t^{8-2N} \left(J^{-1}\bp\eta\right)\right\|_0^2+\|\p_3^N\p_t^{8-2N} q\|_0^2+\frac{c_0}{4}\left|A^{3i}\p_3^N\p_t^{8-2N} \eta_i\right|_0^2\bigg|_{t=T}\\
\lesssim& \varepsilon\|\p_3^{N+1}\p_t^{6-2N} v\|_0^2+\PP_0+ P(\EE(T))\int_0^T P(\EE(t))\dt,
\end{aligned}
\end{equation}which together with \eqref{tn}, \eqref{tn7}, \eqref{tn6} concludes the proof of Proposition \ref{prop mixed}.

\section{Control of weighted normal derivatives}\label{normal8}

Now we consider the most general case $\p_*^I=\p_t^{i_0}(\sigma\p_3)^{i_4}\TP_1^{i_1}\TP_2^{i_2}\p_3^{i_3}$ with $i_1+i_2+2i_3+i_4=8$ and $i_4>0$. The presence of the weighted normal derivatives $(\sigma\p_3)^{i_4}$ makes the following difference from the non-weighted case.
\begin{enumerate}
\setlength{\itemsep}{0pt}
\setlength{\parsep}{0pt}
\setlength{\parskip}{0pt}
\item Extra terms are produced when we commute $\p_*^I$ with $\p_3$ because $\sigma$ is a function of $y_3$. Once $\p_3$ falls on the weight function, we will lose a weight and $(\sigma\p_3)$ becomes $\p_3$, which causes a loss of derivative. This appears when we commute $\p_*^I$ with $\pa^i$ that falls on $Q$ or $v_i$ and commute $\p_*^I$ with $\bp$.
\item There is no boundary integral because the weight function $\sigma$ vanishes on $\Gamma$.
\end{enumerate}

To overcome the difficulty mentioned above, we can again use the techniques, similar with those in the previous sections.
\begin{itemize}
\setlength{\itemsep}{0pt}
\setlength{\parsep}{0pt}
\setlength{\parskip}{0pt}
\item Invoke the MHD equation and the continuity equation to replace $\pa Q$ and $\diva v$ by tangential derivatives.
\item Produce a weight funtion by using $b_0^3|_{\Gamma}=0$ and $\TP Q|_{\Gamma}=0$.
\item In particular, if $\p_*^I$ does not contain time derivative, we need to add an extra modification term in the good unknown of $v$.
\end{itemize}

First we analyze $[\bp,\p_t^{i_0}(\sigma\p_3)^{i_4}\TP_1^{i_1}\TP_2^{i_2}\p_3^{i_3}] f$. Compared with the non-weighted case, we need to control the extra term $b_0^3\p_3(\sigma^{i_4})~(\p_t^{i_0}\TP_1^{i_1}\TP_2^{i_2}\p_3^{i_3+i_4} f)=i_4b_0^3(\p_3\sigma)\left(\p_t^{i_0}(\sigma\p_3)^{i_4-1}\TP_1^{i_1}\TP_2^{i_2}\p_3^{i_3+1}\right) f$. Using $b_0^3|_{\Gamma}=0$, one can produce a weight function $\sigma$ as in \eqref{t8C22}. Therefore
\[
\left\|b_0^3(\p_3\sigma)\left(\p_t^{i_0}(\sigma\p_3)^{i_4-1}\TP_1^{i_1}\TP_2^{i_2}\p_3^{i_3+1} f\right)\right\|_0\lesssim\|\p_3 b_0\|_{L^{\infty}}\|(\sigma\p_3)\p_t^{i_0}(\sigma\p_3)^{i_4-1}\TP_1^{i_1}\TP_2^{i_2}\p_3^{i_3} f\|_0\leq\|b_0\|_3\|f\|_{8,*}.
\]

Next we analyze the commutator between $\p_*^I=\p_t^{i_0}(\sigma\p_3)^{i_4}\TP_1^{i_1}\TP_2^{i_2}\p_3^{i_3}$ and $\pa f$. Compared with the non-weighted case, we shall analyze an extra term $C_{\sigma}$ below. In specific, one has
\begin{equation}\label{wbad1}
\begin{aligned}
&\p_t^{i_0}(\sigma\p_3)^{i_4}\TP_1^{i_1}\TP_2^{i_2}\p_3^{i_3}(A^{li}\p_l f)\\
=&\sigma^{i_4}\p_t^{i_0}\TP_1^{i_1}\TP_2^{i_2}\p_3^{i_3+i_4}(A^{li}\p_l f)\\
=&\sigma^{i_4}\left(A^{li}\p_l(\p_t^{i_0}\TP_1^{i_1}\TP_2^{i_2}\p_3^{i_3+i_4}) f\right)+\underbrace{\sigma^{i_4}[\p_t^{i_0}\TP_1^{i_1}\TP_2^{i_2}\p_3^{i_3+i_4},A^{li}]\p_lf}_{\mathring{C}}\\
=&A^{li}\p_l\left(\sigma^{i_4}\p_t^{i_0}\TP_1^{i_1}\TP_2^{i_2}\p_3^{i_3+i_4}f\right)-(i_4\p_3\sigma)\underbrace{A^{3i}\left((\sigma\p_3)^{i_4-1}\p_t^{i_0}\TP_1^{i_1}\TP_2^{i_2}\p_3^{i_3+1}f\right)}_{C_{\sigma}}+\mathring{C}.
\end{aligned}
\end{equation}

The term $\mathring{C}$ consists of the commutators produced in the same way as the non-weighted case. It can be analyzed in the same way as in previous sections by just considering $(\sigma\p_3)$ as a tangential derivative on the boundary. As for the extra term, we do the following computation
\begin{equation}\label{w1}
\begin{aligned}
&A^{3i}\left((\sigma\p_3)^{i_4-1}\p_t^{i_0}\TP_1^{i_1}\TP_2^{i_2}\p_3^{i_3+1}f\right)\\
=&(\sigma\p_3)^{i_4-1}\p_t^{i_0}\TP_1^{i_1}\TP_2^{i_2}\p_3^{i_3}(A^{3i}\p_3 f)-\left[(\sigma\p_3)^{i_4-1}\p_t^{i_0}\TP_1^{i_1}\TP_2^{i_2}\p_3^{i_3},A^{3i}\right]\p_3 f\\
=:&C^{\sigma}_1(f)+C^{\sigma}_2(f).
\end{aligned}
\end{equation}

Note that $i_0+i_1+i_2+i_4=8-2i_3$. We know the leading order terms in $C^{\sigma}_2$ are $\left((\sigma\p_3)^{i_4-1}\p_t^{i_0}\TP_1^{i_1}\TP_2^{i_2}\p_3^{i_3}A^{3i}\right)f$ and $(\dd A^{3i})(\dd^{6-2i_3}\p_3^{i_3+1} f)$, where $\dd$ represents any one of  $(\sigma\p_3),\p_t,\TP_1,\TP_2$. Recall that $A^{3i}$ cosists of $\TP\eta\cdot\TP\eta$. This shows that the highest order term in $\left((\sigma\p_3)^{i_4-1}\p_t^{i_0}\TP_1^{i_1}\TP_2^{i_2}\p_3^{i_3}A^{3i}\right)\p_3f$ is $(\dd^{8-2i_3}\p_3^{i_3}\eta)(\TP\eta) f$ whose $L^2(\Omega)$ norm can be directly controlled by $\|\eta\|_{8,*}\|\TP\eta\|_{L^{\infty}}\|\p_3 f\|_{L^{\infty}}$. As for the second term, we have
\[
\|(\dd A^{3i})(\dd^{6-2i_3}\p_3^{i_3+1} f)\|_0\lesssim\|(\dd\TP\eta)(\TP\eta)\|_{L^{\infty}}\|f\|_{8,*}.
\]Therefore, $C^{\sigma}_{2}$ can be directly controlled.

The control of $C^{\sigma}_1$ is more complicated. We should use the structure of MHD system \eqref{CMHDL} to replace one normal derivative by one tangential derivative.
\begin{align}
\label{qn}A^{3i}\p_3 Q=&-\sum_{L=1}^2A^{Li}\TP_L Q-R\p_t v^i+J^{-1}\bp(J^{-1}\bp\eta)\\
\label{vn}A^{3i}\p_3v_i=&\diva v-A^{1i}\TP_1 v_i-A^{2i}\TP_2v_i=-\frac{JR'(q)}{\rho_0}\p_t q-\sum_{L=1}^2A^{Li}\TP_L v_i
\end{align}

When $f=Q$, we plug \eqref{qn} into $C^{\sigma}_1(Q)$ to get
\begin{equation}\label{qbad1}
\begin{aligned}
C^{\sigma}_1(Q)=&(\sigma\p_3)^{i_4-1}\p_t^{i_0}\TP_1^{i_1}\TP_2^{i_2}\p_3^{i_3}(A^{3i}\p_3 Q)\\
=&-\sum_{L=1}^2(\sigma\p_3)^{i_4-1}\p_t^{i_0}\TP_1^{i_1}\TP_2^{i_2}\p_3^{i_3}(A^{Li}\TP_L Q)\\
&-(\sigma\p_3)^{i_4-1}\p_t^{i_0}\TP_1^{i_1}\TP_2^{i_2}\p_3^{i_3}(R\p_t v^i)+(\sigma\p_3)^{i_4-1}\p_t^{i_0}\TP_1^{i_1}\TP_2^{i_2}\p_3^{i_3}\left(J^{-1}\bp(J^{-1}\bp\eta))\right)\\
=:&C_{11}^{\sigma}(Q)+C_{12}^{\sigma}(Q)+C_{13}^{\sigma}(Q).
\end{aligned}
\end{equation}

When $f=v_i$, we plug \eqref{vn} into $C^{\sigma}_1(v)$ to get
\begin{equation}\label{vbad1}
\begin{aligned}
C^{\sigma}_1(v)=&(\sigma\p_3)^{i_4-1}\p_t^{i_0}\TP_1^{i_1}\TP_2^{i_2}\p_3^{i_3}(A^{3i}\p_3 v_i)\\
=&-\sum_{L=1}^2(\sigma\p_3)^{i_4-1}\p_t^{i_0}\TP_1^{i_1}\TP_2^{i_2}\p_3^{i_3}(A^{Li}\TP_L v_i)-(\sigma\p_3)^{i_4-1}\p_t^{i_0}\TP_1^{i_1}\TP_2^{i_2}\p_3^{i_3}\left(\frac{JR'(q)}{\rho_0}\p_t q\right)\\
=:&C_{11}^{\sigma}(v)+C_{12}^{\sigma}(v).
\end{aligned}
\end{equation}

The terms $C_{12}^{\sigma}(Q)$ and $C_{12}^{\sigma}(v)$ can be directly controlled. Note that $i_0+i_1+i_2+(i_4-1)=7-2i_3$, so
\begin{align}
\label{qbad12} \|C_{12}^{\sigma}(Q)\|_0\lesssim& \|R\|_{7,*}\|v\|_{8,*}\lesssim \|q\|_{7,*}\|v\|_{8,*},\\
\label{vbad12} \|C_{12}^{\sigma}(v)\|_0\lesssim& \|\rho_0\|_{7,*}\|q\|_{8,*}.
\end{align}

 Using $b_0^3|_{\Gamma}=0$, one can produce a weight function $\sigma$ as in \eqref{t8C22} when all the derivatives fall on $J^{-1}\bp\eta$.
\begin{equation}\label{qbad13}
\begin{aligned}
\|C_{13}^{\sigma}(Q)\|_0\lesssim&\|J^{-1}\bp(\sigma\p_3)^{i_4-1}\p_t^{i_0}\TP_1^{i_1}\TP_2^{i_2}\p_3^{i_3}(J^{-1}\bp\eta)\|_0\\
&+\left\|\left[(\sigma\p_3)^{i_4-1}\p_t^{i_0}\TP_1^{i_1}\TP_2^{i_2}\p_3^{i_3},J^{-1}\bp\right](J^{-1}\bp\eta)\right\|_0\\
\lesssim&\|\p_3 (J^{-1}b_0)\|_{L^{\infty}}\|(\sigma\p_3)^{i_4}\p_t^{i_0}\TP_1^{i_1}\TP_2^{i_2}\p_3^{i_3}(J^{-1}\bp\eta)\|_0\\
\lesssim&\|b_0\|_{7,*}\|J^{-1}\bp\eta\|_{8,*}.
\end{aligned}
\end{equation}

As for $C_{11}^{\sigma}$, the highest order term can be merged into the modified Alinhac good unknowns. One has
\begin{equation}\label{wbad111}
\begin{aligned}
C_{11}^{\sigma}(f):=&-\sum_{L=1}^2(\sigma\p_3)^{i_4-1}\p_t^{i_0}\TP_1^{i_1}\TP_2^{i_2}\p_3^{i_3}(A^{Li}\TP_L f)\\
=&-\sum_{L=1}^2\left((\sigma\p_3)^{i_4-1}\p_t^{i_0}\TP_1^{i_1}\TP_2^{i_2}\p_3^{i_3}A^{Li}\right)\TP_L f \underbrace{-\sum_{L=1}^2\left[(\sigma\p_3)^{i_4-1}\p_t^{i_0}\TP_1^{i_1}\TP_2^{i_2}\p_3^{i_3},A^{Li}\right]\TP_L f}_{C_{111}^{\sigma}(f)}\\
=&-\sum_{L=1}^2A^{Lp}\left((\sigma\p_3)^{i_4-1}\p_t^{i_0}\TP_1^{i_1}\TP_2^{i_2}\p_3^{i_3}\p_k \eta_p\right) A^{ki}\TP_L f\underbrace{-\sum_{L=1}^2\left(\left[(\sigma\p_3)^{i_4-1}\p_t^{i_0}\TP_1^{i_1}\TP_2^{i_2}\p_3^{i_3},A^{Lp}A^{ki}\right]\p_k \eta_p\right) \TP_L f}_{C_{112}^{\sigma}(f)} +C_{111}^{\sigma}(f).
\end{aligned}
\end{equation}Since $i_0+i_1+i_2+i_4=8-2i_3$, one can directly control the $L^2(\Omega)$-norms of $C_{111}^{\sigma}(f),~C_{112}^{\sigma}(f)$ by $P(\|\eta\|_{8,*})\|f\|_{8,*}$. For the first term in the RHS of \eqref{wbad111}, one can proceed in the following ways
\begin{itemize}
\setlength{\itemsep}{0pt}
\setlength{\parsep}{0pt}
\setlength{\parskip}{0pt}
\item $f=Q$: Since $\TP_L Q|_{\Gamma}=0$, one can produce a weight function as in \eqref{t8C4q} and thus
\begin{equation}
\begin{aligned}
&\left\|A^{Lp}\left((\sigma\p_3)^{i_4-1}\p_t^{i_0}\TP_1^{i_1}\TP_2^{i_2}\p_3^{i_3}\p_k \eta_p\right) A^{ki}\TP_L Q\right\|_0\\
\lesssim& \sum_{M=1}^2\left\|A^{Lp}\left((\sigma\p_3)^{i_4-1}\p_t^{i_0}\TP_1^{i_1}\TP_2^{i_2}\p_3^{i_3}\TP_M\eta_p\right) A^{Mi}\TP_L Q\right\|_0+\|A^{Lp}A^{3i}\TP\p_3 Q\|_{L^{\infty}}\|(\sigma\p_3)^{i_4}\p_t^{i_0}\TP_1^{i_1}\TP_2^{i_2}\p_3^{i_3}\eta_p\|_0\\
\lesssim& P(\|\eta\|_{3})\|Q\|_{7,*}\|\eta\|_{8,*}.
\end{aligned}
\end{equation}
\item $f=v_i$: When $\p_*^I$ contains time derivative $(i_0>0)$, then it can be directly controlled due to $\p_t \eta=v$
\begin{equation}
\left\|A^{Lp}\left((\sigma\p_3)^{i_4-1}\p_t^{i_0}\TP_1^{i_1}\TP_2^{i_2}\p_3^{i_3}\p_k \eta_p\right) A^{ki}\TP_L v_i\right\|_0=\left\|A^{Lp}\left((\sigma\p_3)^{i_4-1}\p_t^{i_0-1}\TP_1^{i_1}\TP_2^{i_2}\p_3^{i_3}\p_k v_p\right) A^{ki}\TP_L v_i\right\|_0\lesssim P(\|\eta\|_{3})\|v\|_{5,*}\|v\|_{8,*}.
\end{equation}
If $i_0=0$, then it can be written in the form of covariant derivative plus a controllable term.
\begin{equation}\label{wbad112}
\begin{aligned}
&-A^{Lp}\left((\sigma\p_3)^{i_4-1}\TP_1^{i_1}\TP_2^{i_2}\p_3^{i_3}\p_k \eta_p\right) A^{ki}\TP_L v_i\\
=&-A^{ki}\p_k\left((\sigma\p_3)^{i_4-1}\TP_1^{i_1}\TP_2^{i_2}\p_3^{i_3}\eta_p A^{Lp}\TP_L v_i\right)\\
&+\underbrace{A^{3i}(\p_3\sigma)\left((i_4-1)(\sigma\p_3)^{i_4-2}\TP_1^{i_1}\TP_2^{i_2}\p_3^{i_3+1}\eta_p \right)A^{Lp}\TP_L v_i+\nabla_A^i(A^{Lp}\TP_L v_i)\left((\sigma\p_3)^{i_4-1}\TP_1^{i_1}\TP_2^{i_2}\p_3^{i_3}\eta_p\right) }_{C_{113}^{\sigma}(v_i)}\\
=&-\pa^i\left((\sigma\p_3)^{i_4-1}\TP_1^{i_1}\TP_2^{i_2}\p_3^{i_3}\eta_p A^{Lp}\TP_L v_i\right)+C_{113}^{\sigma}(v_i).
\end{aligned}
\end{equation}We note that the first term in $C_{113}^{\sigma}(f)$ appears when $\p_k~(k=3)$ falls on the weight function and $i_4\geq 2$ and can also be directly controlled by $P(\|\eta\|_{8,*})\|f\|_{8,*}$.
\end{itemize}

 Next we merge the covariant derivative terms in $C_{\sigma}$ into the modified Alinhac good unknowns, i.e., for $\p_*^I=\p_t^{i_0}(\sigma\p_3)^{i_4}\TP_1^{i_1}\TP_2^{i_2}\p_3^{i_3}$ we define
\begin{equation}\label{wvgood}
\VV^{\sigma}_i:=
\begin{cases}
\p_*^I v_i-\p_*^I\eta\cdot\pa v_i+(\Delta_v^{\sigma})_i~~~&i_0\geq 1\\
\p_*^I v_i-\p_*^I\eta\cdot\pa v_i+(\Delta_v^{\sigma})_i+\sum\limits_{L=1}^2\left((i_4\p_3\sigma)(\sigma\p_3)^{i_4-1}\TP_1^{i_1}\TP_2^{i_2}\p_3^{i_3}\eta_p\right) A^{Lp}\TP_L v_i, ~~~&i_0=0,
\end{cases}
\end{equation} and
\begin{equation}\label{wqgood}
\QQ^{\sigma}:=\p_*^I Q-\p_*^I\eta\cdot\pa Q+\Delta_Q^{\sigma}.
\end{equation}
Then one has
\begin{align}
\label{wvgoodgrad} \p_*^I(\pa\cdot v)=\pa\cdot\VV^{\sigma}+C^{\sigma}(v),\\
\label{wvgoodgrad} \p_*^I(\pa Q)=\pa\QQ^{\sigma}+C^{\sigma}(Q),
\end{align}with $\|C^{\sigma}(f)\|_0\lesssim P(\EE(t))\|f\|_{8,*}.$ Here the ``extra modification terms" $\Delta_v^{\sigma}$ and $\Delta_Q^{\sigma}$ comes from the analysis of $\mathring{C}$ in \eqref{wbad1} whose precise expressions can be derived in the same way as Section \ref{sect t8} $\sim$ Section \ref{mixed}. The term $\left((i_4\p_3\sigma)(\sigma\p_3)^{i_4-1}\p_t^{i_0}\TP_1^{i_1}\TP_2^{i_2}\p_3^{i_3}\eta_p\right) A^{Lp}\TP_L f$ comes from \eqref{wbad1} and \eqref{wbad112}. Finally, the commutator $C^{\sigma}(f)$ consists of the commutator part in $\mathring{C}$, $C_{111}^{\sigma}(f)\sim C_{113}^{\sigma}(f), C_{12}^{\sigma}(f)$ and $C_{13}^{\sigma}(Q).$

Recall that $\sigma|_{\Gamma}=0$ and $\TP Q|_{\Gamma}=0$ imply $\QQ^{\sigma}|_{\Gamma}=0$. Therefore the boundary integral $\ig N_3\hat{A}^{3i}\QQ^{\sigma}\VV_i^{\sigma}\dyy$ vanishes. Hence, we can get the following estimates for $\p_*^I:=\p_t^{i_0}(\sigma\p_3)^{i_4}\TP_1^{i_1}\TP_2^{i_2}\p_3^{i_3}$
\begin{align}
\label{wmix}\|\p_*^I v\|_0^2+\left\|\p_*^I\left(J^{-1}\bp\eta\right)\right\|_0^2+\|\p_*^I q\|_0^2\bigg|_{t=T}\lesssim \PP_0+P(\EE(T))\int_0^T P(\EE(t))\dt.
\end{align}

\section{A priori estimates, uniqueness and continuous dependence on data}\label{apriori}

\subsection{Finalizing the a priori energy estimates}

Combining the $L^2$-energy conservation \eqref{conserve} with \eqref{n4n4}, \eqref{t8t8}, \eqref{tntn} and \eqref{wmix}, and then choosing $\eps>0$ to be suitably small, we finally get the following energy inequality
\begin{equation}\label{EE1}
\EE(T)-\EE(0)\lesssim \PP_0+P(\EE(T))\int_0^T P(\EE(t))\dt
\end{equation}under the a priori asuumptions \eqref{small}-\eqref{sign2}. By the Gronwall-type inequality, one can find some $T_1>0$ depending only on the initial data, such that
\begin{equation}\label{EEE}
\sup_{0\leq t\leq T_1}\EE(t)\leq P(\EE(0)).
\end{equation}This completes the a priori estimates of \eqref{CMHDL}.

\subsection{Justification of the a priori assumptions}

It suffices to justify the a priori assumptions \eqref{small}-\eqref{sign2}. First, invoking $\p_t J=J\diva v$ and $J|_{t=0}=1$, we get
\[
\|J-1\|_{7,*}\leq\int_0^T \|J\diva v\|_{7,*}\dt\lesssim\int_0^T P(\|\p\eta\|_{L^{\infty}})\|\p_t q\|_{7,*}\dt\leq \int_0^TP(\|\p\eta\|_{L^{\infty}})\|q\|_{8,*}\dt.
\]Therefore choosing $T>0$ to be sufficiently small yields \eqref{small}. The Rayleigh-Taylor sign condition in $[0,T_1]$ can be justified by proving $\p Q/\p N$ is a H\"older-continuous function in $t$ and $y$ variables. In specific, from the energy estimates we know that
$$\frac{\p Q}{\p N}\in L^{\infty}([0,T];H^{\frac52}(\Gamma)),~\p_t\left(\dfrac{\p Q}{\p N}\right)\in L^{\infty}([0,T];H^{\frac32}(\Gamma)).$$ By using the 2D Sobolev embedding $H^{\frac12}(\Gamma)\hookrightarrow L^4(\Gamma)$ and Morrey's embedding $W^{1,4}\hookrightarrow C^{0,\frac14}$ in 3D domain, we get the H\"older continuity of the Rayleigh-Taylor sign
$$\frac{\p Q}{\p N}\in W^{1,\infty}([0,T];H^{\frac32}(\Gamma))\hookrightarrow W^{1,\infty}([0,T];W^{1,4}(\Gamma))\hookrightarrow W^{1,4}([0,T]\times\Gamma)\hookrightarrow C_{t,x}^{0,\frac14}([0,T]\times\Gamma).$$ Therefore, \eqref{sign2} holds in a positive time interval provided that $-\dfrac{\p Q_0}{\p N}\geq c_0>0$ holds initially. Theorem \ref{CMHDEE} is proved.

\subsection{Control of initial energy by the initial data satisfying the compatibility conditions}\label{data}
Finally we need to show $\EE(0)<\infty$. Define $f_{(j)}:=\p_t^jf|_{t=0}$ to be the initial data of $\p_t^j f$ for $j\in\N$. We know the initial data should satisfy the following properties:
\begin{itemize}
\setlength{\itemsep}{0pt}
\setlength{\parsep}{0pt}
\setlength{\parskip}{0pt}
\item The compatibility conditions \eqref{cck} up to 7-th order.
\item The constraints $\nabla\cdot B_0=0,~B_0\cdot n|_{\{0\}\times\p\DD_0}=0$ and the Rayleigh-Taylor sign condition \eqref{sign} on $\{0\}\times\p\DD_0$.
\item The norms of the initial datum of the time derivatives of $(v,b,Q)$ can be controlled by the norms of initial data $(v_0,b_0,Q_0)$.
\end{itemize}

We note that the compatibility conditions up to order $m$ can be expressed in Lagrangian coordinates by using the formal power series solution to \eqref{CMHDL} in $t$:
\[
\hat{v}(t,y)=\sum v_{(j)}(y)\frac{t^j}{j!},~\hat{b}(t,y)=\sum b_{(j)}(y)\frac{t^j}{j!},\hat{Q}(t,y)=\sum Q_{(j)}(y)\frac{t^j}{j!},
\] satisfying $Q_{(j)}|_{\Gamma}=0$ for $j=0,1,\cdots,m$. Since we study the solutions in (anisotropic) Sobolev spaces, such compatibility conditions have to be expressed in a weak form
\begin{equation}\label{ccd}
Q_{(j)}(y)\in H_0^1(\Omega),~~0\leq j\leq m.
\end{equation}

From $(v_0,b_0,Q_0)\in H_*^8(\Omega)$ and the system \eqref{CMHDL}, one can only get $(v_{(j)},b_{(j)},Q_{(j)})\in H_*^{8-2j}(\Omega)$ for $0\leq j\leq 4$. To guarantee $(v_{(j)},b_{(j)},Q_{(j)})\in H_*^{8-j}(\Omega)$ and $Q_{(j)}\in H_0^1(\Omega)$, the initial data should be constructed in standard Sobolev space $H^8(\Omega)$ with $$\sum_{j=1}^8\|(v_{(j)},b_{(j)},Q_{(j)})\|_{8-j}^2\lesssim P(\|v_0\|_8,\|b_0\|_8,\|Q_0\|_8).$$ See the construction in \cite[Lemma 4.1]{TW2020MHDLWP}.

On the one hand, by Lemma \ref{embedding} we know $(v_{(j)},b_{(j)},Q_{(j)})\in H^{8-j}(\Omega)\hookrightarrow H_*^{8-j}(\Omega)$ which satisfies our requirement and implies $\EE(0)\lesssim P(\|v_0\|_8,\|b_0\|_8,\|Q_0\|_8)$. On the other hand, if we directly construct the initial data $(v_0,b_0,q_0)\in H_*^8(\Omega)$ such that $(v_{(j)},b_{(j)},Q_{(j)})\in H_*^{8-j}(\Omega)$, then it is not clear in which sense the boundary conditions and the compatibility conditions are satisfied. For example, $Q_{(7)}\in H_*^{1}(\Omega)$, but the trace of such function in that space has no meaning in general. This also explains why we require $Q_{(7)}\in H_0^1(\Omega)$ in \eqref{ccd}. See also Secchi \cite[Theorem 2.1]{Secchi1996}. Therefore, the initial data $(v_0,b_0,Q_0)$ should be constructed in the standard Sobolev space $H^8(\Omega)$.

\subsection{Continuous dependence on initial data and uniqueness}\label{continuity}
Now we prove Theorem \ref{CMHDEE2} by using a similar argument as in the proof of a priori bounds. Assume $U_0^{(i)}=(v_0^{(i)},b_0^{(i)},q_0^{(i)})\in H_*^8(\Omega)~(i=1,2)$ to be two initial datum of \eqref{CMHDL} satisfying the hypothesis of Theorem \ref{CMHDEE}. Suppose also $U^{(i)}(t,\cdot)=(\eta^{(i)}(t,\cdot),v^{(i)}(t,\cdot),b^{(i)}(t,\cdot),q^{(i)}(t,\cdot))~(i=1,2)$ to be the solutions to \eqref{CMHDL} with initial data $U_0^{(i)}$. Then we derive the system of $([\eta],[v],[q])$ as follows, where $[f]:=f^{(1)}-f^{(2)}$ for any function $f$.
\begin{equation}\label{CMHDLgap}
\begin{cases}
\p_t[\eta]=[v]~~~&\text{ in }\Omega, \\
R^{(1)}\p_t[v]-{J^{(1)}}^{-1}b_0^{(1)}\cdot\p\left({J^{(1)}}^{-1}b_0^{(1)}\cdot\p[\eta]\right)+\nabla_{A^{(1)}}[Q]+\nabla_{[A]}Q^{(2)}=f_{v}+f_{b},~~~&\text{ in }\Omega, \\
\ff'(q^{(1)})\p_t[q]+\dive_{A^1}[v]+\dive_{[A]}v^{(2)}=f_q~~~&\text{ in }\Omega, \\
[Q]=0~~~&\text{ on }\Gamma,\\
([\eta], [v], [b], [q])|_{t=0} = (\mathbf{0},[v_0],[b_0],[q_0]),
\end{cases}
\end{equation}with the initial constraints (divergence-free condition for $b_0$, $b_0^3|_{\Gamma}=0$ and the Rayleigh-Taylor sign condition) for each $i=1,2$. Here $Q^{(i)}:=q^{(i)}+\frac12 |b^{(i)}|^2$, $b^{(i)}={J^{(i)}}^{-1}b_0^{(i)}\cdot\p\eta^{(i)}$, and $\ff(q^{(i)}):=\log R^{(i)}(q^{(i)})$. The source terms $f_{v,b}$ and $f_q$ are defined by
\begin{align}
f_{v}:=&-[R]\p_tv^{(2)}+[b_0]\cdot\p({J^{(2)}}^{-1}b_0^{(2)}\cdot\p\eta^{(2)}),~f_b:=b_0^{(1)}\cdot\p([J^{-1}b_0]\cdot\p\eta^{(2)}),\\
f_q:=&-(\ff'(q^{(1)})-\ff'(q^{(2)}))\p_t q^{(2)}.
\end{align}

Note that system \eqref{CMHDLgap} has the same structure as \eqref{CMHDL} on the left side of each equation. And it is not difficult to see that the $\|\cdot\|_{6,*}$ norm of both source terms can be directly controlled, because each solution $U^{(i)}$ are bounded in $\|\cdot\|_{8,*}$ norm. In particular, the second term and the fourth term in $f_{b,v}$ can be controlled by writing the $[\cdot]$ terms back to the form $f^{(1)}-f^{(2)}$ and using the bounds for each $U^{(i)}$. Therefore, the estimates for $([\eta],[v],[b],[q])$ in $\|\cdot\|_{6,*}$ norm should follow in a similar way as in the proof of Theorem \ref{CMHDEE}. It is even easier because we no longer need to design the ``modified good unknowns" when taking $\p_*^I$ with $\lee I\ree = 6$. Indeed, given the derivative $\p_*^I$ with $\lee I\ree=6$, we define $\FF^{(i)}:=\p_*^I f^{(i)}-\p_*^I\eta^{(i)}\cdot\nabla_{A^{(i)}} f^{(i)},~i=1,2,$ to be the Alinhac good unknowns for $f^{(i)}$ with respect to $\p_*^I$. Then define $[\FF]:=\FF^{(1)}-\FF^{(2)}$ and we have
\begin{align}
[\FF]=&~\p_*^I [f]-\p_*^I\eta^{(1)}\cdot\nabla_{A^{(1)}} [f]-\p_*^I\eta^{(1)}\cdot\nabla_{[A]}f^{(2)}-\p_*^I[\eta]\cdot\nabla_{A^{(2)}} f^{(2)},\\
\p_*^I(\nabla_{A^{(1)}}[f]+\nabla_{[A]}f^{(2)})=&~\nabla_{A^{(1)}}[\FF]+C^{(1)}([f]),\\
\|[\FF]-\p_*^I [f]\|_0+&~\|\p_t([\FF]-\p_*^I[f])\|_0+\|C^{(1)}([f])\|_0\leq P(\EE^{(1)}(t),\EE^{(2)}(t))[\EE](t).
\end{align}We also obtain the evolution equation for the good unknown $[\VV]$
\begin{align}\label{VVgap}
&\rho_0^{(1)}\p_t[\VV]-b_0^{(1)}\cdot\p\p_*^I\left({J^{(1)}}^{-1}b_0^{(1)}\cdot\p[\eta]\right)+\nabla_{A^{(1)}}[\QQ]\nonumber \\
=&~\p_*^I f_b+\underbrace{\p_*^If_{v}+C^{(1)}([Q])-J^{(1)}[\p_*^I,R^{(1)}]\p_t[v]-J^{(1)}[\p_*^{I},{J^{(1)}}^{-1}b_0^{(1)}\cdot\p]\left({J^{(1)}}^{-1}b_0^{(1)}\cdot\p[\eta]\right)}_{=:\GG},
\end{align}where $\GG$ satisfies $\|\GG\|_0\leq  P(\EE(t))[\EE](t)$.

Multiplying $[\VV]$ in \eqref{VVgap} and integrating by part, we still get the following terms as in previous sections
\begin{align}\label{EEgapeq}
\frac12\ddt\io\rho_0^{(1)}|[\VV]|^2\dy=&-\ig N_3[\QQ]J^{(1)}A^{(1)3i}[\VV]_i\dyy+\io\GG\cdot[\VV]\dy+\io\p_*^I f_b\cdot[\VV]\dy\nonumber\\
&+\io J^{(1)}[\QQ]\nabla_{A^{(1)}}\cdot[\VV]\dy-\io\p_*^I\left({J^{(1)}}^{-1}b_0^{(1)}\cdot\p[\eta]\right)\bp[\VV]\dy,
\end{align}where the second term on the right side can be directly bounded. For the third term, we can integrate $b_0^{(1)}\cdot\p$ by parts to avoid more than 6 derivatives falling on $[J^{-1}b_0]$. When $b_0^{(1)}\cdot\p$ falls on $\p_*^I [v]$, we write $[v]=\p_t\eta$ and integrate by parts in $\p_t$ to control this term.

Since most of the steps are identical to the previous sections, we no longer repeat all those details. Below, we show the details of some key steps that are slightly different from the previous sections, and we only take $\TP^6$-estimates and $\p_3^3$ estimates for examples.
\paragraph*{Interior cancellation structure in Section \ref{stat3}.} We take $\p_*^I=\TP^6$ for example. Plugging the expression of $[\VV]$ into the last term in \eqref{EEgapeq}, we get
\begin{equation}
\begin{aligned}
&-\io\TP^6\left({J^{(1)}}^{-1}b_0^{(1)}\cdot\p[\eta]\right)(b_0^{(1)}\cdot\p)[\VV]\dy\\
=&-\frac12\ddt\io J^{(1)}\left|\TP^6\left({J^{(1)}}^{-1}b_0^{(1)}\cdot\p[\eta]\right)\right|^2\dy+\frac12\io\p_tJ^{(1)}\left|\TP^6\left({J^{(1)}}^{-1}b_0^{(1)}\cdot\p[\eta]\right)\right|^2\dy\\
&-\io J^{(1)}\TP^6\left({J^{(1)}}^{-1}b_0^{(1)}\cdot\p[\eta]\right)\cdot\left[\TP^6\p_t,J^{(1)}\right]({J^{(1)}}^{-1}b_0^{(1)}\cdot\p[\eta])\dy\\
&+\io\TP^6\left({J^{(1)}}^{-1}b_0^{(1)}\cdot\p[\eta]\right)\cdot\left[\TP^6,b_0^{(1)}\cdot\p\right]([v])\dy\\
&+\io\TP^6\left({J^{(1)}}^{-1}b_0^{(1)}\cdot\p[\eta]\right)(b_0^{(1)}\cdot\p)\left(\TP^6\eta^{(1)}\cdot\nabla_{A^{(1)}} [v]-\TP^6\eta^{(1)}\cdot\nabla_{[A]}v^{(2)}-\TP^6[\eta]\cdot\nabla_{A^{(2)}} v^{(2)}\right)\dy\\
\lesssim&-\frac12\ddt\io J^{(1)}\left|\TP^6\left({J^{(1)}}^{-1}b_0^{(1)}\cdot\p[\eta]\right)\right|^2\dy+P(\EE^{(1)}(t),\EE^{(2)}(t))[\EE](t),
\end{aligned}
\end{equation}where we note that the third term on the right side contains the analogue of $K_{11}$ term defined in Section \ref{stat3}, that is, $\TP^6\p_t$ may fall on $J^{(1)}$. Here we already know $\|J^{(1)}(t,\cdot)\|_{8,*}\leq P(\EE(t))$ and thus this term can be directly controlled.

Similarly, the fourth term in \eqref{EEgapeq} will produce the energy term of $[q]$ as stated in Section \ref{stat3} plus the term
\[
\io J^{(1)}\TP^6({J^{(1)}}^{-1}b_0^{(1)}\cdot\p[\eta])({J^{(1)}}^{-1}b_0^{(1)}\cdot\p[\eta])\TP^6(\nabla_{A^{(1)}}\cdot [v])\dy,
\]which is also directly controlled by $P(\EE^{(1)}(t),\EE^{(2)}(t))[\EE](t)$ due to $\|A^{(1)}(t,\cdot)\|_{6,*}\lesssim \|\eta(t,\cdot)\|_{8,*}^2\leq P(\EE(t))$.

\paragraph*{Boundary energy in $\TP^6$-estimates.}We plug $[Q]|_{\Gamma}=0$ and $Q^{(i)}=0$ into the boundary integral to get
\begin{equation}
\begin{aligned}
-\ig[\QQ]N_3A^{(1)3i}[\VV]_i\dyy=&\ig N_3\p_3Q^{(1)}J^{(1)}A^{(1)3j}\TP^6[\eta]_j A^{(1)3i}[\VV]_i\dyy\\
&+\ig N_3\p_3 Q^{(1)}(\TP^6[\eta]_j A^{(1)3j}+\TP^6\eta_j^{(2)}[A]^{3j})A^{(1)3i}[\VV]_i\dyy,\\
=&-\frac12\ddt\ig(-N_3\p_3 Q^{(1)})|A^{(1)3i}\TP^6[\eta]_i|^2\dyy+\frac12\ig(-N_3\p_3 \p_tQ^{(1)})|A^{(1)3i}\TP^6[\eta]_i|^2\dyy\\
&+\ig N_3\p_3Q^{(1)}A^{(1)3j}\TP^6[\eta]_jA^{(1)3i}(\TP^6\eta_r[A]^{kr}\p_k v_i^{(1)}-\TP^6\eta_r^{(2)} A^{(2)kr}\p_k[v]_i)\dyy\\
&+\ig N_3\p_3 Q^{(1)}(\TP^6[\eta]_j A^{(1)3j}+\TP^6\eta_j^{(2)}[A]^{3j})A^{(1)3i}[\VV]_i\dyy\\
\lesssim&-\frac{c_0}{4}\left|A^{(1)3i}\TP^6[\eta]_i\right|_0^2+ P(\EE^{(1)}(0),\EE^{(2)}(0))[\EE](t).
\end{aligned}
\end{equation}where we use the Rayleigh-Taylor sign condition for $Q^{(1)}$ in the first term. The other terms can be controlled after integrating by parts in $\p_t$ or in $\TP$ and using the trace lemma: $|[v]|_5\lesssim\|[v]\|_{6,*}$ and $|\eta^{(i)}|_{7}\lesssim \|\eta^{(i)}\|_{8,*}$.

\paragraph*{Boundary terms in $\p_3^3$-estimates.} The boundary integral now reads
\begin{equation}
\begin{aligned}
-\ig N_3J^{(1)}[\QQ][\VV]_i A^{(1)3i}\dyy,
\end{aligned}
\end{equation}where $$[\FF]=\p_3^3 [f]-\p_3^3\eta^{(1)}\cdot\nabla_{A^{(1)}} [f]-\p_3^3\eta^{(1)}\cdot\nabla_{[A]}f^{(2)}-\p_3^3[\eta]\cdot\nabla_{A^{(2)}} f^{(2)}.$$ What we need to do is again to reduce one $\p_3$ falling on $[Q],[v]_i$ to a tangential derivative by repeatedly invoking system \eqref{CMHDLgap}. Here we only list the identities analogous to those in Section \ref{sect n4bdry}. We first have
\begin{align}
\p_3^3\eta^{(1)}_j A^{(1)3j}=\p_3^2(\underbrace{\p_3\eta^{(1)}_j A^{(1)3j}}_{=1})-[\p_3^2,A^{(1)3j}]\p_3\eta_{(j)}.
\end{align}Then the third component of the second equation in \eqref{CMHDLgap} gives
\begin{equation}
\begin{aligned}
J^{(1)}A^{(1)3k}\p_3^3[Q]=&\p_3^2(\nabla_{\hat{A}^{(1)}}^k [Q])-[\p_3^2,A^{(1)3k}]\p_3[Q]\\
=&\p_3^2\left(-\rho_0^{(1)}\p_t[v]^k+b_0^{(1)}\cdot\p({J^{(1)}}^{-1}b_0^{(1)}\cdot\p[\eta]^k)+f_{v}^k+f_b^k-\nabla_{[A]}^kQ^{(2)}\right)\\
&-\sum_{L=1}^2\p_3^2(A^{(1)Li}\TP_L Q)-[\p_3^2,A^{(1)3k}]\p_3[Q],
\end{aligned}
\end{equation}and thus the top-order derivative on $[Q],[v]$ becomes $\p_3^2\dd$ for some tangential derivative $\dd$. Note also that $\p_3^2[A]^{mk}\p_mQ^{(2)}=\p_3^2[A]^{3k}\p_3 Q^{(2)}$ due to $Q^{(2)}|_{\Gamma}=0$ and $A^{3l}$ only contains tangential derivative $\TP\eta$. The term $\p_3^2A^{(1)Li}$ can be directly controlled by $\|\eta^{(1)}\|_4^2\leq \EE(t)$, so we no longer needs the subtle cancellation structures introduced in section \ref{sect n4bdry}. Similarly, using the third equation of \ref{CMHDLgap}, we get
\begin{equation}
\begin{aligned}
A^{(1)3i}\p_3^3[v]_i=&\p_3^2(\dive_{A^{(1)}} [v])-[\p_3^2,A^{(1)3i}]\p_3[v]_i\\
=&\p_3^2\left(-\ff'(q^{(1)})\p_t[q]+f_q-\dive_{[A]}v^{(2)}\right)-\sum_{L=1}^2\p_3^2(A^{(1)Li}\TP_L [v]_i)-[\p_3^2,A^{(1)3i}]\p_3[v]_i,
\end{aligned}
\end{equation}where the term $-\p_3^2\dive_{[A]} v^{(2)}$ may have a term in which $\p_3^2$ falls on $\p_3[\eta]$, that is, $A^{(2)lr}\p_3^3[\eta]_rA^{(1)3i}\p_l v_i^{(2)}=A^{(1)3i}\p_3^3[\eta]\cdot\nabla_{A^{(2)}} v^{(2)}$ which cancels with the last term in $A^{(1)3i}[\VV]_i$.

After these reductions, the top-order terms on the boundary becomes the following form, which is controlled analogously to \eqref{IBIB7}. For $\lee I\ree=6$, we have
\begin{equation}\label{IBIB5}
\begin{aligned}
&\ig N_3(\p_*^{I-e_3}\dd [f])(\p_*^{I-e_3} \dd [g])h\dyy\\
=&\io(\p_3\p_*^{I-e_3}\dd [f])(\p_*^{I-e_3} \dd [g])h\dy+\io(\p_*^{I-e_3}\dd [f])(\p_3\p_*^{I-e_3} \dd [g])h\dy+\io(\p_*^{I-e_3}\dd [f])(\p_*^{I-e_3} \dd [g])\p_3h\dy\\
\overset{\dd}{=}&-\io(\p_3\p_*^{I-e_3}[f])(\p_*^{I-e_3} \dd^2 [g])h\dy+\io(\p_3\p_*^{I-e_3}[f])(\p_*^{I-e_3} \dd [g]) \dd h\dy\\
&-\io(\p_*^{I-e_3}\dd^2 [f])(\p_3\p_*^{I-e_3} [g])h\dy+\io(\p_*^{I-e_3}\dd [f])(\p_3\p_*^{I-e_3} [g]) \dd h\dy +\io(\p_*^{I-e_3}\dd [f])(\p_*^{I-e_3} \dd [g])\p_3h\dy\\
\lesssim&~\|[f]\|_{6,*}\|[g]\|_{6,*}\|h\|_{3},
\end{aligned}
\end{equation}where the terms in $h$ only have at most one derivative on each variable.

Define the energy functional for \eqref{CMHDLgap}
\begin{equation}
[\EE](t):=\|[\eta](t,\cdot)\|_{6,*}^2+\|[v](t,\cdot)\|_{6,*}^2+\|(J^{-1}b_0)^{(1)}\cdot\p[\eta](t,\cdot)\|_{6,*}^2+\|[q](t,\cdot)\|_{6,*}^2+\sum_{\lee I\ree = 6}\left|A^{(1)3i}\p_*^I[\eta]_i\right|_0^2.
\end{equation}
We can finally get the estimates for \eqref{CMHDLgap}
\begin{equation}
[\EE](t)\leq [\EE](0)+\int_0^T P(\EE^{(1)}(t),\EE^{(2)}(t))[\EE](t)\dt,
\end{equation}and thus by Gr\"onwall's inequality, there exists some $T_2\in[0,T_1]$ ($T_1>0$ is the time for the a priori bounds obtained in Theorem \ref{CMHDEE}) depending only on the initial data and $c_0$ such that
\[
\sup_{0\leq t\leq T_2}\EE(t)\leq P(\EE^{(1)}(0),\EE^{(2)}(0))[\EE](0)\leq P(\|[v_0],[b_0],[Q_0]\|_6)P(\||v_0^{(1)},v_0^{(2)}\|_8,\|b_0^{(1)},b_0^{(2)}\|_8,\|Q_0^{(1)},Q_0^{(2)}\|_8).
\]
Note that the energy inequality is linear in $[\EE](t)$ because \eqref{CMHDLgap} is a \textit{linear} system of $[v],[b],[Q]$. In particular, if the two given initial datum are equal, we must have $[\EE](0)=0$ and thus $[\EE](t)=0$ in $[0,T_2]$. This proves the uniqueness and continuous dependence on initial data provided the local-in-time solution exists.

\section{On the local existence of solutions}\label{sect LWP}
\subsection{Local existence theorem for smooth data satisfying compatibility conditions up to infinite order}\label{sect Nash}

As stated in Section \ref{statLWP}, we need to prove Theorem \ref{CMHDsmooth}, that is, a local existence theorem from $C^{\infty}$ data to $C^{\infty}$ solution. Following the Nash-Moser iteration scheme presented in Alinhac-G\'erard \cite{AlinhacNash} and Secchi \cite{SecchiNotes} (also adopted in \cite{Lindblad2005LWP1, Lindblad2005LWP2, TW2020MHDLWP,TW2021MHDSTLWP}), to solve a nonlinear system $L(U)=f$, we start with an approximate solution $U^a$ and then $U=U^a+V$ is a solution to $L(U)=f$ if we can prove $V$ solves $\mathbf{L}(V):=L(U^a+V)-L(U^a)=f^a$ where $f^a:=-L(U^a)$ with $V|_{t=0}=0$.
\begin{rmk}[Existence of smooth approximate solution]
In \cite[Lemma 4.1]{TW2020MHDLWP}, an approximate solution was constructed in Sobolev space. Here we can follow Lindblad \cite[Lemma 16.3]{Lindblad2005LWP2} to construct a smooth approximate solution by introducing the power series $Q^a(t,y):=\sum_k\chi(t/\eps_k)Q_{(k)}(y)t^k/k!$ where $Q_{(k)}:=\p_t^kQ|_{t=0}\in C^{\infty}(\bar\Omega)\cap H_0^1(\Omega)$ are defined in Appendix \ref{appendix data inf}, $\chi(\cdot)\in C_c^{\infty}(\R)$ is a smooth cut-off function which equals to 1 in $[-1,1]$ and vanishes outside $[-2,2]$, and $\eps_k>0$ are chosen suitably small such that the series converges in $H^N(\Omega)$ for any $N$. Then solve $v^a,\eta^a$ from the MHD system \eqref{CMHDL}.
\end{rmk}

To solve the increment $V$, we start with $V_0=0$ and inductively define $V_{n+1}=V_n+\delta V_n$ where $\delta V_n$ is the solution to the linearized problem $\mathbf{L}'(U^a+S_{\theta_n}V_n)(\delta V_n)=f_n$ with $S_{\theta_n}$ being the smoothing operator, $\lim_n\theta_n=+\infty$. By Taylor expansion, we have
\[
\mathbf{L}(V_{n+1})-\mathbf{L}(V_n)=\mathbf{L}'(U^a+V_n)(\delta V_n)+e_n'=\mathbf{L}'(U^a+S_{\theta_n}V_n)(\delta V_n)+e_n'+e_n'',
\]where $e_n'$ is the quadratic error produced by the expansion, and $e_n''$ is the substitution error produced by replacing the basic state $U^a+V_n$ by the smooth one $U^a+S_{\theta_n}V_n$.

When proving the convergence of $\sum_n\delta V_n$ in a certain Sobolev space (here we assume it is $H_*^s(\Omega)$), we need to start with the following induction hypothesis
\begin{quote}Induction hypothesis $(H_{n-1}):\forall 0\leq i\leq n-1,~\|\delta V_i\|_{s,*}\lesssim \eps\theta_i^{s-\alpha-1}\Delta_i$ with $\eps\ll 1$ and $\Delta_i:=\theta_{i+1}-\theta_i$,
\end{quote} and then prove a similar estimate for $\delta V_n$, that is, $\|\delta V_n\|_{s,*}\lesssim \eps\theta_n^{s-\alpha-1}\Delta_n$. In \cite{TW2020MHDLWP}, this induction step was proven in Lemma 4.14.

To prove the existence in $C^{\infty}$, we need to get an improved estimate in the induction, that is, we start with the induction hypothesis $(H_{n-1})$ above, then we need to prove the following ``improved estimate" which is better than $(H_n)$
\begin{quote} ``Improved estimate": $\|\delta V_n\|_{s,*}\lesssim \theta_n^{s-\alpha'-1}\Delta_n$ for some $\alpha'=\alpha+\gamma$ with $\gamma>0$ a fixed positive number.
\end{quote} Once this is done, then we can replace the induction hypothesis $(H_{n-1})$ by the improved one and repeat this again and again, and finally we will see $\sum_n\delta V_n$ converges in $C([0,T];H_*^{s+k\gamma})$ for any $k\in\N$ where $T>0$ is independent of $k$. See also the ``additional regularity" argument summarized in \cite[pp.152]{AlinhacNash}, \cite[Section 3.7]{SecchiNotes} and adapted by Lindblad \cite[(18.40)-(18.43)]{Lindblad2005LWP1}.

First, we find that, if there are only the first two error terms $e_n',e_n''$ in \cite[(4.26)-(4.27)]{TW2020MHDLWP}, then one can get the following estimates for $\delta V_n$ based on $(H_{n-1})$ above:
\[
\|\delta V_n\|_{s,*}\lesssim \Delta_n\left(\theta_n^{s-\alpha-1}\|f^a\|_{\alpha,*}+\eps^2\theta_n^{\zeta_2(s)-1}\right)+\Delta_n\left(\theta_n^{5-\alpha}\|f^a\|_{\alpha,*}+\eps^2\theta_n^{\zeta_2(6)-1}\right)\eps\theta_n^{(s+4-\alpha)_+}
\]where $\zeta_2(s):=\max\{(s+2-\alpha)_++12-2\alpha,s+8-2\alpha\}$,~$\alpha\geq 12,~\tilde{\alpha}:=\alpha+3, ~6\leq s\leq \tilde{\alpha}-2$. So the above estimates give us
\begin{equation}
\|\delta V_n\|_{s,*}\lesssim(\|f^a\|_{\alpha,*}+\eps^2)\theta_n^{s-\alpha-1}\Delta_n\Rightarrow\|\delta V_n\|_{s,*}\lesssim\eps\theta_n^{s-\alpha-1}\Delta_n
\end{equation}for $\eps>0$ and $\|f^a\|_{\alpha,*}/\eps$ sufficiently small. But if we compare the powers of $\theta_n$, we find that $\zeta_2(6)-1=\max\{3-\alpha,14-2\alpha\}\leq(5-\alpha)-2$ and $\zeta_2(s)-1\leq (s-\alpha-1)-2$. Based on this and the smoothness of $f$, we can get the improved estimates
\begin{equation}
\|\delta V_n\|_{s,*}\lesssim(\|f^a\|_{\alpha+2,*}+\eps^2)\theta_n^{s-\alpha-1-2}\Delta_n\Rightarrow\|\delta V_n\|_{s,*}\lesssim\theta_n^{s-\alpha-1-2}\Delta_n
\end{equation}for $\eps>0$ sufficiently small.

Then, starting from this new estimate, that is, replacing $\alpha$ by $\alpha'=\alpha+2$, we can get improved estimates for the errors $e_n',~e_n''$ and then the total error $E_n:=e_0+\cdots+e_{n-1}$ with $e_i:=e_i'+e_i''$, the source term $f_n$ for the linearized problem, and $\delta V_n$
\begin{align*}
\|e_i\|_{s,*}\lesssim \eps\theta_i^{\zeta_2'(s)-1}\Delta_i,~\|S_{\theta_n}e_{n-1}\|_{s,*}\lesssim\Delta_{n-1}\eps\theta_{n-1}^{\zeta_2'(s)-1}\\
\|(S_{\theta_n}-S_{\theta_{n-1}})E_{n-1}\|_{s,*}\lesssim\Delta_{n-1}\eps\theta_{n-1}^{\zeta_2'(s)-1},~\|(S_{\theta_n}-S_{\theta_{n-1}})f^a\|_{s,*}\lesssim\|f^a\|_{\alpha',*}\theta_n^{s-\alpha'-1}\Delta_n\\
\|f_n\|_{s,*}\lesssim\Delta_n(\theta_n^{s-\alpha'-1}\|f^a\|_{\alpha',*}+\eps\theta_n^{\zeta_2'(s)-1})\\
\|\delta V_n\|_{s,*}\lesssim \Delta_n\left(\theta_n^{s-\alpha'-1}\|f^a\|_{\alpha',*}+\eps\theta_n^{\zeta_2'(s)-1}\right)+\Delta_n\left(\theta_n^{5-\alpha'}\|f^a\|_{\alpha',*}+\eps\theta_n^{\zeta_2'(6)-1}\right)\theta_n^{(s+4-\alpha')_+},
\end{align*}where $\zeta_2'(s)$ is defined by replacing $\alpha$ with $\alpha'$ (and hence $\zeta_2'(s)\leq \zeta_2(s)-4$). Note that $\eps^2$ in the previous estimates now becomes $\eps$ because we use $\|fg\|_{s,*}\lesssim \|f\|_{s,*}\|g\|_{4,*}+\|f\|_{4,*}\|g\|_{s,*}$ and the $\eps$ arising from $\|f\|_{s,*}$ is now replaced by 1. We know that $\zeta_2'(s)-1\leq (s-\alpha'-1)-2$ and thus we obtain
\begin{equation}
\|\delta V_n\|_{s,*}\lesssim(\|f^a\|_{\alpha+2+2,*}+\eps)\theta_n^{s-\alpha'-1-2}\Delta_n\Rightarrow\|\delta V_n\|_{s,*}\lesssim\theta_n^{s-\alpha'-1-2}\Delta_n.
\end{equation}

Again we replace the induction hypothesis $(H_{n-1})$ by $\|\delta V_i\|_{s,*}\lesssim \theta_n^{s-\alpha''-1}\Delta_i$ for $i\leq n-1$ with $\alpha''=\alpha+2\times 2$ to proceed the Nash-Moser iteration. Repeat this again and again, we will get  $\|\delta V_i\|_{s,*}\lesssim \theta_n^{s-\alpha-2k-1}\Delta_i$ for $i\leq n$ and for any $k\in\N$. Hence, we can follow the argument in Lindblad \cite[(18.40)-(18.43)]{Lindblad2005LWP1} to show that the series $\sum_n\delta V_n$ converges in $C([0,T]; H_*^{s+2k}(\Omega))$ for any $k$ with $T$ independent of $k$ and similarly $\sum_n\p_t^r\delta V_n$ converges in $C([0,T]; H_*^{s-r+2k}(\Omega))$ for $0\leq r\leq s$. So the local existence and uniqueness in $C([0,T]\times C^{\infty}(\bar\Omega))$ is proven, provided that one can construct a smooth data satisfying the compatibility conditions up to infinite order, which will be achieved in Appendix \ref{appendix data inf}. The additional regularity in the time variable can be obtained by differentiating \eqref{CMHDL} by $\p_t$ repeatedly.

However, there are two extra error terms $e_n''',D_{n+\frac12}\delta\Psi_n$ in \cite[(4.26)-(4.27)]{TW2020MHDLWP}, which are produced because the free surface of $\Omega(t)=\T^2\times(-\infty, \varphi(t))$ is directly flattened by an explicit diffeomorphism instead of using Lagrangian coordinates.
\begin{itemize}
\item \textbf{$e_n'''$: modifications in the boundary conditions.} Under the setting of \cite{TW2020MHDLWP}, the kinematic boundary condition $\p_t \varphi=v\cdot N$ and the constraint $H\cdot N=0$ are involved in the equation. But the linearization breaks the structure of these two boundary conditions. So after adding the increment in each step of Nash-Moser iteration, extra modifications are required to guarantee these two boundary conditions hold for the basic state $V_{n+1}=V_n+\delta V_n$.
\item \textbf{$D_{n+\frac12}\delta\Psi_n$: Dropping the zero-th order term when replacing $\delta V_n$ by its ``good unknown" $\delta\dot{V}_n$ in the linearized problem.} To solve and prove the tame estimates for the linearized problem for $\delta V_n$, the authors of \cite{TW2020MHDLWP} replaced the variables $\delta V_n$ by the good unknown (without derivative) $\delta \dot{V}_n:=\delta V_n-\delta \Psi_n(\p_3 V_n/\p_3\Phi_n)$ and drop a zero-th order term to get the so-called effective linearized problem. This step produces the problematic error $D_{n+\frac12}\delta\Psi_n$ due to that dropped term. \textbf{This is a bad term, as it contributes to $\eps^2\theta_n^{s-\alpha-1}\Delta_n$ which prevents us improving the estimates of $\delta V_n$.}
\end{itemize}

Under the setting of Lagrangian coordinates, these two terms are not needed. For example, there were no such error terms in Lindblad \cite{Lindblad2005LWP1,Lindblad2005LWP2} where the local existence for smooth solutions to Euler equations are proved by using Nash-Moser iteration.

Indeed, when using Lagrangian coordinates, the material derivative becomes $\p_t$ and the kinematic boundary condition then becomes ``$\p_t$ is a tangential derivative". In other words, there is no description for the position of $\p\Omega(t)$ in Lagrangian coordinates. Instead, the information of free surface is reflected by $\eta$ which is defined as the flow map of velocity. Besides, Lemma \ref{b} shows that the magnetic field is completely determined by its initial data and the flow map, i.e., $b=J^{-1}\bp\eta$, and thus the MHD system \eqref{CMHDL} only includes the variables $\eta,v,q$. The boundary constraint is now just $b_0^3=0$ that has no dependence on time. Hence, we don't have to consider the propagation of this condition when doing the iteration. The formulation \eqref{CMHDL} does not affect the tame estimates, as we can still get the energy of $J^{-1}\bp\delta\eta$ together with $\delta v$ and $\delta q$ as in (2.4). In fact, the linearized momentum equation still has the form parallel to the nonlinear problem $\p_t\delta v-\bp^2\delta\eta+\pa\delta Q=\cdots$ (with $\rho, J$ omitted) and multiplying this by $\delta v$ and integrating by parts in section \ref{stat2}-\ref{stat3} gives energy estimates. When the variation operator $\delta$ falls on $J^{-1}$, we can use $\rho_0=RJ$ to get $\delta\rho_0=R\delta J+JR'(q)\delta q$ and the positivity of density $R$ and then $\delta J$ can be expressed in terms of $\delta q$. This shows why we can avoid the modification error $e_n'''$.

The error $D_{n+\frac12}\delta\Psi_n$ can also be avoided. In Lagrangian coordinates, the analogue of $\delta \dot{V}_n$ is equal to $\delta V_n-\delta\eta_n\cdot\pa V_n$. When doing tame estimates, we may still use $\delta\dot{V}_n$ to do calculation, but we finally derive the energy inequality for $\delta V_n$ instead of $\delta\dot{V}_n$. The reason is that their difference can be estimated by $\|\delta\eta_n\|_{s,*}\|\pa V_n\|_{L^{\infty}}\leq \|\pa V_n\|_{L^{\infty}}\int_0^T\|\delta v_n\|_{s,*}$. The advantange is that $\delta\eta_n$ has the same regularity as $\delta v_n$ so we don't have derivative loss for this term; while under the setting of \cite{TW2020MHDLWP}, the regularity for $\delta\Psi_n$ is not the same as $\p_t\delta\Psi_n$.

On the other hand, dropping the zero-th order term $\delta\eta_n\cdot\pa V_n$ when solving the linearized problem is not necessary. Indeed, the dropped term is a zero-th order term that can be moved to the right side, and the estimates for the ``effective" linearized problem (cf. \cite[(3.25)]{TW2020MHDLWP}) shows that the interior inhomogeneous term comes with an time integral. Thus, one can solve the linearized problem by using contraction mapping theorem in some $[0,T_0']$. This $T_0'$ may be smaller than $T_0$ obtained in \cite[Theorem 3.1]{TW2020MHDLWP}, but since the system is linear, it can be continued to the full time interval $[0,T_0]$.

In Lagrangian coordinates, we have $\delta\eta_n$ is the flow map of $\delta v_n$, and one can alternatively use the Galerkin method presented in Gu-Luo-Zhang \cite[Section 7.1]{GuLuoZhang2021MHDST} to prove the local existence of the linearized system. If we expand $\delta\eta$ to be $\sum_j Z_j(t)e_j(y)$ for some Galerkin basis $\{e_j\}$, then the zero-th order term dropped in \cite{TW2020MHDLWP} is just equal to $\sum_jZ_j(t)e_j(y)\cdot(\text{coefficients only involving basic state})$ and the velocity becomes $\sum_j Z_j'(t)e_j(y)$. The momentum equation is still an ODE involving $Z_j''(t),~Z_j(t)$ and terms involving magentic field and pressure. Thus, we no longer need to take into account the error $D_{n+\frac12}\delta\Psi_n$ as in \cite{TW2020MHDLWP}. Based the above discussion, we claim that Theorem \ref{CMHDsmooth} holds.

\subsection{Continuation of smooth solutions}\label{sect continue}

To prove the continuation criterion for smooth solution stated in Theorem \ref{CMHDcontinue}, we need to prove the energy inequality in the following form
\begin{equation}
\EE_m(t)\leq \EE_m(0)+\int_0^T P(\EE_{m-1}(t))\EE_m(t)\dt,
\end{equation}that is, the energy inequality for $\EE_m(t)$ is linear in $\EE_m(t)$ (the highest order terms). Once this energy inequality is proven, then Theorem \ref{CMHDcontinue} follows by using contradiction. Indeed, for a given $m$, if $T_m^*<+\infty$ and we still have $\EE_{m-1}(t)\leq M$ in $[0,T_m^*]$ for some constant $M>0$, then the above inequality shows that $\EE_m(t)\leq \EE_m(0)(1+P(M)te^{P(M)t})$ in $[0,T_m^*]$ and thus it still remains bounded, which contradicts with the maximality of $T_*^m$.

To verify the linearity of the highest order terms in the energy estimates, it suffices to analyze the commutators, either the terms in the modified good unknowns or the reduction of normal derivatives on the boundary, such that the highest order term is linear. For simplicity of notations, we only verify the case $m=8$ and $\p_*^I =\TP^8$ for the commutators appearing in the interior estimates, and $\p_*^I =\p_3^4$ for commutators arising in the reduction of normal derivatives on the boundary.

\paragraph*{Commutators $[\TP^8,f]\dd g$ for $f,g\in H_*^{8}(\Omega)$.} Here $\dd$ is a tangential derivative, such as $\p_t$ or $\bp$. Such commutators appear when the energy terms are produced. The terms included in such commutators have the form $\TP^N f\TP^{8-N}\dd g$ for $1\leq N\leq 8$. Indeed, when $1\leq N\leq 4$, we put $L^{\infty}(\Omega)$ norm on $f$ and $L^2(\Omega)$ norm on $g$. When $5\leq N\leq 8$, we put $L^{\infty}(\Omega)$ norm on $g$ and $L^2(\Omega)$ norm on $f$. By using the definition of anisotropic Sobolev space and the Sobolev embedding $H^2(\Omega)\hookrightarrow L^{\infty}(\Omega)$, we will find $$\|[\TP^8,f]\dd g\|_0\lesssim\|f\|_{8,*}\|g\|_{5,*}+\|f\|_{7,*}\|g\|_{6,*}+\|f\|_{6,*}\|g\|_{7,*}+\|f\|_{5,*}\|g\|_{8,*},$$ which is linear in $\|\cdot\|_{8,*}$ norm.

\paragraph*{The error term $I_3$ when producing the energy of $q$.} In \eqref{I3} and \eqref{I*3}, we have the following term
\[
\int_0^T\io\frac{JR'(q)}{\rho_0}\p_*^I\eta_p\hat{A}^{lp}\p_l Q\p_*^I\p_tq\dy\dt,
\]which is controlled after integrating by parts in $\p_t$. So, it introduces a term without time integral
\[
\io\frac{JR'(q)}{\rho_0}\p_*^I\eta_p\hat{A}^{lp}\p_l Q\p_*^Iq\dy\lesssim\eps\|q\|_{8,*}^2+\frac{1}{4\eps}\|\p_*^I\eta\|_0^2\left\|\frac{JR'(q)}{\rho_0}\hat{A}^{lp}\p_l Q\right\|_{L^{\infty}}^2.
\]
Note that the terms in $L^{\infty}(\Omega)$ norm only have at most one derivative, so this term is bounded by $\|\cdot\|_{6,*}$ norms of $\eta,\rho_0,Q$. Next using $\eta=\text{Id}+\int_0^Tv\dt$ gives the energy estimate that is linear in $\|\cdot\|_{8,*}$ norm.

\paragraph*{Commutators in modifided Alinhac good unknowns.} This is the most involved part in the paper. We take the case $\p_*^I=\TP^8$, i.e., the most difficult case, for an example. Recall the calculations in section \ref{ts8AGU} and we find that the control of $C_3,C_4,C_5$ is parallel to $C_1$ and the control of $C_6$ is parallel to $C_2$, so we only focus one the control of $C_1(f)$ and $C_2(f)$.

Recall that $C_1(f)=8\pa^i(\pa^p\TP f)\TP^7\eta_p-8([\TP^6,A^{lp}A^{mi}]\p_m\eta_p)\TP\p_l f$. In the first term, the top-order part has the form $P(\p\eta)\p^2\TP \eta\TP^7 \eta$, whose $L^2(\Omega)$ norm can be controlled by using $H^1(\Omega)\hookrightarrow L^6(\Omega)$
\begin{equation}
\begin{aligned}
\|\p^2\TP f\TP^7 \eta\|_0^2=&\io \p^2\TP f~\TP^7 \eta~\p^2\TP f~\TP^7 \eta\dy\\
\overset{\TP}{=}&-\io\p^2\TP^2 f~\TP^7 \eta~\p^2\TP f~\TP^6 \eta\dy-\io\p^2\TP f~\TP^8 \eta~\p^2\TP f~\TP^6 \eta\dy\\
\lesssim&~\|\p^2\TP^2f\|_{1}\|\p^2\TP f\|_{1}\|\TP^7\eta\|_0\|\TP^6\eta\|_1+\|\p^2\TP f\|_{1}^2\|\TP^8\eta\|_0\|\TP^6\eta\|_1\\
\lesssim&\|\eta\|_{8,*}\|\eta\|_{7,*}\|f\|_{7,*}\|f\|_{8,*}+\|\eta\|_{8,*}^2\|f\|_{7,*}^2\lesssim \EE_8(t)\EE_7(t).
\end{aligned}
\end{equation}
In the second term of $C_1(f)$, we need to control $\|[\TP^6,A^{lp}A^{mi}]\p_m\eta_p\|_0$, whose top-order part reads $P(\p\eta,\TP\eta)(\TP^{N}\p\eta)(\TP^{6-N}\p\eta)$ for $1\leq N\leq 6$. So it is controlled by
\[
P(\|\eta\|_{7,*})(\|\TP^6\p\eta\|_0\|\p\eta\|_{L^{\infty}}+\|\TP^5\p\eta\|_0\|\TP\p \eta\|_{L^{\infty}}+\|\TP^4\p\eta\|_0\|\TP^3\p\eta\|_{L^{\infty}})\leq P(\|\eta\|_{7,*})\|\eta\|_{8,*}
\]

The control of $C_2(f)$ is divided into three parts in section \ref{ts8AGU}, and $C_2(v)$, $C_{21}(Q)$ and $C_{22}(Q)$ are already controlled in our desired form. Now we analyze $C_{23}(Q)$ that has the form $\TP^N \hat{A}^{lp}\TP^{7-N}\p_l Q$ for $1\leq N\leq 6$. For $N=5,6$, we put $L^{2}(\Omega)$ norm on $A$ and $L^{\infty}(\Omega)$ norm on $Q$; and put $L^{\infty}(\Omega)$ norm on $A$ and $L^{2}(\Omega)$ norm on $Q$ for $1\leq N\leq 2$. When $n=3,4$, we need a observation that when $l=3$, $\hat{A}^{3p}=\TP\eta\times\TP\eta$ only contains tangential derivative; and when $l=1,2$, $\p_l$ itself is tangential. So when $N=4$, we have either $(\TP^4\p \eta\times \TP\eta)\TP^4 Q$ or $(\TP^5 \eta\times \TP\eta)\TP^3\p Q$  which is controlled by $$\|\TP^4\p\eta\|_0\|\p\eta\|_{L^{\infty}}\|\TP^4 Q\|_{L^{\infty}}+\|\TP^5\eta\|_{L^3}\|\p\eta\|_{L^{\infty}}\|\TP^3\p Q\|_{L^6}\leq \|\eta\|_{6,*}^2\|Q\|_{8,*}.$$

Similar approach applies to the modified good unknowns for other derivatives which shoule be easier than the case $\p_*^I=\TP^8$. So we show that, for a function $f\in H_*^8(\Omega)$ and its ``modified" good unknown $\FF$ with respect to $\p_*^I$, the following property holds
\[
\p_*^{I}(\pa^i f)=\pa^i\FF+C^i(f),\text{ with }\|C(f)\|_{8,*}\leq P(\EE_7(t))\EE_8(t),
\]where $\EE_m(t)$ is the energy functional defined in \eqref{EE}.

\paragraph*{The estimates for the modification terms $\Delta_f$ on $\Gamma$.} In the boundary estimates in section \ref{sect t8bdry}, we have to control the $L^2(\Gamma)$ norms of $\Delta_Q,\Delta_v$ and $\p_t(J\Delta_Q)$. Indeed, the analysis in \eqref{modifyt} and the Sobolev inequality $|f|_{W^{1,\infty}}\leq \|f\|_{5,*}$ have already make the energy estimates linear in the top-order term, that is,
\[
|\Delta_Q|_0+|\Delta_v|_0+|\p_t(J\Delta_Q)|_0\leq P(\EE_7(t))\EE_8(t).
\]

\paragraph*{Commutators arising in the reduction of normal derivatives.} When there is a normal derivative $\p_3$ included in $\p_*^I$, we need to repeatedly use the MHD system to reduce one normal derivative to one tangential derivative, which then produces lots of commutators. We take $\p_*^I=\p_3^4$ for an example. There are two types of commutators in section \ref{sect n4bdry} and we need to control their $|\cdot|_0$ norm.
\begin{itemize}
\item $[\p_3^3,A^{3i}]\p_3 f$. This is easy to control, it equals to $\p_3^3 A^{3i}\p_3 f+3\p_3^2A^{3i}\p_3^2 f+3\p_3A^{3i}\p_3^3 f$ and is then controlled by
\begin{align}
\|[\p_3^3,A^{3i}]\p_3 f\|\leq &~P(|\TP\eta,\TP\p\eta|_{L^{\infty}})\left(|\p_3^3\eta|_1|\p_3 f|_{L^{\infty}}+|\p_3^2\TP\eta|_{L^4}|\p_3^2 f|_{L^4}+|\p_3\TP\eta|_{L^{\infty}}|\p_3^3 f|_0\right) \nonumber\\
\leq &~P(\|\eta\|_{7,*})(\|\eta\|_{4}+\|f\|_{4}),
\end{align}where we use the fact that $A^{3i}=\TP\eta\times\TP\eta.$
\item $[\p_3^2, A^{Lp}A^{mi}]\p_3\p_m\eta_p$, $L=1,2$. The analysis is similar as above, because the top-order term has the form $(\p_3^3\eta)(\p_3^2\eta)P(\p\eta)$.
\end{itemize}

Summarizing the above analysis, we show that, for $m\geq 8$, the energy inequality is in fact
\[
\EE_m(T)\leq \EE_m(0)+P(\EE_{m-1}(T))\int_0^T\EE_m(t) P(\EE_{m-1}(t))\dt,
\]so we have proved a continuation criterion.

\subsection{Passing to the case of initial data satisfying compatibility conditions up to finite order}\label{sect limit}

Finally, we prove Theorem \ref{CMHDLWPm}. First we recall that, we say the initial data $(v_0,b_0,Q_0)$ satisfies the compatibility  conditions up to $k$-th order, if $Q_{j}:=\p_t^j Q|_{t=0}=0$ on $\Gamma$ holds for $0\leq j\leq k$; and we say $(v_0,b_0,Q_0)$ satisfies the compatibility conditions up to infinite order if $Q_{j}:=\p_t^j Q|_{t=0}=0$ on $\Gamma$ holds for all $j\geq0,~j\in\Z$.

 Given an integer $m\geq 8$ and initial data $U_0:=(v_0,b_0,Q_0)$ satisfying the compatibility conditions up to $(m-1)$-th order, assume we already find a sequence of smooth data $U_0^{(n)}:=(v_0^{(n)},b_0^{(n)},Q_0^{(n)})$ satisfying the compatibility conditions up to infinite order, such that $\|U_0^{(n)}-U_0\|_{m}\to 0$ as $n\to\infty$. Now we introduce the following procedure:
\begin{enumerate}
\item For each $n$, using Theorem \ref{CMHDsmooth}, we know there exists a unique smooth solution $U^{(n)}(t,\cdot)\in C^{\infty}([0,T^{(n)}];\Omega)$ for some $T^{(n)}>0$. The lifespan $T^{(n)}$ may depend on $n$ at this point.
\item By our a priori estimates (Theorem \ref{CMHDEE}), we have $\|U^{(n)}(t,\cdot)\|_{m,*}\leq P(\|U^{(n)}(0,\cdot)\|_{m,*})\leq P(\|U_0^{(n)}\|_{m})$ provided the existence and the right side is independent of $n$. Using the continuation criterion (Theorem \ref{CMHDcontinue}), we can extend the solution a bit more after $T^{(n)}$, until the a priori bounds become invalid. So, \textbf{the lifespan of $\{U^{(n)}(t)\}$ in $H_*^8(\Omega)$ has a lower bound $T_0$ independent of $n$}.
\item Using Theorem \ref{CMHDEE2}, we know $\|U^{(k)}(t,\cdot)-U^{(l)}(t,\cdot)\|_{m-2,*}\leq C(\|U_0^{(k)}-U_0^{(l)}\|_{m-2})P(\|U_0^{(k)},U_0^{(l)}\|_{m})$. Here $C(\cdot)>0$ is a continuous function of its arguments and $C(x)\to 0$ as $|x|\to 0$. This is because the equations for $U^{(k)}-U^{(l)}$ are \textbf{linear} in $U^{(k)}-U^{(l)}$, so $(\p_t^kU^{(k)}-\p_t^kU^{(l)})|_{t=0}$ can be expressed linearly in terms of $U_0^{(k)}-U_0^{(l)}$. When $k,l\to\infty$, the right side converges to zero, and then $\{U^{(n)}(t)\}$ has a limit $U(t)\in H_*^{m-2}(\Omega)$ in $[0,T_0]$.
\item The limit $U(t)$ must be a solution with the given initial data $U_0\in H^{m}(\Omega)$ by using Sobolev embedding. Since each $U^{(n)}(t)$ belogns to $H_*^{m}(\Omega)$ in $[0,T_0]$, we know by the a priori bounds, the limit $U(t)$ also satisifies $\|U(t)\|_{m,*}\leq P(\|U_0\|_{m})$ in some $[0,T_m]$ with $T_m$ only depending on $\|U_0\|_{m}$, $c_0$ in the Rayleigh-Taylor sign, and $A_0$ in the equation of state.
\end{enumerate}

Therefore, \textbf{under the assumption of Theorem \ref{CMHDLWPm}}, for each given data $U_0:=(v_0,b_0,Q_0)$ satisfying the compatibility conditions up to $(m-1)$-th order, we prove that there exists a unique solution to \eqref{CMHDL} in $C([0,T_m],H_*^m(\Omega))$ for some $T_m$ only depending on $\|U_0\|_{m}$, $c_0$ in the Rayleigh-Taylor sign, and $A_0$ in the equation of state. This solution also satisfies the a priori bounds, continuous dependence on initial data and uniqueness as stated in Theorem \ref{CMHDEE} and Theorem \ref{CMHDEE2}.

\begin{appendix}
\section{Construction of smooth data satisfying compatibility conditions up to infinite order}\label{appendix data}

In the appendix, we prove the existence of smooth data satisfying compatibility conditions up to infinite order. Assume $m\geq 8$ is an integer and we are given an initial data $(w_0,b_0,P_0)$ satisfying the compatibility conditions up to $(m-1)$-th order in $H^m(\Omega)$, where $P_0$ is the total pressure, while the fluid pressure is denoted by $p_0=P_0-\frac12|b_0|^2$. For simplicity of notations, we assume $R'(q)/R|_{R=1}=1$.

The initial constraints and compatibility conditions for $(w_0,b_0,P_0)$ are
\begin{itemize}
\item (Compatibility conditions) $P_{(j)}:=\p_t^j P|_{t=0}=0$ on $\Gamma$, $0\leq j\leq m-1$.
\item (Initial constraints) $\dive b_0=0$ in $\Omega$, $b_0^3|_{\Gamma}=0$, and the Rayleigh-Taylor sign condition $-\frac{\p P_0}{\p N}\geq c_0>0$ on $\Gamma$.
\end{itemize}

\subsection{Compatibility conditions in terms of initial data}\label{appendix data 1}

First we express the compatibility conditions in terms of $w_0,b_0,P_0$. The zero-th order compatibility condition is $P_0|_{\Gamma}=0$. To express the first-order compatibility condition, we use $P=p+\frac12|b_0|^2$, where $p$ is the fluid pressure, the continuity equation $\p_t p+\dive v=0$, and $\p_t b=b\cdot\nab v-b\dive v$ to get (we omit the coefficient $A$ as it equals to Id at $t=0$. The appearance of $A$ does not affect the essence of the proof.)
\begin{equation}\label{data order 1}
\dive w_0=-|b_0|^2\dive w_0+(\bar{b}_0\cdot\cnab)w_0\cdot b_0\text{ on }\Gamma,
\end{equation}and we define the right side of \eqref{data order 1} to be the functional $\MM_{-1}(w_0,b_0)$.
Next we take divergence in the momentum equation to get a wave equation of $P$
\begin{equation}\label{MHDwave}
\p_t^2 P-\Delta P=\p_t^2(\frac12|b|^2)+\nab_iw^j\nab_jw^i-\nab_ib^j\nab_jb^i,~~P_0|_{\Gamma}=0
\end{equation}and again use $\p_t b=b\cdot\nab v-b\dive v$ and $\p_t v\sim\bp b-\nab P$ to get
\[
\p_t^2 P-\Delta P=\MM_0(w_0,b_0,P_0)+\NN_0(w_0,b_0)~~\text{ on }\{t=0\},
\] where $\NN_0(w_0,b_0):=\nab_iw_0^j\nab_jw_0^i-\nab_ib_0^j\nab_jb_0^i$ and $\MM_0(w_0,b_0,P_0)$ is defined by
\[
\MM_0(w_0,b_0,P_0):=|b_0|^2\Delta P_0-(b_0\cdot\nab)^2 P_0+(b_0\cdot\nab)^2b_0\cdot b_0+\RR_0(w_0,b_0)
\]and $\RR_0(w_0,b_0)$ only contains the first-order derivative of $b_0$ and $v_0$
\[
\RR_0(w_0,b_0):=P_0(b_0)\left((\nab^{i_1} w_0)(\nab^{i_2} w_0)+(\nab^{j_1} b_0)(\nab^{j_2} b_0)\right),
\]where $P_0(b_0)$ is a polynomial of $b_0$ only contains cubic and quadratic terms, and $(i_1,i_2,j_1,j_2)=(1,1,0,0)\text{ or }(0,0,1,1)$.

We see that the 2nd-order compatibility condition $P_{(2)}:=\p_t^2 P|_{t=0}=0$ on $\Gamma$ is equivalent to (we use $b_0^3|_{\Gamma}=0$)
\begin{equation}\label{data order 2}
-\Delta P_0=\MM_0(w_0,b_0,P_0)+\NN_0(w_0,b_0,P_0)=|b_0|^2\Delta P_0-(\bar{b}_0\cdot\cnab)^2 P_0+(\bar{b}_0\cdot\cnab)^2b_0\cdot b_0+\RR(w_0,b_0) \text{ on }\Gamma.
\end{equation}

Taking time derivatives in the wave equation above repeatedly we get for $k \geq 1$
\[
P_{(k+2)}-\Delta P_{(k)}=\MM_k(w_0,b_0,P_0)+\NN_k(w_0,b_0,P_0)\text{ in }\Omega,
\]where, after long and tedious calculations, the functionals $\MM_k,\NN_k$ have the following form for $r\geq 1$
\begin{align}
k=2r-1,~~\MM_{2r-1}(w_0,b_0,P_0)=&-|b_0|^2\Delta^r\dive w_0+(\bar{b}_0\cdot\cnab)^2\Delta^{r-1}\dive w_0\nonumber \\
&+\sum_{l=2}^{r+1}\underbrace{b_0^{i_1}\cdots b_0^{i_{2l}}(\nabla^{2r+1}w_0)}_{<2^{l}\text{ terms}}+\RR_{2r-1}(w_0,b_0,P_0),\label{M odd}\\
k=2r,~~\MM_{2r}(w_0,b_0,P_0)=&~|b_0|^2\Delta^{r+1} P_0-(\bar{b}_0\cdot\cnab)^2\Delta^{r}P_0+\RR_{2r}(w_0,b_0,P_0),\nonumber\\
&+\sum_{l=2}^{r+1}\underbrace{(\bar{b}_0\cdot\cnab)^{r+2}(\nab^rb_0) b_0^{i_1}\cdots b_0^{i_{2l}}+(\bar{b}_0\cdot\cnab)^{2}(\nab^{2r} P_0) b_0^{j_1}\cdots b_0^{j_{2l}}}_{<2^{l}\text{ terms }}\label{M even};
\end{align}and the term $\RR_k$, where every top-order term has $(k+1)$-th order derivative, has the following form
\[
\RR_k(w_0,b_0,P_0)=P_k(b_0)\left(C^{k}_{i_1\cdots i_m,j_1\cdots j_n, k_1\cdots k_l}(\nab^{i_1} w_0)\cdots(\nab^{i_m} w_0)(\nab^{j_1} b_0)\cdots(\nab^{j_n} b_0)(\nab^{k_1} P_0)\cdots(\nab^{k_l} P_0)\right),
\]where $P_k(\cdot)$ is a polynomial of its arguments and the lowest power is 4 and the indices above satisfy
\begin{align*}
1\leq i_1,\cdots, i_m,j_1,\cdots, j_n\leq k+1, 0\leq k_1,\cdots, k_l \leq k+1,\\
i_1+\cdots+i_m+j_1+\cdots+j_n+k_1+\cdots+k_l=k+1.
\end{align*}
The term $\NN_k(w_0,b_0,P_0)$ has the following form
\begin{align*}
\NN_k(w_0,b_0,P_0)=&~P_{k,1}(b_0)(\nab^{1+2\lfloor\frac{k}{2}\rfloor}w_0)(\nab w_0)+P_{k,2}(b_0)(\nab^{2\lceil \frac{k}{2}\rceil} P_0)(\nab w_0)+P_{k,0}(b_0)(\nab^{k+1}b_0)(\nab w_0)\\
&+P'_k(b_0)D^{k}_{i_1\cdots i_m,j_1\cdots j_n, k_1\cdots k_l}\left((\nab^{i_1} w_0)\cdots(\nab^{i_m} w_0)(\nab^{j_1} b_0)\cdots(\nab^{j_n} b_0)(\nab^{k_1} P_0)\cdots(\nab^{k_l} P_0)\right),
\end{align*}where $P_{k,1}(\cdot), P_{k,2}(\cdot),P_k'(\cdot)$ are polynomials of their arguments and $P_{k,0}(\cdot)$ is a polynomial of its arguments and the lowest power is 2. The indices above satisfy
\begin{align*}
1\leq i_1,\cdots, i_m,j_1,\cdots, j_n\leq k, 0\leq k_1,\cdots, k_l \leq k,\\
i_1+\cdots+i_m+j_1+\cdots+j_n+k_1+\cdots+k_l=k+1.
\end{align*}
So the $k$-th compatibility condition can be equivalently written as
\begin{align}
\label{data order odd} k=2r+1,~~\Delta^r\dive w_0=&\sum_{j=0}^{r}\Delta^j(\MM_{2r-1-2j}(w_0,b_0,P_0)+\NN_{2r-1-2j}(w_0,b_0,P_0))\text{ on }\Gamma,\\
\label{data order even} k=2r,~~-\Delta^r P_0=&\sum_{j=0}^{r-1}\Delta^j(\MM_{2r-2-2j}(w_0,b_0,P_0)+\NN_{2r-2-2j}(w_0,b_0,P_0))\text{ on }\Gamma,
\end{align}where $\MM_{-1}(w_0,b_0):=|b_0|^2\dive w_0-(b_0\cdot\nab)w_0\cdot b_0$ and $\NN_{-1}:=0$.

\subsection{Regularization of the given data and recovery of compatibility conditions}

To construct a smooth data satisfying the compatibility conditions up to infinite order, the first step is to regularize the given data such that we get smooth functions. By the standard approximation of Sobolev function, we know for any given $\eps>0$, there exists $(w_0^\eps,b_0^\eps,P_0^\eps)\in C^{\infty}(\Omega)$ such that
\[
\|w_0^\eps-w_0,b_0^\eps-b_0,P_0^\eps-P_0\|_s<\eps.
\]

However, such smooth approximation does not preserve the boundary conditions, even for the vanishing boundary conditions for $P_0$ and $b_0^3$. So we need to recover the compatibility conditions up to the same order as the given data.

From now on, we assume
\begin{itemize}
\item $m=8$, that is, the given data satisfies the compatibility conditions \eqref{data order odd}-\eqref{data order even} up to 7-th order. This corresponds to the minimal requirement in Theorem \ref{CMHDEE}.
\item $\|b_0\|_{L^{\infty}(\bar\Omega)}<\delta_0<1$ where $\delta_0$ is a suitably small number to be determined: to absorb the terms containing $(k+2)$-th order derivative arising in $\MM_k$. Note that we do not need $\|b_0\|_{L^{\infty}(\bar\Omega)}$ to be arbitrarily small in the proof.
\end{itemize}

\subsubsection{Recovering the initial constraints}
The new data should also satisfy the initial constraints: divergence-free condition of magnetic field, vanishing normal component of magnetic field on the boundary and the Rayleigh-Taylor sign condition. The Rayleigh-Taylor sign condition still holds for $P_0^\eps$, as $-\p_3P_0^\eps$ is just a small perturbation of $-\p_3 P_0$. We then modify $b_0^\eps$. First, we introduce $\tilde{b}_0^\eps$ defined by
\begin{align}
\tilde{b}_0^{\eps,1}=b_0^{\eps,1},~~\tilde{b}_0^{\eps,2}=b_0^{\eps,2};~~-\Delta\tilde{b}_0^{\eps,3}=-\Delta b_0^{\eps,3}\text{ in }\Omega,~\tilde{b}_0^{\eps,3}=0\text{ on }\Gamma,
\end{align}and then $\tilde{b}_0^\eps\in C^{\infty}(\Omega)$ and the elliptic estimates imply $\|\tilde{b}_0^\eps-b_0\|_s\leq\|b_0^\eps-b_0\|_s+|0-0|_{s-0.5}=O(\eps)$. Next, we recover the divergence-free condition by introducing $\bz^\eps:=\tilde{b}_0^\eps+\nab\varphi$ with $\varphi$ determined by
\begin{align}
-\Delta\varphi=\dive \tilde{b}_0^\eps\text{ in }\Omega,~\p_3\varphi=0\text{ on }\Gamma.
\end{align}
With this modification, we now have $\dive \bz^\eps=\dive\tilde{b}_0^\eps+\Delta\varphi=0$ in $\Omega$, and $\bz^3|_{\Gamma}=0$ still holds thanks to the Neumann boundary condition $\p_3\varphi=0\text{ on }\Gamma$. So, $\bz^\eps$ is the desired magnetic field that we need, and it still satisfies a smallness assumption $\|\bz\|_{L^{\infty}(\bar{\Omega})}<2\delta_0$. We'll drop $\eps$ in $\bz$ for the sake of clean notations.

\subsubsection{Recovering the compatibility conditions up to $(m-1)$-th order}
Next we focus on the modification of $w_0^\eps,P_0^\eps$. After the regularization, we don't even know if $P_0^\eps=0$ on $\Gamma$ holds or not. So the first step is to recover the 0-th order compatibility condition $P_0|_\Gamma=0$. We define $\pp^{(1)}$ by
\begin{align}
-\Delta \pp^{(1)}=-\Delta P_0^\eps\text{ in }\Omega,~~\pp^{(1)}=0\text{ on }\Gamma.
\end{align}Since $P_0|_{\Gamma}=0$, we know
\[
\|\pp^{(1)}-P_0\|_s\leq\|P_0^\eps-P_0\|_s+|0-0|_{s-0.5}=O(\eps).
\]

Next we define $\ww^{(1)}$ to be the following function such that $(\ww^{(1)},\bz,\pp^{(1)})$ satisfies the compatibility condition up to first order: $\ww^{(1),1}=w_0^{\eps, 1},~\ww^{(1),2}=w_0^{\eps,2}$ and $\ww^{(1),3}$ is determined by the bi-harmonic system
\begin{equation}
\begin{cases}
\Delta^2\ww^{(1),3}=\Delta^2w_0^{\eps,3}~~&\text{ in }\Omega,\\
\ww^{(1),3}=w_0^{\eps,3},~~\p_3\ww^{(1),3}=-\TP_1 w_0^{\eps,1}-\TP_2 w_0^{\eps,2}+\MM_{-1}(\ww^{(1)},\bz)~~&\text{ on }\Gamma,
\end{cases}
\end{equation}where $\MM_{-1}(\ww^{(1)},\bz^{(1)})$ is given by \eqref{data order 1}. Note that the second boundary contidition only involves $\p_3\ww^{(1)}$ because the tangential components are the same of $w_0^\eps$. So the elliptic estimates give us
\begin{equation}
\begin{aligned}
\|\ww^{(1)}-w_0\|_s\lesssim&~\|\Delta^2w_0^\eps-\Delta^2 w_0\|_{s-4}+|w_0^\eps-w_0|_{s-0.5}+|\p_3\ww^{(1)}-\p_3 w_0|_{s-1.5}\\
\lesssim&~O(\eps)+|b_0|_{L^{\infty}}^2|\nab\ww^{(1)}-\nab w_0|_{s-1.5},
\end{aligned}
\end{equation}where the last term can be absorbed by the left side if we pick $\delta_0$ sufficiently small. Therefore, by the second boundary condition, we know $(\ww^{(1)},\bz,\pp^{(1)})$ satisfies the compatibility condition up to first order.

Again, we construct $\pp^{(2)}$ such that $(\ww^{(1)},\bz,\pp^{(2)})$ satisfies the compatibility condition \eqref{data order even} up to 2nd order. The new pressure is defined by the poly-harmonic system
\begin{equation}
\begin{cases}
-\Delta^3 \pp^{(2)}=-\Delta^3 \pp^{(1)}~~&\text{ in }\Omega,\\
\pp^{(2)}=\pp^{(1)}=0,~~\p_3\pp^{(2)}=\p_3\pp^{(1)}~~&\text{ on }\Gamma,\\
-\Delta\pp^{(2)}=\MM_0(\ww^{(1)},\bz,\PP^{(1)})+\NN_0(\ww^{(1)},\bz)~~&\text{ on }\Gamma,
\end{cases}
\end{equation}and thus
\begin{equation}
\begin{aligned}
\|\pp^{(2)}-P_0\|_s\lesssim&~\|\Delta^3\pp^{(1)}-\Delta^3 P_0\|_{s-6}+|\p_3\pp^{(1)}-\p_3 P_0|_{s-1.5}\\
&+|\MM_0(\ww^{(1)},\bz,\|\QQ_{(m+7)}^{(m-1)}\|_0+\|\QQ_{(m+6)}^{(m-1)}\|_1^{(1)})+\MM_0(\ww^{(1)},\bz)-\MM_0(w_0,b_0,P_0)-\NN_0(w_0,b_0)|_{s-2.5}\\
\lesssim&~O(\eps)+|b_0|_{L^{\infty}}^2|\p_3^2(\pp^{(2)}-P_0)|_{s-2.5},
\end{aligned}
\end{equation}where the last term is again absorbed by the left side if we choose $|b_0|\leq \delta_0$ to be suitably small. It should also be noted that, $\MM_{k}$ also has other terms containing $(k+2)$-th order derivative, but there are at least two derivatives appearing as $(\bar{b}_0\cdot\cnab)$, and thus we can replace $\pp^{(2)}$ with $\pp^{(1)}$ using the remaining boundary conditions.

Next we construct $\ww^{(2)}$ via the following system such that $(\ww^{(2)},\bz,\pp^{(2)})$ satisfies the compatibility condition \eqref{data order odd} up to 3rd order.
\begin{equation}
\begin{cases}
\Delta^4\ww^{(2),3}=\Delta^4\ww^{(1),3}~~&\text{ in }\Omega,\\
\p_3^j\ww^{(2),3}=\p_3^j\ww^{(1),3},(0\leq j\leq 2)~~&\text{ on }\Gamma\\
\Delta\dive\ww^{(2)}=\MM_1(\ww^{(2)},\bz,\pp^{(2)})+\NN_1(\ww^{(2)},\bz,\pp^{(2)})+\Delta\MM_{-1}(\ww^{(1)},\bz)~~&\text{ on }\Gamma,
\end{cases}
\end{equation}and similarly as above we can get
\begin{equation}
\begin{aligned}
\|\ww^{(2)}-w_0\|_s\lesssim&~\|\Delta^4\ww^{(1)}-\Delta^4 w_0\|_{s-8}+\sum_{j=0}^2|\p_3^j\ww^{(1)}-\p_3^jw_0|_{s-j-0.5}\\
&+|(\MM_1+\NN_1+\Delta\MM_{-1})(\ww^{(1)},\bz,\pp^{(2)})-(\MM_1+\NN_1+\Delta\MM_{-1})(w_0,b_0,P_0)|_{s-3.5}\\
\lesssim&~O(\eps)+|b_0|_{L^{\infty}}^2|\p_3^3(\ww^{(2)}-w_0)|_{s-3.5},
\end{aligned}
\end{equation}where the last term is again absorbed by the left side if we choose $|b_0|\leq \delta_0$ to be suitably small.

So, we can repeat the above procedures such that $\pp^{(m)}$ is determined by the poly-harmonic equation $\Delta^{2m-1}\pp^{(m)}=\Delta^{2m-1}\pp^{(m-1)}$ in $\Omega$ equipped with the boundary conditions $\p_3^j\pp^{(m)}=\p_3^j\pp^{(m-1)}$ on $\Gamma$ for $0\leq j\leq 2m-3$ and the compatibility condition \eqref{data order even} for the case $k=2m-2$ with $(w_0,b_0,P_0)$ replaced by $(\ww^{(m-1)},\bz,\pp^{(m)})$.

Similarly, $\ww^{(m)}$ is determined by $\ww^{(m),1,2}=\ww^{(m-1),1,2}$ and $\Delta^{2m}\ww^{(m),3}=\Delta^{2m}\ww^{(m-1),3}$ in $\Omega$ equipped with the boundary conditions $\p_3^j\ww^{(m),3}=\p_3^j\ww^{(m-1),3}$ on $\Gamma$ for $0\leq j\leq 2m-2$ and the compatibility condition \eqref{data order odd} for the case $k=2m-1$ with $(w_0,b_0,P_0)$ replaced by $(\ww^{(m)},\bz,\pp^{(m)})$.

Since the given rough data $(w_0,b_0,P_0)$ satisfies the compatibility conditions up to 7-th order, we stop the above procedure after we get $(\ww^{(4)},\bz,\pp^{(4)})$ which is a smooth data and also satisfies the compatibility conditions up to 7-th order. We rename this smooth data to be $(\vv,\bz,\qq)$. For any given $\eps>0$, we construct a smooth data $(\vv,\bz,\qq)$ that satisfies the compatibility conditions up to the same order as the given rough data $(w_0,b_0,P_0)$ and has the following approximation
\begin{equation}
\|\vv-w_0\|_8+\|\bz-b_0\|_8+\|\qq-P_0\|_8=O(\eps).
\end{equation}

\subsection{Extend the compatibility conditions up to infinite order}\label{appendix data inf}
\subsubsection{Formal constructions}

We then try to extend the initial data such that the compatibility conditions are fulfilled up to infinite order. First we briefly state some formal construction. Recall in section \ref{appendix data}, for a given data $(w_0,b_0,P_0)$, the corresonding solution satisfies the wave equation
\[
\p_t^2P-\Delta_b P=\MM_0(v,b,P)+\NN_0(v,b),~~\Delta_b:=(1+|b|^2)\Delta-(b\cdot\p)^2,
\]and $\MM_0(v,b,P):=\bp^2b\cdot b+\RR(v,b)$ where $\RR$ only contains the first-order derivatives of $b,v$. So if we start with an irrotational velocity $w_0=\nab\psi$ and define $P_{(-1)}:=-\psi$, then since $P=p+\frac12|b|^2$ we have
\[
\p_t P=\p_tp+b\cdot\p_t b=-\dive w+b\cdot((b\cdot\p) w-b\dive w)=-(1+|b|^2)\dive w+(b\cdot\p) w\cdot b,
\]which then gives, after restricting it to $\{t=0\}$
\[
-\Delta_bP_{(-1)}=-P_{(1)}+\bp \nab\psi\cdot b_0,
\]where the right side only depends on the given data of velocity and magnetic field. Taking more time derivatives and setting $t=0$ yields an infinite elliptic system of the form
\[
-\Delta_{\bz}P_{(k)}=-P_{(k+2)}+\NN_{k}'(P_{(-1)},\cdots, P_{(k-1)}),~~k\geq -1,
\]where $\NN_k'$ is a functional that only depends on the derivatives of its arguments and $b_0$ up to a certain order. This system has similar structure as in \cite[Lemma 16.1]{Lindblad2005LWP2} and thus can be solved in a similar manner. The only difference comes from the appearance of magnetic field, but $\bp$ is a tangential derivatiev and we can pick suitable $b_0$ such that its normal component vanishes in a neighborhood of the boundary.

\subsubsection{Full construction procedure}
For specific calculations, we now define the desired smooth data $(\vv^{\infty},\bz^\infty,\qq^\infty)$ by
\begin{equation}
\vv^{\infty}:=\vv-\nab\QQ_{(-1)}^{\infty},~~\bz^\infty:=\bz.
\end{equation}And after a long and tedious calculation, we find $\qq^\infty$ is determined by the following infinite elliptic system in $\Omega$, where $\lambda:=1+|\bz|^2$.
\begin{align}
\label{QQ-1 eq}-\lambda\Delta\QQ_{-1}^{\infty}=& -(\QQ_{(1)}^{\infty}-\QQ_{(1)})-\bzp^2 \QQ_{-1}^{\infty}+\NN'_{-1}(\bz,\QQ_{-1}^{\infty})\\
\label{QQ0 eq}-\lambda\Delta\QQ_0^{\infty}=&-\QQ_{(2)}^{\infty}-\bzp^2 \QQ_{0}^{\infty}+(\bz\cdot\nab)^2\bz\cdot\bz+\NN'_{0}(\bz,\vv^{\infty},\QQ_{-1}^{\infty})\\
\label{QQ1 eq}-\lambda\Delta\QQ_{(1)}^{\infty}=&-\QQ_{(3)}^{\infty}-\bzp^2 \QQ_{1}^{\infty}+(\bz\cdot\nab)^3\vv^{\infty}\cdot\bz-|\bz|^2(\bz\cdot\nab)^2(\nab\cdot\vv^{\infty})+\NN'_{1}(\bz,\vv^{\infty},\QQ_{-1}^{\infty},\QQ_0^{\infty}),\\
\label{QQ2 eq}-\lambda\Delta\QQ_{(2)}^{\infty}=&-\QQ_{(4)}^{\infty}-\bzp^2\Delta \QQ_{0}^{\infty}+\NN'_{2}(\bz,\vv^{\infty},\QQ_{-1}^{\infty},\QQ_0^{\infty},\QQ_{(1)}^{\infty})
\end{align}and for $k\geq 2$
\begin{align}
\label{QQk eq}-\lambda\Delta\QQ_{(k)}^{\infty}=&-\QQ_{(k+2)}^{\infty}-\bzp^2\Delta \QQ_{(k-2)}^{\infty}+\NN'_{k}(\bz,\vv^{\infty},\QQ_{-1}^{\infty},\QQ_0^{\infty},\QQ_{(1)}^{\infty},\cdots,\QQ_{(k-1)}^{\infty}),
\end{align}with vanishing boundary conditions for each $\QQ_{(k)}$. These $\NN_k'$'s have the following structure
\begin{equation}
\begin{aligned}
&~~~~\NN'_k(\vv^\infty,\bz,\QQ^\infty_{(-1)},\cdots,\QQ^{(\infty)}_{(k-1)})\\
&=P(\bz)C^{k;m_1\cdots m_r}_{i_1\cdots i_m,j_1\cdots j_n, k_1\cdots k_r}\left((\nab^{i_1} \vv^{\infty})\cdots(\nab^{i_m} \vv^{\infty})(\nab^{j_1} \bz)\cdots(\nab^{j_n} \bz)(\nab^{k_1} \QQ_{(m_1)})\cdots(\nab^{k_r} \QQ_{(m_l)})\right),
\end{aligned}
\end{equation}where the indices satisfy
\begin{align*}
i_1+\cdots+i_m+j_1+\cdots+j_n+(k_1+m_1)+\cdots+(k_r+m_r)=k+2,\\
1\leq i_1,\cdots, i_m,j_1,\cdots,j_m,k_1,\cdots,k_r\leq k+1,\\
-1\leq m_1,\cdots, m_r\leq k-1,~1\leq k_1+m_1, \cdots, k_r+m_r\leq k+1.
\end{align*}
\begin{rmk}
Note that \eqref{QQ2 eq} is derived from taking two time derivatives in \eqref{QQ0 eq}. On can also In fact, taking two time derivatives in \eqref{QQ0 eq}, we know the right side has top-order terms $\bzp^2 \QQ_{2}-(\bz\cdot\nab)^2\mathbf{b}_2\cdot\bz$. For the latter term $(\bz\cdot\nab)^2\mathbf{b}_2\cdot\bz$, we recall that $\QQ=\mathbf{q}+\frac12|\bz|^2$. Taking one time derivative and using continuity equation, we get $\QQ_{(1)}=-\dive\vv+\mathbf{b}_1\cdot\bz$. Taking one more time derivative and using the momentum equation $\mathbf{v}_1=\bzp\bz-\nab\QQ_0$, we get $\dive\mathbf{v}_1=\dive(-\nab \QQ_0+\bzp\bz)$. Using $\dive \bz=0$ we know $\dive\bzp\bz$ is of lower order. So we have $\QQ_{(2)}\sim \Delta\QQ_0+\mathbf{b}_2\cdot\bz$, and thus
\[
(\bz\cdot\nab)^2 \QQ_{2}-(\bz\cdot\nab)^2\mathbf{b}_2\cdot\bz=(\bz\cdot\nab)^2\Delta\QQ_0+\text{ lower order terms.}
\]

We choose to write in this form because it makes equations shorter. Alternatively one can differentiate \eqref{QQ1 eq} in time variable again and again to get the form $-\Delta_{\bz}\QQ_{(k)}^\infty=-\QQ_{(k+2)}^\infty+\MM_k'(\bz^\infty,\QQ_{(-1)},\cdots,\QQ_{(k-1)})+\NN'_k(\bz^\infty,\QQ_{(-1)},\cdots,\QQ_{(k-1)})$ where $\Delta_{\bz}:=(1+|\bz|^2)\Delta-\bzp^2$, and $\MM_k'$ denotes the terms containing $(k+2)$-th order derivative.
\end{rmk}

This elliptic system has a parallel structure as \cite[(16.11)]{Lindblad2005LWP2}.  Following \cite[Lemma 16.2]{Lindblad2005LWP2}, we impose the system with boundary conditions
\[
\QQ_{(k)}^\infty|_{\Gamma}=\QQ_{0,k},~~-\nab_N\QQ_{(k)}^\infty|_{\Gamma}=\QQ_{1,k},~~k\geq -1.
\]Then the system has a formal power series in the distance to the boundary
\[
\overline{\QQ}_{(k)}(r,\omega)\sim\sum \QQ_{n,k}(\omega)\frac{(1-r)^n}{n!},
\]where $r$ is the distance to the boundary and $\omega$ is the angular variable. Let $0\leq \chi(\cdot)\leq 1$ be a smooth bump function on $\R$ that equals to 1 in $[-1,1]$ and vanishes outside $[-2,2]$. Then, by \cite[Lemma 16.2]{Lindblad2005LWP2}, there exist $\eps_{k,n}$ such that
\[
\overline{\QQ}_{(k)}(r,\omega)=\sum\chi\left(\frac{1-r}{\eps_{k,n}}\right)\QQ_{n,k}(\omega)\frac{(1-r)^n}{n!},
\] such that the above elliptic system holds to infinite order on the boundary. Note that $\bzp$ is tangential on the boundary and $\bz$ has smallness assumption, so the extra terms involving $\bz$ will not affect the convergence of the power series.

Now let $(\tilde{\vv},\bz,\tilde{\qq})$ are functions that vanish to infinite order on the boundary. Define $$\ww^{\infty}:=\tilde{\vv}-\nab\overline{\QQ}_{(-1)},~\pp^\infty:=\tilde{\qq}+\overline{\QQ}_0,~\mathbf{P}_{(1)}:=-(1+|\bz|^2)(\dive\tilde{\vv}+\Delta\overline{\QQ}_{(-1)})+\bzp\ww^{\infty}\cdot\bz,$$ where $\overline{\QQ}_0,\overline{\QQ}_{(-1)}$ are given by the above construction. Then inductively one can show that $\mathbf{P}_{(k)}^\infty=\tilde{\qq}_{(k)}+\overline{\QQ}_{(k)}$, where $\overline{\QQ}_{(k)}$ are constructed above and $\tilde{\qq}_{(k)}$ vanishes to infinite order on the boundary. Hence, choosing boundary data such that $\QQ_{0,k}=0$ for $k\geq 0$ and $\QQ_{1,k}\geq c_0>0$, then the Rayleigh-Taylor sign condition for $\pp^\infty$ is fulfilled and also we have $\mathbf{P}^{\infty}_{(k)}|_{\Gamma}=0$ for all $k\geq 0$.

\end{appendix}


{\small
\begin{thebibliography}{99}
\addcontentsline{toc}{section}{References}
\setlength{\itemsep}{0.5ex}
\begin{spacing}{0.9}
\bibitem{ABZ2014LWP}
Alazard, T., Burq, N., Zuily, C.
\newblock {\em On the Cauchy problem for gravity water waves.}
\newblock  Invent. Math., 198(1), 71--163. 2014.

\bibitem{Alinhacgood89}
Alinhac, S.
\newblock {\em Existence d'ondes de rar\'efaction pour des syst\`emes quasi-lin\'eaires hyperboliques multidimensionnels.(French. English summary) [Existence of rarefaction waves for multidimensional hyperbolic quasilinear systems].}
\newblock Commun. Partial Differ. Equ., 14(2), 173-230, 1989.

\bibitem{AlinhacNash}
Alinhac, S., G\'erard, P.
\newblock {\em Pseudo-differential Operators and the Nash–Moser Theorem.}
\newblock Translated from the 1991 French original by Stephen S.Wilson. American Mathematical
Society, Providence, 2007.

\bibitem{ChenDingMHDST}
Chen, P., Ding, S.
\newblock {\em Inviscid limit for the free-boundary problems of MHD equations with or without surface tension.}
\newblock arXiv: 1905.13047, preprint, 2019.

\bibitem{ChenWangCMHDVS}
Chen, G.-Q., Wang, Y.-G.
\newblock {\em Existence and Stability of Compressible Current-Vortex Sheets in Three-Dimensional Magnetohydrodynamics.}
\newblock Arch. Ration. Mech. Anal., 187(3), 369-408, 2008.

\bibitem{ChenSX1982}
Chen, S.-X.
\newblock {\em Initial boundary value problems for quasilinear symmetric hyperbolic systems
with characteristic boundary.}
\newblock Translated from Chin. Ann. Math. 3(2), 222–232 (1982). Front. Math. China 2(1), 87–102 (2007).

\bibitem{ChenSX2013}
Chen, S.-X.
\newblock {\em Multidimensional nonlinear systems of conservation laws.} (in Chinese)
\newblock SCIENTIA SINICA Mathematica, 43(4), 317-332, 2013.

\bibitem{CL2000priori}
Christodoulou, D., Lindblad, H.
\newblock {\em On the motion of the free surface of a liquid.}
\newblock  Commun. Pure Appl. Math., 53(12), 1536--1602, 2000.

\bibitem{CMST2012MHDVS}
Coulombel, J.-F., Morando, A., Secchi, P., Trebeschi, P.
\newblock {\em A priori Estimates for 3D Incompressible Current-Vortex Sheets.}
\newblock  Commun. Math. Phys., 311(1), 247-275, 2012.

\bibitem{CHS2013LWP}
Coutand, D., Hole, J., Shkoller, S.
\newblock {\em Well-Posedness of the Free-Boundary Compressible 3-D Euler Equations with Surface Tension and the Zero Surface Tension Limit.}
\newblock  SIAM J. Math. Anal., 45(6), 3690-3767, 2013.

\bibitem{CLS2010priori}
Coutand, D., Lindblad, H., Shkoller, S.
\newblock {\em A priori estimtes for the free-boundary 3D compressible Euler equations in physical vacuum.}
\newblock  Commun. Math. Phys., 296(2), 559-587, 2010.

\bibitem{CS2007LWP}
Coutand, D., Shkoller, S.
\newblock {\em Well-posedness of the free-surface incompressible Euler equations
  with or without surface tension.}
\newblock  J. Amer. Math. Soc., 20(3), 829--930, 2007.

\bibitem{CS2010LWP}
Coutand, D., Shkoller, S.
\newblock {\em A simple proof of well-posedness for the free-surface incompressible Euler equations.}
\newblock  Discrete and Continuous Dynamical Systems (Series S), 3(3), 429-449, 2010.

\bibitem{CS2012LWP}
Coutand, D., Shkoller, S.
\newblock {\em Well-posedness in smooth function spaces for the moving-boundary three-dimensional compressible Euler equations in physical vacuum.}
\newblock  Arch. Ration. Mech. Anal., 206(2), 515-616, 2012.

\bibitem{DL2019limit}
Disconzi, M. M. and Luo, C.
\newblock {\em On the incompressible limit for the compressible free-boundary Euler equations with surface tension in the case of a liquid.}
\newblock Arch. Ration. Mech. Anal., 237(2), 829-897, 2020.

\bibitem{Ebin1987}
Ebin, D.~G.
\newblock {\em The equations of motion of a perfect fluid with free boundary are not well posed.}
\newblock  Commun. Partial Differ. Equ., 12(10), 1175--1201, 1987.


\bibitem{Ginsberg2018rel}
Ginsberg, D.
\newblock {\em A priori estimates for a relativistic liquid with free surface boundary.}
\newblock J. Hyperbolic Differ. Eq., 16(03), 401-442, 2019.

\bibitem{GL2021rel}
Ginsberg, D., Lindblad, H.
\newblock {\em On the local well-posedness for the relativistic Euler equations for an isolated liquid body.}
\newblock  Ann. PDE, to appear, 2022.

\bibitem{GLL2019LWP}
Ginsberg, D., Lindblad, H., Luo, C.
\newblock {\em Local well-posedness for the motion of a compressible, self-gravitating liquid with free surface boundary.}
\newblock Arch. Ration. Mech. Anal., 236(2), 603-733, 2020.

\bibitem{MHDphy}
Goedbloed, H., Keppens, R. and Poedts, S.
\newblock {\em Magnetohydrodynamics of Laboratory and Astrophysical plasmas.}
\newblock  Cambridge University Press, 2020.

\bibitem{Guaxi1}
Gu, X.
\newblock {\em Well-posedness of axially symmetric incompressible ideal magnetohydrodynamic equations with vacuum under the non-collinearity condition.}
\newblock  Commun. Pure \& Appl. Anal., 18(2), 569-602, 2019.

\bibitem{Guaxi2}
Gu, X.
\newblock {\em Well-posedness of axially symmetric incompressible ideal magnetohydrodynamic equations with vacuum under the Rayleigh-Taylor sign condition.}
\newblock  arXiv: 1712.02152, preprint, 2017.

\bibitem{GuLuoZhang2021MHDST}
Gu, X., Luo, C., Zhang, J.
\newblock {\em Local Well-posedness of the Free-Boundary Incompressible Magnetohydrodynamics with Surface Tension.}
\newblock arxiv: 2105.00596, preprint, 2021.

\bibitem{GuLuoZhang2021MHD0ST}
Gu, X., Luo, C., Zhang, J.
\newblock {\em Zero Surface Tension Limit of the Free-Boundary Problem in Incompressible Magnetohydrodynamics.}
\newblock Nonlinearity, 35(12), 6349-6398, 2022.

\bibitem{GuWang2016LWP}
Gu, X., Wang, Y.
\newblock {\em On the construction of solutions to the free-surface incompressible
  ideal magnetohydrodynamic equations.}
\newblock  J. Math. Pures Appl., Vol. 128: 1-41, 2019.

\bibitem{GuoMHDSTviscous}
Guo, B., Zeng, L., Ni, G.
\newblock {\em Decay rates for the Viscous Incompressible MHD with and without Surface Tension.}
\newblock Comput. Math. Appl., 77(12), 3224-3249, 2019.

\bibitem{Hao2015gas}
Hao, C.
\newblock {\em Remarks on the free boundary problem of compressible Euler equations in physical vacuum with general initial densities.}
\newblock  Discrete and Continuous Dynamical Systems (Series B), 20(9), 2885-2931, 2015.

\bibitem{Hao2017MHD}
Hao, C.
\newblock {\em On the motion of free interface in ideal incompressible MHD.}
\newblock Arch. Ration. Mech. Anal., 224(2), 515-553, 2017.

\bibitem{HaoLuo2014priori}
Hao, C., Luo, T.
\newblock {\em A priori estimates for free boundary problem of incompressible inviscid magnetohydrodynamic flows.}
\newblock Arch. Ration. Mech. Anal., 212(3), 805--847, 2014.

\bibitem{HaoLuo2018ill}
Hao, C., Luo, T.
\newblock {\em Ill-posedness of free boundary problem of the incompressible ideal MHD.}
\newblock  Commun. Math. Phys., 376(1), 259-286, 2020.

\bibitem{HaoLuo2019LLWP}
Hao, C., Luo, T.
\newblock {\em Well-posedness for the linearized free boundary problem of incompressible ideal magnetohydrodynamics equations.}
\newblock J. Differ. Equ., Vol. 299, 542-601, 2021.

\bibitem{IT2020gas}
Ifrim, M., Tataru, D.
\newblock {\em The compressible Euler equations in a physical vacuum: a comprehensive Eulerian approach.}
\newblock Preprint, arxiv: 2007.05668, 2020.

\bibitem{Jang2014gas}
Jang, J., Masmoudi, N.
\newblock {\em Well-posedness of Compressible Euler Equations in a Physical Vacuum.}
\newblock  Commun. Pure Appl. Math., 68(1): 61-111, 2014.

\bibitem{lopatinskii}
Kreiss, H.-O.
\newblock {\em Initial boundary value problems for hyperbolic systems.}
\newblock Commun. Pure  Appl. Math., 23(3), 277-298, 1970.

\bibitem{LeeMHD1}
Lee, D.
\newblock {\em Initial value problem for the free boundary magnetohydrodynamics with zero magnetic boundary condition.}
\newblock Commun. Math. Sci., 16(3): 589-615, 2018.

\bibitem{LeeMHD2}
Lee, D.
\newblock {\em Uniform estimate of viscous free-boundary magnetohydrodynamics with zero vacuum magnetic field.}
\newblock  SIAM Journal of Mathematical Analysis, 49(4), 2710-2789, 2017.

\bibitem{Lindblad2003LLWP1}
Lindblad, H.
\newblock {\em Well-posedness for the linearized motion of an incompressible liquid with free surface boundary.}
\newblock  Commun. Pure Appl. Math., 56(2), 153-197, 2003.

\bibitem{Lindblad2005LWP1}
Lindblad, H.
\newblock {\em Well-posedness for the motion of an incompressible liquid with free surface boundary.}
\newblock  Ann. Math., 162(1), 109-194, 2005.

\bibitem{Lindblad2003LLWP2}
Lindblad, H.
\newblock {\em Well-posedness for the linearized motion of a compressible liquid with free surface boundary.}
\newblock Commun. Math. Phys., 236(2), 281-310, 2003.

\bibitem{Lindblad2005LWP2}
Lindblad, H.
\newblock {\em Well-posedness for the motion of a compressible liquid with free surface boundary.}
\newblock Commun. Math. Phys.  260(2), 319-392, 2005.

\bibitem{LL2018priori}
Lindblad, H., Luo, C.
\newblock {\em A priori estimates for the compressible Euler equations for a liquid with free surface boundary and the incompressible limit.}
\newblock Commun. Pure Appl. Math., 71(7), 1273-1333, 2018.

\bibitem{Luo2018CWW}
Luo, C.
\newblock {\em On the Motion of a Compressible Gravity Water Wave with Vorticity.}
\newblock Annals of PDE, 4(2), 2506-2576, 2018.

\bibitem{LuoZhang2019MHD2.5}
Luo, C., Zhang, J.
\newblock {\em A regularity result for the incompressible magnetohydrodynamics equations.
  with free surface boundary.}
\newblock Nonlinearity, 33(4), 1499-1527, 2020.

\bibitem{LuoZhang2019MHDST}
Luo, C., Zhang, J.
\newblock {\em A priori estimates for the incompressible free-boundary magnetohydrodynamics equations with surface tension.}
\newblock SIAM J. Math. Anal., 53(2), 2595-2630, 2021.

\bibitem{LuoZhang2020CWW}
Luo, C., Zhang, J.
\newblock {\em Local well-posedness for the motion of a compressible gravity water wave with vorticity.}
\newblock J. Differ. Equ., Vol. 332:333-403, 2022.

\bibitem{LuoXinZeng2014gas}
Luo, T., Xin, Z., Zeng, H.
\newblock {\em Well-posedness for the motion of physical vacuum of the three-dimensional compressible Euler equations with or without self-gravitation.}
\newblock  Arch. Ration. Mech. Anal., 213(3), 763-831, 2014.

\bibitem{MRgood2017}
Masmoudi, N., Rousset, F.
\newblock {\em Uniform Regularity and Vanishing Viscosity Limit for the Free Surface Navier-Stokes Equations.}
\newblock  Arch. Ration. Mech. Anal., 223(1), 301-417, 2017.

\bibitem{MTT2014MHDLLWP}
Morando, A., Trakhinin, Y., Trebeschi, P.
\newblock {\em Well-posedness of the linearized plasma-vacuum interface problem in ideal incompressible MHD.}
\newblock  Quarterly of Applied Mathematics, 72(3), 549-587, 2014.


\bibitem{OS1998MHDill}
Ohno, M., Shirota, T.
\newblock {\em On the initial-boundary-value problem for the linearized equations of magnetohydrodynamics.}
\newblock  Arch. Ration. Mech. Anal., 144(3), 259-299, 1998.

\bibitem{anisotropictrace}
Ohno, M., Shizuta, Y., Yanagisawa, T.
\newblock {\em The trace theorem on anisotropic Sobolev spaces.}
\newblock  Tohoku Math. J., 46(3), 393-401, 1994.

\bibitem{Solonnikov}
Padula, M., Solonnikov, V. A.
\newblock {\em On the free boundary problem of magnetohydrodynamics.}
\newblock  Zap. Nauchn. Semin. POMI, 385, 135–186 (2010). J. Math. Sci. (N.Y.) 178 (2011), no. 3, 313–344, 2010.

\bibitem{Secchi1995}
Secchi, P.
\newblock {\em Well-posedness for Mixed Problems for the Equations of Ideal Magneto-hydrodynamics.}
\newblock  Archiv. der. Math. 64(3), 237–245, 1995.

\bibitem{Secchi1996}
Secchi, P.
\newblock {\em Well-posedness of characteristic symmetric hyperbolic systems.}
\newblock  Arch. Ration. Mech. Anal., 134(2), 155-197, 1996.

\bibitem{SecchiNotes}
Secchi, P.
\newblock {\em On the Nash-Moser Iteration Technique.}
\newblock In: Amann, H., Giga, Y., Kozono, H., Okamoto, H., Yamazaki, M. (eds) Recent Developments of Mathematical Fluid Mechanics. Advances in Mathematical Fluid Mechanics, pp. 443-457. Birkhäuser, Basel 2016.

\bibitem{Secchi2013CMHDLWP}
Secchi, P., Trakhinin, Y.
\newblock {\em Well-posedness of the plasma--vacuum interface problem.}
\newblock  Nonlinearity, 27(1): 105-169, 2013.

\bibitem{SZ1}
Shatah, J., Zeng, C.
\newblock {\em Geometry and a priori estimates for free boundary problems of the
  Euler's equation.}
\newblock  Commun. Pure Appl. Math., 61(5), 698-744, 2008.

\bibitem{SZ2}
Shatah, J., Zeng, C.
\newblock {\em A priori estimates for fluid interface problems.}
\newblock  Commun. Pure Appl. Math., 61(6), 848-876, 2008.

\bibitem{SZ3}
Shatah, J., Zeng, C..
\newblock {\em Local well-posedness for fluid interface problems.}
\newblock  Arch. Ration. Mech. Anal., 199(2), 653-705, 2011.

\bibitem{SWZ2015MHDLWP}
Sun, Y., Wang, W., Zhang, Z.
\newblock {\em Nonlinear Stability of the Current-Vortex Sheet to the Incompressible MHD Equations.}
\newblock  Commun. Pure Appl. Math., 71(2), 356-403, 2018.

\bibitem{SWZ2017MHDLWP}
Sun, Y., Wang, W., Zhang, Z.
\newblock {\em Well-posedness of the Plasma-Vacuum Interface Problem for Ideal Incompressible MHD.}
\newblock  Arch. Ration. Mech. Anal., 234(1), 81-113, 2019.


\bibitem{Trakhinin2008CMHDVS}
Trakhinin, Y.,
\newblock {\em The existence of current-vortex sheets in ideal compressible magnetohydrodynamics.}
\newblock Arch. Ration. Mech. Anal., 191(2), 245-310, 2009.

\bibitem{Trakhinin2009gas}
Trakhinin, Y.
\newblock {\em Local existence for the free boundary problem for nonrelativistic and Relativistic compressible Euler equations with a vacuum boundary condition.}
\newblock  Commun. Pure Appl. Math., 62(11), 1551-1594, 2009.

\bibitem{Trakhinin2016MHD}
Trakhinin, Y.
\newblock {\em On well-posedness of the plasma-vacuum interface problem: the case of non-elliptic interface symbol.}
\newblock Commun. Pure Appl. Anal., 15 (4), 1371-1399, 2016.

\bibitem{TW2020MHDLWP}
Trakhinin, Y., Wang, T.
\newblock {\em Well-posedness of Free Boundary Problem in Non-relativistic and Relativistic Ideal Compressible Magnetohydrodynamics.}
\newblock Arch. Ration. Mech. Anal., 239(2), 1131-1176, 2021.

\bibitem{TW2021MHDSTLWP}
Trakhinin, Y., Wang, T.
\newblock {\em Well-Posedness for the Free-Boundary Ideal Compressible Magnetohydrodynamic Equations with Surface Tension.}
\newblock Math. Ann., to appear, 2021.

\bibitem{WangXingood}
Wang, Y., Xin, Z.
\newblock {\em Vanishing viscosity and surface tension limits of incompressible viscous surface waves.}
\newblock SIAM J. Math. Anal., 53(1), 574-648, 2021.

\bibitem{WangXinMHDST}
Wang, Y., Xin, Z.
\newblock {\em Global Well-posedness of Free Interface Problems for the incompressible Inviscid Resistive MHD.}
\newblock Commun. Math. Phys., 388(3), 1323-1401, 2021.

\bibitem{Wu1997LWP}
Wu, S.
\newblock {\em Well-posedness in Sobolev spaces of the full water wave problem in
  2-D.}
\newblock Invent. Math., 130(1), 39--72, 1997.

\bibitem{Wu1999LWP}
Wu, S.
\newblock {\em Well-posedness in Sobolev spaces of the full water wave problem in
  3-D.}
\newblock J. Amer. Math. Soc., 12(2), 445--495, 1999.

%

\bibitem{1991MHDfirst}
Yanagisawa, T., Matsumura, A.
\newblock {\em The fixed boundary value problems for the equations of ideal magneto-hydrodynamics with a perfectly conducting wall condition.}
\newblock Commun. Math. Phys., 136(1), 119-140, 1991.

\bibitem{Zhang2019CRMHD}
Zhang. J.
\newblock {\em A priori Estimates for the Free-Boundary Problem of Compressible Resistive MHD Equations and Incompressible Limit.}
\newblock arxiv: 1911.04928, preprint, 2019.

\bibitem{Zhang2020CRMHD}
Zhang. J.
\newblock {\em Local Well-posedness of the Free-Boundary Problem in Compressible Resistive Magnetohydrodynamics.}
\newblock arxiv: 2012.13931, preprint, 2020.

\bibitem{Zhang2021elasto}
Zhang. J.
\newblock {\em Local Well-posedness and Incompressible Limit of the Free-Boundary Problem in Compressible Elastodynamics.}
\newblock Arch. Rational Mech. Anal., 244(3), 599-697, 2022.

\bibitem{ZZ2008LWP}
Zhang, P., Zhang, Z.
\newblock {\em On the free boundary problem of three-dimensional incompressible Euler equations.}
\newblock Commun. Pure Appl. Math., 61(7), 877-940, 2008.
\end{spacing}
\end{thebibliography}
}
\end{document}